\documentclass[a4paper,11pt,reqno]{article}
\usepackage{a4wide}
\usepackage[english]{babel}
\usepackage{amssymb}
\usepackage{amsmath}
\usepackage{kpfonts}
\usepackage{amsthm}
\usepackage{hyperref}
\hypersetup{colorlinks,citecolor=red,filecolor=purple,linkcolor=blue,urlcolor=black}
\usepackage[T1]{fontenc}
\usepackage{enumerate}
\usepackage{graphicx}
\usepackage{relsize}
\usepackage{multicol}

\newcommand\T[3]{T^{#2, #1}_{#3}}

\newcommand{\R}{\mathbb R}
\newcommand{\Esp}{\mathbb E}
\newcommand{\E}{\mathbb E}
\newcommand{\p}{\mathbb P}

\newcommand{\1}[1]{\mathbf{1}\!_{\{#1\}}}

\def\be{\begin{eqnarray}}
\def\ee{\end{eqnarray}}
\def\ben{\begin{eqnarray*}}
\def\een{\end{eqnarray*}}
\def\Bea{\begin{eqnarray*}}
\def\Eea{\end{eqnarray*}}
\def\d{\text{d}}

\newtheorem{prop}{Proposition}[section]
\newtheorem{defi}[prop]{Definition}
\newtheorem{assumption}[prop]{Assumption}
\newtheorem{lem}[prop]{Lemma}
\newtheorem{thm}[prop]{Theorem}

\newtheorem{cor}[prop]{Corollary}

\newcounter{example}
\title{Approximation of  stochastic processes by \\ non-expansive  flows   and  coming down from infinity}

\author{Vincent Bansaye\thanks{CMAP, Ecole Polytechnique, CNRS, route de
    Saclay, 91128 Palaiseau Cedex-France; E-mail: \href{mailto:vincent.bansaye@polytechnique.edu}{\texttt{vincent.bansaye@polytechnique.edu}}}
}

\begin{document}
\maketitle

\begin{abstract} This paper deals with the approximation of semimartingales in  finite dimension  by dynamical systems.  We give trajectorial  estimates  uniform 
with respect to the initial condition   for a well chosen distance.  This relies on a non-expansivity property of the  flow  and allows to consider 
   non-Lipschitz vector fields. The fluctuations of the process are controlled 
using    the martingale technics  initiated in \cite{BBL} and      stochastic calculus. \\
  Our main motivation is the trajectorial description of the behavior  of   stochastic processes starting from large initial values.  We  state  general properties 
  on the coming down from infinity of  one-dimensional SDEs, with a focus on stochastically monotone processes.  In particular, we recover and complement  known results on $\Lambda$-coalescent  and birth and death processes. Moreover, using Poincar\'e's compactification technics
for  flows  close to infinity, we  develop this approach   in two dimensions  for  competitive stochastic models. We thus classify the coming down from infinity 
of  Lotka-Volterra diffusions and  provide uniform estimates for the scaling limits of competitive birth and death processes.
\end{abstract}


%\noindent\emph{MSC 2010:} 

\tableofcontents

\section{Introduction}

The approximation of stochastic processes  has been largely developed and we refer  e.g. to  \cite{EK, JS} for  general statements both for deterministic approximation and study of the fluctuations.
A particular attention has been paid 
to  random perturbation of 
dynamical systems  \cite{SH, FW}  and the study of fluid and scaling limits of random models, see   \cite{DN} for a survey about approximation  of Markov chains. 
In this paper, we are interested in  stochastic processes $(X_t : t\geq 0)$ taking values in a Borel subset $E$ of $\R^{\d}$, which can be written as
$$X_t=X_0+\int_0^t \psi(X_s)ds + R_t,$$
where $ R$ is  a semimartingale. We aim at proving that $X$ remains close to the flow $\phi(x_0,t)=x_t$ given by
$$x_t=x_0+\int_0^t \psi(x_s)ds.$$
The point  here is to estimate the probability of this event uniformly with respect to 
the initial condition $x_0\in D$, when the  drift term $\psi$ may be non-Lipschitz on $D$.
Our main motivation for such estimates is the  description of the coming down from infinity, which amounts to let the initial condition $x_0$ go to infinity,
and the uniform scaling limits of stochastic processes describing population models on unbounded domains. \\
%We provide some illustrations and applications of the results for coalescent and competitive processes in one and two dimensions. \\

%\marginpar{Contraction et closer discutable}
The approach relies on a contraction property of the  flow, which provides  stability on the dynamics. This notion is used in particular in control theory. More
 precisely, we say   that
the vector field $\psi$ is   non-expansive  on a domain $D$ when it prevents two trajectories from moving away  for the euclidean norm on a subset
$D$ of $\R^{\d}$. This amounts to
%\marginpar{A voir}
$$\forall x,y \in D, \qquad (\psi(x)-\psi(y)){\boldsymbol .}(x-y)\leq 0,$$
where $.$ is the usual scalar product on $\R^{\d}$.
Actually, the distance between two solutions may increase provided that this increase is not too fast. This allows to deal with additional Lipschitz component 
or bounded perturbation in the flow
and it is  required for the applications considered here. Thus  we are working  with \emph{$(L, \alpha)$ non-expansive vector fields} :  

\begin{defi} \label{defnexp} The vector field $\psi : D \rightarrow \R^{\emph{d}}$ is $(L, \alpha)$ non-expansive on   
   $D\subset \R^{\emph{d}}$ if
for any $x,y\in D$,
$$(\psi(x)-\psi(y)){\boldsymbol .}(x-y)\leq L\parallel x-y\parallel_2^2+\alpha\parallel x-y\parallel_2.$$
\end{defi}
%:
%$$\forall x,y \in D, \qquad (\psi(x)-\psi(y)){\boldsymbol .}(x-y)\leq L\parallel x-y\parallel_2^2+\alpha\parallel x-y\parallel_2.$$
%This will meet the assumptions required for our applications. 
%Roughly speaking, we are needing   the  non-Lipschitz component of the vector field has to be non-expansive. \\
\noindent The non-expansivity property ensures that the drift term can not   make the distance between the stochastic process $X$ and  the dynamical system $x$ explode
 because of small fluctuations due to the perturbation $R$.
To control the size of these fluctuations, we  use 
martingale technics  in Section \ref{main} : let us mention  \cite{DN} in the  context of scaling limits and    \cite{BBL}
for a pioneering work on the speed of coming down from infinity of  $\Lambda$-coalescents. In this latter,  the short time behavior of the $\log$ of the number of blocks is captured and the non-expansivity argument for the flow
 is replaced by a technical result relying on  the  monotonicity of suitable functions in dimension $1$ (Lemma 10 therein).  \\

These results are developed and  specified
when $X$ satisfies a Stochastic Differential Equation (SDE),   in Section \ref{EDS}, which allows 
a diffusion component and random jumps given by a Poisson point measure. This  covers  the range of our applications. 
We then estimate the probability that the stochastic process remains close to the dynamical system as soon as this latter is in    a domain 
$D$ where a transformation $F$ ensuring 
$(L,\alpha)$-non-expansivity can be found.
These
estimates hold for any  $x_0 \in D$    and  a well chosen distance $d$, which is bound  to capture the fluctuations of $X$ around the flow $\phi$. Informally, we obtain that
for any $\varepsilon>0$,
%We prove in Section \ref{main} under some assumptions that for $\varepsilon,T>0$,
\be
\label{bornespap}
\mathbb P_{x_0}\left( \sup_{t\leq T\wedge T_D(x_0)} d(X_t,\phi(x_0,t))  \geq \varepsilon \right) \leq C_T \int_0^T \overline{V}_{d,\varepsilon}(x_0,t)  dt ,
\ee
where $T_D(x_0)$  corresponds to the exit time of the domain $D$  for the flow $\phi$ started at $x_0$. The distance 
$d$ is of the form
$$d(x,y)=\parallel F(x)-F(y)\parallel_2,$$
where $F$ is of class $\mathcal C^2$, 
so that   we can use the stochastic calculus. The perturbation needs to be  controlled for this distance $d$ in  a tube  around the trajectory of the dynamical system and
%\marginpar{To check}
$$\overline{V}_{d,\varepsilon}(x_0,t)=\sup_{\substack{ x\in E \\ d(x, \phi(x_0,t))\leq \varepsilon }} \left\{\varepsilon^{-2} \parallel V_F(x) \parallel_1 + \varepsilon^{-1} \parallel \widetilde{b}_F(x)  \parallel_1  \right\},$$
where $V_F$ will be given by the quadratic variation of $F(X)$ and 
$ \widetilde{b}_F$ will be an additional approximation term arising from It\^o's formula applied to $F(X)$.  \\
Relevant choices of $F$ will be illustrated through several examples. They    are both  linked to the geometry of the flow since the
   $(L,\alpha)$ non-expansivity property has to be satisfied   and to the control
  of the size of the fluctuations induced by $R$.   We refer in particular to  the role of the fluctuations for the examples of Section \ref{Ex} and 
the  adjunction   involved in the last section for gluing  transformations providing non-expansitivity. \\

The estimate (\ref{bornespap}) becomes uniform with  respect to $x_0 \in D$ as soon as 
$\overline{V}_{d,\varepsilon}(x_0,.)$ can be bounded by an integrable  function of the time.
%$$\sup_{x_0 \in D}\mathbb P_{x_0}\left( \sup_{t\leq T\wedge T_D(x_0)} \parallel F(X_t)-F(\phi_F(x_0,t))\parallel_2 >\varepsilon\right) \leq \frac{C\exp(4LT)}{\varepsilon^{2}} \int_0^T \left[1+ \overline{V}_{F,\varepsilon}(x_0,s) \right] ds ,$$
%avec 
It allows then to characterize the coming down from infinity for  stochastic differentials equations in $\R^{\d}$.
Roughly speaking, we consider an unbounded domain $D$  and let $T$
go to $0$  to derive from (\ref{bornespap}) that for any  $\varepsilon>0$,
$$\lim_{T\rightarrow 0}\sup_{ x_0  \in D}\mathbb P_{x_0}\left( \sup_{t\leq T} d(X_t,\phi(x_0,t))  \geq \varepsilon\right) =0.$$
Letting then $x_0$ go to infinity  enables to describe
 the coming down from infinity of processes in several ways.
First, the control of the fluctuations of the process $X$  for large initial values  by a dynamical system gives  a way
to prove the tightness of $\p_{x_0}$ for $x_0 \in D$. 
%We obtain then the weak convergence of $\p_{x_0}$ as $x_0\rightarrow \infty$ for monotone SDEs. 
Moreover we can link in general the coming down from infinity of the process $X$ to the coming down from infinity of the flow $\phi$, in the vein of \cite{BBL, limxsi, BMR}, which focus  respectively on $\Lambda$ coalescence, $\Xi$ coalescent and
 birth and death processes.

In dimension $1$,  following \cite{D, BMR}, we use a monotonicity property to identify the limiting values of  $\p_{x_0}$ as $x_0\rightarrow \infty$ and
we determine when the  process comes down from infinity and how it comes down from infinity (Section \ref{dimensionun}). In particular, we recover the speed of coming
down from infinity of $\Lambda$-coalescent \cite{BBL} with $F=\log$ and in that case $V_F$ is bounded. We also recover some results of \cite{BMR} for birth and death processes and we can provide finer estimates for regularly varying death rates. Here $F$ is polynomial and $V_F$ is unbounded so this latter has to be controlled  along the trajectory of the dynamical system. Finally, we consider  the example of transmission control protocol which is  non-stochastically monotone and $F(x)=\log (1+\log (1+x))$ is required to control  its large fluctuations for large values.

In higher dimension, the coming down from infinity of a dynamical system is a more delicate problem in general.  Poincar\'e has initiated a  theory to study dynamical systems close to infinity, which is particularly powerful for polynomial  vector fields (see e.g. Chapter 5 in \cite{D}). We develop this approach for   
competitive Lotka-Volterra models in dimension $2$ in Section \ref{descente2}, which was a main motivation for this work. We  classify  the  ways the dynamical system can come down  from infinity and describe the  counterpart for the stochastic process, which differs when the dynamical system is getting close from the boundary of $(0,\infty)^2$. \\ 

The uniform estimates (\ref{bornespap})  can also  be used to
prove scaling limits of  stochastic processes $X^K$ to dynamical systems, which are uniform with respect to the initial condition, without requiring  Lipschitz  property for the vector field $\psi$. It involves     a suitable distance  $d$ as introduced above to capture the fluctuations of the process :
$$\lim_{K\rightarrow \infty}\sup_{x_0  \in D}\mathbb P_{x_0}\left( \sup_{t\leq T} d(X_t^K,\phi(x_0,t))  \geq \varepsilon\right) =0,$$
for some fixed $T, \varepsilon>0$.
 It is illustrated  in this paper by the convergence of birth and death processes with competition to  Lotka-Volterra competitive dynamical system in Section \ref{descente3}. \\

Let us end up with  other motivations for this work, some of which being linked to works in progress. 
First, our original motivation for studying the coming down from infinity is the description of the time for extinction for competitive models in varying environment. Roughly speaking, competitive periods make  large sizes of   populations quickly decrease, which  can be captured by the coming down from infinity.
Direction and speed of coming down from infinity are then involved to quantify the time extinction or determine coexistence of populations.
 Let us also note that the approach developed here could be  extended to the varying environment framework by comparing the stochastic process 
to a non-autonome dynamical system. 
Second,  the coming down from infinity is linked to the uniqueness of the quasistationary distribution, see  \cite{Doorn1991} for birth and death processes and  \cite{Cattiaux2009} for some diffusions. Recently, 
the coming down from infinity has appeared as a key assumption for the geometric convergence of the conditioned process to the quasistationary distribution, uniformly with respect to the initial distribution. We refer to \cite{CV} for details, see in particular Assumption $(A1)$ therein. \\

%Finally, several challenges are left open to deal with the coming down from infinity of other models in dimension 2 and  other scaling limits.  \\ \\
%\marginpar{ Cattiaux Collet M\'el\'eard pour premiers travaux sur Lotka Volterra}
%Finally, let  us  mention \cite{H} for some numerical related problems and  the works of Goetz Kersting  (see e.g. \cite{K2}) for  motivating results for the nomalization by a dynamical 
%systems.  \\ \\
%Et environnement changeant perspectives, la technique permet de quantifier le retour vers les compacts dans les environnements d\'efavorables et s\'etend aux environnements changeants, WIP.
{\bf Notation.}
In the whole paper  ${\boldsymbol .}$  stands for the canonical scalar product on $\R^{\d}$, 
  $\parallel .\parallel_2$ the associated euclidean norm and  $\parallel .\parallel_1$  the $L^1$ norm. \\ For convenience,  we write  $x=(x^{(i)} :  i=1, \ldots ,\d)\in \R^{\d}$  a row vector of real numbers. % elements  the elements of $\R^d$ and   a row vector of $\R^{\d}$.
  The product $xy$ for $x,y\in \R^{\d}$
  is the vector $z\in \R^{\d}$ such that 
  $z_i=x_iy_i$.  \\
 % \marginpar{Besoin ?}
  We denote by $\overline{B}(x,\varepsilon)=\{ y\in \R^{\d} :  \ \parallel y-x\parallel_2 \leq \varepsilon\}$ the euclidean closed ball centered in $x$ with radius $\varepsilon$. More generally, we note $\overline{B}_d(x,\varepsilon)=\{ y\in O :d(x,y) \leq \varepsilon\}$ the closed ball centered in $x\in O$ with radius $\varepsilon$  associated with  the application   $d :  O\times O \rightarrow \R^+$. \\
  When $\chi=(\chi^{(1)}, \ldots, \chi^{(d)})$ is differentiable on an open set of $\R^{\d}$ and takes values in $\R^{\d}$, we denote by 
  $J_{\chi}$ its Jacobian matrix and $$(J_{\chi}(x))_{i,j}= \frac{\partial}{\partial x_j}  \chi^{(i)} (x) \qquad (i,j=1, \ldots, \text{d}).$$ 
  We write $F^{-1}$ the reciprocal function of a bijection $F$ and $A^{-1}$ the inverse 
 of an invertible matrix $A$.
 Moreover, the  transpose of a matrix $A$ is denoted by $A^*$. \\
  By convention, we assume that $\sup \varnothing =0$, $\sup [0,\infty)=+\infty$, $\inf \varnothing=\infty$ and  if $x,y \in \R \cup \{\infty\}$, we write  $x\wedge y$ for  the smaller element of $\{x,y\}$. \\
  %\marginpar{Expliquer la politique de constante $C_{.}$ constantes diff qui d\'ependent de ...}
  We write $d(x)\sim_{x\rightarrow a} g(x)$  when $d(x)/g(x) \rightarrow 1$ as $x\rightarrow a$.\\
  We also use  notation $\int_a^. f(x)dx<\infty$  (resp. $=\infty$) for $a\in [0,\infty]$ when there exists $a_0\in (a,\infty)$ such that
  $\int_a^{a_0} f(x)dx $ is well defined and finite (resp. infinite).\\
   Finally,    we denote by $<M>$ the predictable quadratic variation of a continuous local martingale $M$ and by
  % \marginpar{$Var(X)$ chez JS pour var totale}
 $\vert A\vert$  the  total variation of a process $A$ and by $\Delta X_s=X_s-X_{s-}$ the jump at time $s$ of a c\`adl\`ag process $X$. \\
 
 \noindent
 {\bf Outline of the paper.} In the next section, we provide general results for dynamical systems perturbed by semimartingales using the non-expansivity of the flow and  martingale inequality. In Section \ref{EDS}, we  derive approximations results for 
  Markov process described  by SDE.  It relies on a transformation $F$ of the process for which we apply the results of Section \ref{main}. An extension of the result by adjunction of non-expansive domains is provided and required for the applications of the last section. We then study the coming down from infinity for one dimensional SDEs in Section \ref{dimensionun}, with a   focus on stochastically monotone processes. Finally we compare
   the coming down from infinity of two dimensional competitive   Lotka-Volterra diffusions with the coming down  from infinity
   of Lotka-Volterra dynamical systems and  prove  uniform approximations of these latter by  birth and death processes.

\section{Random perturbation of dynamical systems}
\label{main}
Let $(\Omega, \mathcal F, \p)$ be a probability space and $(\mathcal F_t)_{t\geq 0}$ a 
filtration of $\mathcal F$, which satisfies the usual conditions.
We assume that  $X$ is  $\mathcal F_t-$ adapted  c\`adl\`ag process  on $[0,\infty)$
%\marginpar{ where $\mathcal F$ satisfies the usual conditions=cad croissant et coplŽŽte chez JS, qui parle debase, $X$ defini pour tout temps}
which takes its values in a  Borel subset $E$ of   $\R^{\d}$  and satisfies   for every  $t\geq 0$,
$$X_t=X_0+\int_0^t \psi(X_s)ds + R_t,$$ 
where $X_0 \in E$ a.s.,  
  $\psi$ is a Borel measurable function from $\R^\d$ to $\R^\d$ locally bounded  and $(R_t : t\geq 0)$ is a c\`adl\`ag  $\mathcal F_t$-semimartingale. 
Moreover, the process $R$ is  decomposed as 
$$R_t=A_t+M_t, \qquad M_t=M_t^c+M_t^d,$$
with  $A_t$   a c\`adl\`ag  $\mathcal F_t$-adapted  process with a.s. bounded variations paths, $M_t^c$  a continuous 
$\mathcal F_t$-local martingale,   $M_t^d$ a  c\`adl\`ag  $\mathcal F_t$ -local martingale purely discontinuous and
$R_0=A_0=M_0=M_0^c=M_0^d=0$. Let us observe that such a decomposition may be non-unique. \\

%\emph{Les petits sauts compens\'es seront \`a ajouter ... on doit pouvoir \'etendre \'egalement au cas d'un syst\`eme dynamique  non autonome ($\psi=\psi_t$) pour des applications en environnement variable}.  \\
%Our aim is to compare the process $X$ to the  solution $x=\phi(x_0,.)$  of the dynamical system 
%associated  with the vector field  $\psi$ and some initial condition $x_0$ : 
We assume that $\psi$ is locally Lipchitz on a (non-empty) open set $E'$ of $\R^d$ and consider the solution   $x=\phi(x_0,.)$ of   
$$x_t=x_0+\int_0^t \psi(x_s)ds$$
$x_0\in E'$.
%For that purpose, %\marginpar{maximal ?}
This solution  exists, belongs to $E'$ and is unique on a time interval $[0, T'(x_0))$, where $T'(x_0)\in (0,\infty]$. Then, 
to compare
the process $X$ to the  solution $x$,  we define the maximal gap before $t$:
$$S_t:= \sup_{s\leq t} \parallel X_{s}-x_s \parallel_2$$  
for any  $t< T'(x_0)$.
We also set 
\be
\label{tempssortie}
T_{D, \varepsilon}(x_0) = \sup\{ t \in [0, T'(x_0)) \ : \ \forall s\leq t, \   x_s\in D \ \text{and} \   \overline{B}(x_s,\varepsilon) \cap E \subset D\} \    \in [0,\infty]
\ee
  the last time when $x_t$ and its $\varepsilon$-neighborhood in $E$ belong to a domain  $D$. %\marginpar{Op\'erateur $L$ accr\'etif.% ultracontractivit\'e ? 
%Aller voir aussi  Slotine, Bjorn ou Rueffer ?\\
%Doublon maintenant avec intro.\\
%Signaler que sur un voisinage \c ca suff ?}
 As mentioned in the introduction, the key property to control the distance between $(X_t : t\geq 0)$ and $(x_t : t\geq 0)$ before time $T_{D, \varepsilon}(x_0)$ is the $(L,\alpha)$ non-expansivity  property of  $\psi$ on $D$, in the sense of Definition \ref{defnexp}. 
 % Thus, we assume 
%that $\psi$ is $(L,\alpha)$ non-expansive on $D$
%It means that the distance  between two solutions is non-increasing.    This notion is classical in dynamical system and control theory, it yields a stability property to the dynamical system. It amounts to require that the vector field $\psi$ is non-increasing in dimension  one. We  consider here the euclidian distance to use the stochastic calculus and change this distance via a well chosen transformation in the next sections. Thus, we need the following definition, which ensures the non-expansivity property of the flow for the euclidian distance, up to some term we can control.
\noindent When $\alpha=0$, we  simply say that  $\psi$ is $L$ non-expansive on $D$. If additionally $L=0$, we say that
$\psi$ is non-expansive on $D$. We first note that in dimension $1$, the fact that $\psi$  is non-expansive simply means that $\psi$ is non-increasing. More generally,  
when
 $\psi$ is differentiable on a convex open set $O$ which contains $D$,  
  $\psi$ is  $L$ non-expansive on $D$
 if  for any $x\in O$,
 $$\text{Sp} (J_{\psi}+ J_{\psi}^* )\subset (-\infty, 2L],$$
 %\marginpar{Question de convexe suffit}
 where  Sp$(J_{\psi}+ J_{\psi}^*)$ is the spectrum of the symmetric linear operator (and hence diagonalisable)
 $J_{\psi}+ J_{\psi}^*$, see table 1 in \cite{Aminzare} for details and more general results and the last section for an application. Finally, we  observe
that
$$\psi= B+ \chi =B+f+g $$
is $( L,\alpha)$ non-expansive on $D$  if 
 $B$ is a vector field whose euclidean norm is bounded by   $\alpha$  on $D$ and $\chi$ if $L$ non-expansive on $D$. Moreover $\chi=f+g$ is $L$ non-expansive on $D$ if 
   $f$ is Lipschitz with constant  $L$   and  $g$ is non-expansive on $D$.  \\
% \marginpar{Preciser passage \`a sous domaine ?} 
 \\

For convenience and use of Gronwall Lemma, we also introduce for $L,\alpha\geq 0$ and $\varepsilon>0$,%  our control for $(L,\alpha)$ non-expansive vector field hold up to time
\be
\label{temps}
\T{\alpha}{L}{\varepsilon}= \sup\{ T \geq 0 : 4\alpha T \exp(2LT) \leq \varepsilon\} \in (0,\infty],
\ee
which is infinite if and  only if $\alpha=0$, i.e.  as soon as the vector field $\psi$ is  $L$ non-expansive. 

%d\'ecroissant et appara\^itra de fa\c con tr\`es naturelle pour la descente de l'infini (en particulier de mod\`eles stochastiquement  monotones). 
%De fa\c con g\'en\'erale (en dimension sup\'erieure, au voisinage d'un point irr\'egulier pour le champ), cette condition assure que le flot associ\'e \`a  $\psi$ diminue les distances 
%(non-expansive dynamical systems). La d\'efinition donn\'ee ici permet d'ajouter un terme (habituel) Lipschitzien dans le champ de vecteur et une constante utiles pour les applications. \\

\subsection{Trajectorial control for perturbed non-expansive dynamical systems}
\label{trajj}

The following lemma gives the trajectorial result which  allows 
to control the gap between the stochastic process $(X_t  : t \geq 0)$ and the dynamical system 
$(x_t : t \geq 0)$ by the size of the fluctuations of the semimartingale $(R_t : t\geq 0)$
and the gap between the initial positions.
The control of fluctuations involves the following quantity, which is  defined   $t<T'(x_0)$ and $\varepsilon>0$: 
%$$S_t:= \sup_{s\leq t} \parallel X_{s}-x_s \parallel_2$$ and
$$\widetilde{R}_t^{\varepsilon}=\parallel X_0-x_0 \parallel_2^2+ \1{ S_{t-} \leq \varepsilon}\left[2\int_0^t   (X_{s-}-x_s)\bold{.}dR_s +\parallel [M]_t \parallel_1 \right],$$
% $$$$T_\eta:= \inf \left\{  t\geq 0  \ :  \ \parallel X_0-x_0 \parallel_2^2+ \1{ S_{t-} \leq \varepsilon}\left[2\int_0^t   (X_{s-}-x_s)\bold{.}dR_s +\parallel [M]_t \parallel_1 \right] \geq \eta^2\ \right\},$$
where $\int_0^t   (X_{s-}-x_s)\bold{.}dR_s $  is  a
stochastic integral and $[M]=[X]=[R]$ is the  quadratic variation of the semimartingale $R$. We refer to  Chapter I, Theorem 4.31 in   \cite {JS}  for the existence of stochastic integral of   c\`agl\`ad     (and thus predictable  locally bounded)
process  with respect to  semimartingale.
 Moreover,  the expression of the quadratic variation  ensures that
 \be
 \label{corchetdroit}
 \parallel [M]_t \parallel_1=\parallel [X]_t\parallel_1= \parallel <M^c>_t\parallel_1 + \sum_{s\leq t} \parallel \Delta X_{s}\parallel_2^2,
 \ee
see e.g.   Chapter 1, Theorem 4.52 in \cite{JS}. Unless otherwise specified, the identities hold \emph{almost surely   (a.s.)}.
% Finally, we define
%$$T_{D, \varepsilon}(x_0) = \inf\{ t \in [0, T'(x_0)) :  \exists y \not\in D, y\in  \overline{B}(x_t, \varepsilon)\}\wedge T'(x_0)$$
 % the first time when the solution $(x_t : t\in [0, T'(x_0))$ is at distance $\varepsilon$ of the boundary of
 % $D$.
%et$$J^*_{\eta}=\inf\left\{ t\geq 0 : \1{ S_{t-}\leq \varepsilon}\left[ \parallel \Delta A_t +  \Delta M_t^d \parallel_2^2 +2\varepsilon  \parallel \Delta A_t +\Delta M_t^d \parallel_2 \right]\geq \eta^2\right\}.$$
%et en pratique on pourra souvent prendre  $M^*=\inf\left\{ t\geq 0 : \1{ S_{t-}\leq \varepsilon}\parallel \Delta A_t \parallel_2^2 \geq \varepsilon^2 /3\right\}.$
%\marginpar{Notation $X.Y$ pour int\'egarle dans JS; $X.Y$ a un integrant pr\'evisible localement born\'e car $X-$ d'un cadlag et $Y$ une semi martingale}
\begin{lem} \label{key} 
Assume that $\psi$ is $(L, \alpha)$ non-expansive on some domain $D\subset E'$ and let $\varepsilon >0$. \\
Then for any $x_0\in E'$ and   $T< T_{D, \varepsilon}(x_0)\wedge \T{\alpha}{L}{\varepsilon}$,  we have 
$$\left\{ S_T \geq \varepsilon\right\} \subset \big\{\sup_{t\leq T} \widetilde{R}_t^{\varepsilon} > \eta^2\big\},$$
where $\eta=\varepsilon\exp(-LT)/\sqrt{2}$.
\end{lem}

\begin{proof} Let  $x_0\in E'$.
First, we consider the  quadratic variation of  $(X_t-x_t : 0\leq t < T'(x_0))$:
$$[X-x]_t=[M]_t=(X_t-x_t)^2-(X_0-x_0)^2-2\int_0^t (X_{s-}-x_s) d(X_s-x_s),$$
 for $t<T'(x_0)$, 
see  e.g. Chapter 1, Definition 4.4.45 in \cite{JS} or use It\^o formula.
Summing the coordinates of $[M]_t$ and using the definitions of $X$ and $x$, we get
%Jacod Shirayev $[X-x]=[X]=[M]=(X-x)^2-(X_0-x_0)^2-2\int (X-x)_{-} .d(X-x)$), ou par la formule d' It\^o [Theorem 4.57, Jacod Shirayev, p57],\marginpar{intergration ctr martingale locale. Besoin de qqchose sur integrand ? $L^{2,loc}$ ?}
\begin{eqnarray*}
\parallel X_t-x_t\parallel_2^2 % &=& \parallel [X]_t- (X_t-x_t){\boldsymbol .}(X_t-x_t) \\
%&=& \parallel X_0-x_0 \parallel_2^2+2\int_0^t (X_s-x_s){\boldsymbol .}(\psi(X_s)-\psi(x_s))ds
% +2\int_0^t (X_{s-}-x_s){\boldsymbol .}dR_s+ \parallel <M^c>_t\parallel_1 \\
% && \qquad + \sum_{s\leq t} \left[\parallel X_{s}-x_{s}\parallel_2^2-\parallel X_{s-}-x_{s}\parallel_2^2-2(X_{s-}-x_s){\boldsymbol .}(X_s-X_{s-}) \right] \\
 &=& \parallel X_0-x_0 \parallel_2^2+2\int_0^t (X_{s-}-x_s){\boldsymbol .}(\psi(X_{s-})-\psi(x_s))ds
 +2\int_0^t (X_{s-}-x_s){\boldsymbol .}dR_s+ \parallel [M]_t\parallel_1.
 \end{eqnarray*}
Moreover   for any $s< T_{D, \varepsilon}(x_0)$, $x_s\in D$ and $X_{s-}\in D$ on the event $\{S_{s-} \leq \varepsilon\}$. So  using that  $\psi$ is  $(L, \alpha)$ non-expansive on $D$,
$$\1{S_{s-}\leq \varepsilon}(X_{s-}-x_s){\boldsymbol .}(\psi(X_{s-})-\psi(x_s)) \leq  \1{S_{s-}\leq \varepsilon}\left(L\parallel X_{s-}-x_s\parallel_2^2+\alpha\parallel X_{s-}-x_s\parallel_2\right).$$
Then for any $t< T_{D, \varepsilon}(x_0)$, 
%\begin{eqnarray*}
%\1{S_{t-}\leq \varepsilon}  S_t^2 &\leq & 2L\int_0^t \parallel X_{s}-x_s\parallel_2^2ds+2\alpha \int_0^t \parallel X_s-x_s\parallel_2ds\\
%&& \qquad +2\int_0^t (X_{s-}-x_s){\boldsymbol .}dR_s+\parallel <M_t>\parallel_1
%\end{eqnarray*}
\begin{eqnarray*}
\1{S_{t-}\leq \varepsilon} \parallel X_t-x_t\parallel_2^2 
&\leq &\1{S_{t-}\leq \varepsilon}\bigg[ 2L\int_0^t \parallel X_{s}-x_s\parallel_2^2ds
+2\alpha \int_0^t \parallel X_{s}-x_s\parallel_2ds \\
&& \qquad \qquad  \quad + \parallel X_0-x_0 \parallel_2^2+
2\int_0^t (X_{s-}-x_s){\boldsymbol .}dR_s+\parallel [M]_t\parallel_1 \bigg]
\end{eqnarray*}
and by definition of $\widetilde{R}^{\varepsilon}$,
%We assume now that $t<T_{\eta}$ for some $\eta>0$. Then there   exists $\widetilde{\eta}\in (0,\eta)$ such that
\begin{eqnarray*}
\1{S_{t-}\leq \varepsilon}  S_t^2  &\leq & 2L\int_0^t\1{S_{s-}\leq \varepsilon}  S_s^2 ds+2\alpha t \varepsilon+ \sup_{s\leq t} \widetilde{R}^{\varepsilon}_s.
\end{eqnarray*}
By  Gronwall lemma, %\marginpar{Gronwall discontinu}, 
 we obtain  for any $T<  T_{D, \varepsilon}(x_0)$ and $t\leq T$, %pour tout $t\leq T_D(x_0)\wedge T$ et $t< T_{\eta} $,
$$\1{S_{t-}\leq \varepsilon}  S_t^2 \leq \left(2\alpha T\varepsilon+ \sup_{s\leq T} \widetilde{R}^{\varepsilon}_s\right)e^{2LT}.$$
Moreover, for $T<\T{\alpha}{L}{\varepsilon}$, we have $2 \alpha Te^{2LT}<   \frac{\varepsilon}{2}$
and 
$$\left(2\alpha T\varepsilon+ \eta^2\right)e^{2LT} <  \varepsilon^2,$$
recalling that
$\eta=\varepsilon/(\sqrt{2}\exp(LT))$. 
Then %for   $T< T_{D, \varepsilon}(x_0)$, we get
% Donc pour ces temps,  "tant que tu es plus petit qu'$\varepsilon$ et avant $T_{\varepsilon/8}$, tu es plus petit qu'$\varepsilon/2$".
\be
\label{eqquu}
\left\{ \sup_{s\leq T} \widetilde{R}^{\varepsilon}_s \leq \eta^2\right\} \subset \left\{ \sup_{t\leq T} \1{S_{t-}\leq \varepsilon}  S_t^2 < \varepsilon^2\right\}.
\ee
Denoting
$$T_{exit}=\inf \{ s  <  T_{D,\varepsilon}(x_0)\wedge \T{\alpha}{L}{\varepsilon}  : S_s \geq   \varepsilon\},$$ 
and recalling that $S$ is c\`adl\`ag,  we have  $S_{T_{exit}-}\leq\varepsilon$ and  $S_{T_{exit}}\geq \varepsilon$ on the event $\{ T_{exit} \leq T\}$, so using $(\ref{eqquu})$ at time $t=T_{exit}$ ensures that
$$\{T_{exit} \leq T\} \subset  \left\{ \sup_{s\leq T} \widetilde{R}^{\varepsilon}_s > \eta^2\right\},$$
which ends up the proof.
\end{proof}

\subsection{Non-expansivity   and perturbation by martingales}
We use now martingale maximal inequality to
estimate the probability that the distance between the process $(X_t : t \geq 0)$ and the dynamical system $(x_t : t \geq 0)$ goes beyond some level
$\varepsilon>0$. Such arguments are classical and  have been used in  several contexts,  see in particular  \cite{DN} for a survey and applications in scaling limits
and     \cite{BBL} for  the coming down from infinity of $\Lambda$-coalescent, which  have both inspired the results below. \\

% \marginpar{En pratique  $M_t^d$  square locally integrable martingale) et automatique pour sa partie continue $M_t^c$}

%\marginpar{Voir martingales exponentielles}% (cf Barlow Jacka Yor Prop. 4.2.1 1986) ? $C_p$ ne d\'epend pas d'$\varepsilon$ ...}

\begin{prop} \label{ctrltpscourt} 
Assume that $\psi$ is $(L, \alpha)$ non-expansive on some domain $D\subset E'$ and let $\varepsilon >0$. \\
Then for any $x_0\in E'$ and   $T< T_{D, \varepsilon}(x_0)\wedge \T{\alpha}{L}{\varepsilon}$,   for   any
$p\geq 1/2$ and $q\geq 0$, 
 \begin{eqnarray*}
 &&\mathbb P\left( S_T \geq \varepsilon\right) \\
&& \qquad \leq   \mathbb P\left( \parallel X_0-x_0\parallel_2 \geq \varepsilon \frac{e^{-LT}}{2\sqrt{2}}\right) +  C_q\frac{e^{2qLT} }{\varepsilon^{q}}
\E \left(\left( \int_0^T  \1{S_{s-} \leq \varepsilon} d \parallel \vert A\vert_s \parallel_1 \right) ^{q}\right) \\
&& \qquad \qquad  +  C_{p,\emph{d}}\frac{e^{4pLT} }{\varepsilon^{2p}}\left[ \E\left(\left( \int_0^T \1{ S_{t-} \leq \varepsilon} d \parallel <M^c>_t\parallel_1\right)^{p}\right)  
+\E\left( \left( \sum_{t\leq T}\1{ S_{t-} \leq \varepsilon} \parallel \Delta X_{t}\parallel_2^2  \right)^{p} \right)\right],
\end{eqnarray*}
for some positive constants  $C_q$ (resp. $C_{p,\emph{d}}$)  which depend only on $q$ (resp. $p,\emph{d}$).
\end{prop}

\begin{proof}   By definition of $\widetilde{R}^{\varepsilon}$,
$$\left\{ \sup_{t\leq T} \widetilde{R}^{\varepsilon}_t \geq \eta^2\right\} \subset \  \left\{ \parallel X_0-x_0\parallel_2^2 \geq \frac{\eta^2}{4} \right\} \cup B_{\eta},$$
where $B_{\eta}=\{ \sup_{t\leq T}  \1{ S_{t-} \leq \varepsilon} \int_0^t   (X_{s-}-x_s)\bold{.}dR_s  \geq \eta^2/8\}\cup\{ \sup_{t\leq T}   \1{ S_{t-} \leq \varepsilon}\parallel [M]_t \parallel_1 \geq \eta^2/4\}$.
Recalling that $R_t=A_t+M_t$ and $(\ref{corchetdroit})$,
\Bea B_{\eta}&\subset&
\left\{ \sup_{t\leq T} \int_0^t    \1{ S_{s-} \leq \varepsilon}(X_{s-}-x_s)\bold{.}dA_s  \geq\frac{\eta^2}{16}\right\} \cup \left\{ \sup_{t\leq T} 
\int_0^t  \1{ S_{s-} \leq \varepsilon} (X_{s-}-x_s){\boldsymbol .}dM_s \geq\frac{\eta^2}{16}\right\}\\
&& \qquad \quad \cup \left\{\int_0^T \1{ S_{t-} \leq \varepsilon}d\parallel <M^c>_t \parallel_1 \geq \frac{\eta^2}{8}\right\} \cup \left\{
 \sum_{t\leq T} \1{ S_{t-} \leq \varepsilon}\parallel \Delta X_{t}\parallel_2^2   \geq \frac{\eta^2}{8} \right\}.
\Eea
We also know from Lemma \ref{key} that $$\{ S_T \geq \varepsilon \} \subset \left\{ \sup_{s\leq T} \widetilde{R}_t \geq \eta^2\right\}$$
 and using
 Markov inequality yields
\begin{eqnarray}
&&\p (S_T \geq \varepsilon) \nonumber \\
&&\quad \leq  \p\left(\parallel X_0-x_0\parallel_2^2 \geq \frac{\eta^2}{4}\right)+  \mathbb P(B_{\eta})  \nonumber \\
&& \quad \leq  \p\left(\parallel X_0-x_0\parallel_2^2 \geq \frac{\eta^2}{4}\right)+ \left(\frac{16  }{\eta^2}\right)^{q} \E \left(\sup_{t\leq T} \left\vert\int_0^t    \1{ S_{s-} \leq \varepsilon}(X_{s-}-x_s)\bold{.}dA_s  \right\vert^{q}\right)  \nonumber\\
&& \qquad \qquad + \left(\frac{16}{\eta^2}\right)^{2p} \E \left(\sup_{t\leq T} \left\vert\int_0^t   \1{S_{s-}  \leq \varepsilon} (X_{s-}-x_s).d M_s \right\vert^{2p}\right)  \nonumber \\
&& \qquad \qquad
+\left(\frac{8}{\eta^2}\right)^{p} \E\left(\left( \int_0^T  \1{ S_{t-} \leq \varepsilon} d \parallel <M^c>_t\parallel_1\right)^{p}\right) 
 +\left(\frac{8}{\eta^2}\right)^{p} \E\left( \left[ \sum_{t\leq T}\1{ S_{t-} \leq \varepsilon} \parallel \Delta X_{t}\parallel_2^2  \right]^{p} \right). \label{tout}
\end{eqnarray}
First using that $\vert f_s.dg_s\vert \leq \parallel f_s \parallel_2 d\parallel \vert g \vert_s \parallel_1$ since $\vert f_s^{(i)} \vert \leq \parallel f_s\parallel_2$, % \leq \int \vert f\vert d\vert A\vert$, 
we have for $t\leq T$,
\be
\label{inegun}
 \left\vert \int_0^t    \1{ S_{s-} \leq \varepsilon}(X_{s-}-x_s)\bold{.}dA_s \right\vert  \leq \int_0^t   \1{ S_{s-} \leq \varepsilon} \parallel X_{s-}-x_s \parallel_2d A^1_s\leq \varepsilon \int_0^T   \1{ S_{s-} \leq \varepsilon} d A^1_s,
  \ee
where   $A^1_s:= \parallel \vert A\vert_s \parallel_1$ is the sum of the coordinates of the total  variations of the process $A$. \\
Second,  Burkholder Davis Gundy  inequality (see \cite{DM},  93,  chap. VII, p. 287)   for the local martingale
$$N_t=\int_0^t   \1{S_{s-}  \leq \varepsilon} (X_{s-}-x_s).d M_s$$
 ensures that there exists $C_p>0$   such that
 %\marginpar{ en fait $p\geq 1/2$ dans Delacherie Meyer !}
 %  \cite{IW} chap 3  Th 3.1 pour le cas continu, quiest  $\mathcal M^{c,loc}_2$   ; reference dans le cas discontinu : http://arxiv.org/pdf/1308.2648.pdf Marinelly R\"ockner ?}
$$\E \left(\sup_{t\leq T} \left\vert\ N_t  \right\vert^{2p}\right)\\
\leq C_p \E\left( [ N ]_T^p\right).$$
Writing  the coordinates of $X, M$ and $x$ respectively $(X^{(i)}  : i =1, \ldots, \d)$, $(M^{(i)}  : i =1, \ldots, \d)$  and  $(x^{(i)}  : i =1, \ldots, \d)$ and adding that  
  $$[N]_T= \int_0^T \sum_{i,j=1}^{\d}   \1{S_{s-}  \leq \varepsilon} (X_{s-}^{(i)}-x_s^{(i)}) (X_{s-}^{(j)}-x_s^{(j)})d [M^{(i)},M^{(j)}]_s\leq \varepsilon^2
 \int_0^T \sum_{i,j=1}^{\d}   \1{S_{s-}  \leq \varepsilon} d [M^{(i)},M^{(j)}]_s  $$
and that $d [M^{(i)},M^{(j)}]_s\leq d[M^{(i)}]_s+d[M^{(j)}]_s$, we obtain
  %\marginpar{Preciser}
  \begin{eqnarray}
  &&\E \left(\sup_{t\leq T} \left\vert\int_0^t   \1{S_{s-}  \leq \varepsilon} (X_{s-}-x_s).d M_s \right\vert^{2p}\right) \nonumber \\
 && \qquad \leq C_{p,\d} \varepsilon^{2p} \E \left(\left(\int_0^T  \sum_{i=1}^{\d} \1{S_{t-}  \leq \varepsilon} d [M^{(i)}]_t \right)^p\right) \nonumber \\
%  && \qquad \leq C_{p,d}'\varepsilon^2\sup_{t\leq T}  \E \left(\parallel \int_0^t  \1{S_{s-}  \leq \varepsilon} d   [M]_s \parallel_1^p\right)
 % && \qquad \leq C_{p,d}'' \varepsilon^{2p}\left[ \E \left( \left\vert \int_0^t \sum_{i=1}^d \1{S_{s-}  \leq \varepsilon} d   <M^{i,c}>_s \right\vert^p\right)+  \E \left(\left\vert \sum_{s\leq T}  \sum_{i=1}^d  \1{S_{s-}  \leq \varepsilon} (\Delta X_s^i)^2\right\vert^p\right) \right] \\
  &&\qquad \leq C_{p,\d}' \varepsilon^{2p}  \left[ \E\left(\left( \int_0^T \1{ S_{t-} \leq \varepsilon} d \parallel <M^c>_t\parallel_1\right)^{p}\right)  
+\E\left( \left( \sum_{t\leq T}\1{ S_{t-} \leq \varepsilon} \parallel \Delta X_{t}\parallel_2^2  \right)^{p} \right)\right], \label{inegdeux}
  \end{eqnarray}
 for some positive  constants $C_{p,\d}$ and $C_{p,\d}'$, where we recall that $[M^{(i)}]_t=<M^{c,(i)}>_t+\sum_{s\leq t} \left( \Delta X^{(i)}_s\right)^2$.
%Using now  that $$\{ S_T \geq \varepsilon \} \subset \left\{ \sup_{s\leq T} \widetilde{R}_t \geq \eta^2\right\}$$
Plugging (\ref{inegun}) and (\ref{inegdeux}) in (\ref{tout}), 
we get
\begin{eqnarray*}
\p (S_T \geq \varepsilon) 
& \leq &   \p\left(\parallel X_0-x_0\parallel_2^2 \geq \frac{\eta^2}{4}\right)+ \left(\frac{16 \varepsilon }{\eta^2}\right)^{q} \E \left(\left(\int_0^T   \1{S_{s-}  \leq \varepsilon} d A^1_s \right)^{q}\right) \\
&& \qquad \qquad
+\frac{C''_{p,\d}}{\eta^{2p}} \left[\E\left(\left( \int_0^T  \1{ S_{t-} \leq \varepsilon} d \parallel <M^c>_t\parallel_1\right)^{p}\right) 
 +\E\left( \left( \sum_{t\leq T}\1{ S_{t-} \leq \varepsilon} \parallel \Delta X_{t}\parallel_2^2  \right)^{p} \right)\right]
\end{eqnarray*}
for some $C''_{p,\d}$ positive.  Recalling that $\eta=\varepsilon/ (\sqrt{2}\exp(LT))$ ends up the proof.
\end{proof}

\section{Uniform estimates for Stochastic Differential Equations}
\label{EDS}
In this section, we assume that $X=(X^{(i)} : i=1,\ldots,\d)$ is a c\`adl\`ag Markov process which takes values in $E\subset \R^d$ and  is the unique strong solution
of the following SDE on  $[0, \infty)$ : 
%\marginpar{  sens IW signifie  $\mathcal F_t=\sigma(X_0, B_s : s\leq t, N)$ compl par n\'eglig $F_t$ adpt}
%  par  le syst\`eme dynamique dans le cas o\`u $X=(X^{i} : i=1...d)$  est    solution forte sur $[0,T]$ partant d'une condition initiale $X_0$ \`a support $E\subset \R^d$ de l'EDS suivante
\Bea
X_t&=&x_0+\int_0^t b(X_s)ds+\int_0^t \sigma(X_s)dB_s+\int_0^t\int_{\mathcal X} H(X_{s-},z) N(ds,dz)+\int_0^t\int_{\mathcal X} G(X_{s-},z) \widetilde{N}(ds,dz),
\Eea
a.s. for any $x_0 \in E$, 
where $(\mathcal X, \mathcal B_{\mathcal X})$ is a measurable space,
%\marginpar{def component wise et dire qu'on pourrait ajouter $A_t$}
%we recall that the product of two vectors is the vector obtain  by multiplying the coordinates and
\begin{itemize}
\item 
$B=(B^{(i)} : i=1,\ldots,\d)$ is a $\text{d}$-dimensional Brownian motion; % mouvements Brownien ind\'ependants; 
%\marginpar{ $(\mathcal X,\mathcal B_{\mathcal X})$}
\item $N$ is a Poisson Point Measure (PPM)  on $\R^+\times \mathcal X $ with intensity  $dsq(dz)$, where $q$ is a $\sigma$- finite measure on $(\mathcal X, \mathcal B_{\mathcal X})$; and $\widetilde{N}$ is the compensated measure of $N$.
%\marginpar{%associ\'e \`a un PP de class (QL) (Definition 3.1, chapitre II \cite{IW}, mesure $\sigma$ finie et 
%compens\'ee continue adapt\'e bien d\'ef,   automatq pour PPP station,   dt on note $\widetilde{N}$ la mesure compens\'ee}
\item $N$ and $B$ are independent;
\item $b=(b^{(i)} : i=1,\ldots,\d)$,  $\sigma=(\sigma^{(i)}_j : i,j=1,\ldots,\d)$,   $H$ and $G$ are Borel measurable functions   locally bounded, which take values respectively
 $\R^{\d}$, $\R^{2\d}$,  $\R^{\d}$ and  $\R^{\d}$. % fonctions bor\'eliennes localement born\'ees. \\
\end{itemize}
Moreover, we follow the classical convention (see  chapter II in \cite{IW})  and we assume that
%\marginpar{On a le droit darreter $G$}
$HG=0$,  $G$ is bounded  and for any $t\geq 0$,  
$$\int_0^t\int_{\mathcal X} \vert H(X_{s-},z) \vert N(ds,dz)<\infty \quad \text{a.s.}, \quad \quad \E\left(\int_0^t\int_{\mathcal X} \parallel G(X_{s-\wedge \sigma_n},z)\parallel_2^2 ds q(dz)\right)<\infty,   $$
for  some sequence of stopping time $\sigma_n\uparrow \infty$.  We dot not discuss here the conditions which ensure the strong existence and uniqueness of this SDE for any initial condition. This will be standard results 
for the examples considered in this paper and we refer to \cite{DL} for some general statement relevant in our context.
%\marginpar{We  observe that  $( \int_0^t b(X_s)ds : t\geq 0)$  is  $\mathcal F_t$-adapted and a.s. continuous with paths with bounded variations; 
%$( \int_0^t \sigma(X_s)dB_s : t\geq 0)$ belongs to $\mathcal M_2^{c,loc}$, which is the set of  locally square integrable continuous $\mathcal F_t$-martingale;
%$(\int_0^t\int_{\mathcal X} H(X_{s-},z) N(ds,dz) : t\geq 0)$  is  $\mathcal F_t$-adapted and a.s. c\`adl\`ag with paths with bounded variations and $(\int_0^t\int_{\mathcal X} G(X_{s-},z) \widetilde{N}(ds,dz) : t\geq 0)$ is c\`adl\`ag and belongs to $\mathcal M_2^{loc}$, which is the set of  locally square integrable $\mathcal F_t$-martingale.}
%On a ici deux termes \`a variations finies et deux termes de  martingales locales. 
%\marginpar{On suit IW, ca permet de bien d\'efinir chaque terme et d'avoir la decomposition en semimartingale
%$$\E\left( \int_0^T \int_{\mathcal X} \parallel H(X_{s-},z) \parallel_1 dsq(dz)\right)<\infty, \qquad \E\left( \int_0^T \int_{\mathcal X} \parallel G(X_{s-},z) \parallel_2^2 dsq(dz)\right)<\infty $$ 
%et $$\E\left(\int_0^T \parallel\sigma(X_s)\parallel_2^2 ds \right)<\infty $$}

% \ \int_0^t\int_E \vert H(.,z) \vert dsq(dz) <\infty.$$ 

\subsection{Main result}
%\marginpar{Il suffirait de $D$ mais on s'en fout}
 %\marginpar{a t  on besoin de la valeur absolue ?}
We need a transformation $F$  to construct a suitable distance and evaluate the gap between  the process 
$X$ and the associated dynamical system on a domain $D$. 
\begin{assumption} \label{assume}
(i) The domain  $D$ is an open subset of $\R^{\emph{d}}$ and the function $F$ is defined on an open set $O$ which contains $\overline{D\cup E}$. % \rightarrow \R^d$.
\\
(ii) $ F \in \mathcal C^2(  O, \R^{\d})$ and $F$
is a bijection from $D$ into $F(D)$ and its  Jacobian   $J_F$ 
is invertible
 on  $D$.  \\
 (iii) For any $x\in E$,
%  \marginpar{Ca suffit vraiment pour tout compenser ??}
$$\int_{\mathcal X} \left\vert F(x+H(x,z))-F(x)\right\vert  q(dz) <\infty.$$
and the function
$x\in E \rightarrow  h_F(x)= \int_{\mathcal X} [ F(x+H(x,z))-F(x)] q(dz)$ can be extended  to the domain $\overline{D\cup E}$.
 This extension
$h_F$ is  locally bounded on $\overline{D\cup E}$  and locally Lipschitz on $D$. \\
(iv) The function
$b$  is locally Lipschitz on $D$.   
\end{assumption}
\noindent Under this assumption,  $F$ is a $C^2$ diffeomorphism from   $D$ into $F(D)$ and $F(D)$ is on open subset of $\R^d$.  We require in $(iii)$ that the large jumps of $F(X)$ can be compensated. This assumption could be relaxed by letting the large jumps which could not be compensated in an additional term with finite variations, i.e. 
using the term $A_t$ of  the semimartingale $R_t$ in the previous section. But that won't be useful for the applications given here. Under Assumption \ref{assume},  we   set
$b_F=b+J_F^{-1}h_F$, which is well defined and locally Lipschitz on $D$. We note that for any $x\in E\cap D$,
%\marginpar{ $F$ diagonale \\ Attention au domaine de def $T_{E'}$}
%\marginpar{pas inversible sur E}
$$b_F(x)=b(x)+J_F(x)^{-1}\left( \int_{\mathcal X} [F(x+H(x,z))-F(x)] q(dz)\right).$$
We introduce
the flow $\phi_F$ associated to $b_F$ and defined  for $x_0 \in D$ as the unique solution of 
$$\phi_F(x_0,0)=x_0, \qquad \frac{\partial }{\partial t}\phi_F(x_0,t)=b_F(\phi_F(x_0,t)),$$
for  $t \in [0,T_D(x_0))$, where $T_D(x_0)\in (0,\infty]$ is the maximal time until which  the solution exists and belongs to $D$.  We observe that when $H=0$, then $b_F=b$ and $\phi_F=\phi$ do not depend on the transformation $F$.\\
We introduce now the vector field $\psi_F$   defined  by 
$$\psi_F=(J_Fb_F)\circ F^{-1}=(J_Fb+h_F)\circ F^{-1}$$
on the open set $F(D)$.
%\marginpar{intervalle max a repreciser ici ? }
%where  $\mathcal L$ is the infinitesimal generator  of $X$.
We also  set	for any $x\in E$,
\be
\label{defbt}
\widetilde{b}_F(x)=\frac{1}{2} \sum_{i,j=1}^{\d}  \frac{\partial^2 F}{\partial x_i\partial x_j}  (x)\sum_{k=1}^{d}\sigma^{(i)}_k(x)\sigma^{(j)}_k(x) +\int_{\mathcal X} [F(x+G(x,z))-F(x)-J_F(x)G(x,z)] q(dz). 
\ee
Let us note that  the generator of $X$ is given by  $\mathcal L F=\psi_F \circ F+ \widetilde{b}_F$. The term 
$\widetilde {b}_F$ is not contributing significantly to the coming down from infinity in the examples we  consider here and thus considered as an approximation term.
On the contrary,  we need to introduce
\be
\label{defVF}
V_F(x)= \sum_{i,j,k=1}^{\d}  \frac{\partial F}{\partial x_i}(x)\frac{\partial F}{\partial x_j}(x)\sigma^{(i)}_k(x)\sigma^{(j)}_k(x)+\int_{\mathcal X} [F(x+H(x,z)+G(x,z))-F(x)]^2 q(dz).
\ee
for $x\in E$,
to quantify the fluctuations of the process due to the martingale parts. Finally we  use  the following application defined on $O$ (and thus on $D\cup E$) 
 to compare the process  $X$ and the flow  $\phi_F$ : 
$$d_F(x,y)=\parallel F(x)-F(y)\parallel_2.$$
We  observe that $d$ is (indeed)  a  distance (at least) on $D$ and in the examples below  it is  actually a distance on $D\cup E$.
We recall notation $(\ref{temps})$ and the counterpart of  $(\ref{tempssortie})$ is defined by
%\marginpar{check}
%\be
%\label{deftemps}
%T_{D,\varepsilon,F}(x_0)=\inf\{ t\in  [0,T(x_0)) : \exists x \in O- D, \ d_F(x,\phi_F(x_0,t))\leq  \varepsilon\}\wedge T(x_0)
%\ee
%\marginpar{a ton besoin de d\'efinir une boule aussi quand $d$ n'est pas une distance ??? attention a son domaine}
\be
\label{deftemps}
T_{D,\varepsilon,F}(x_0)=\sup\{ t\in  [0,T_D(x_0)) : \forall s\leq t,  \ \overline{B}_{d_F} (\phi_F(x_0,s),   \varepsilon) \cap E \subset D \}.
\ee
%which yields the time when the dynamical system is at distance $\varepsilon$ from the boundary of $D$ for the distance $d_F$.
%where we recall that $\overline{B}_d(x,\varepsilon)=\{ y \in \R^d  : d(x,y)\leq \varepsilon\}$.
%For safe of simplicity and regarding the applications in mind, we set the main resultfor $\varepsilon \leq 1$, $X_0=x_0$  and $p=1$.
\begin{thm}  \label{corutile}
Under Assumption \ref{assume}, we assume that $\psi_F$ is $(L, \alpha)$ non-expansive on $F(D)$. \\
Then for any $\varepsilon >0$ and $x_0\in E \cap D$ and   $T< T_{D, \varepsilon,F}(x_0)\wedge \T{\alpha}{L}{\varepsilon}$,  we have  
$$\mathbb P_{x_0}\left( \sup_{t\leq T} d_F(X_t, \phi_F(x_0,t)) \geq \varepsilon\right) \leq C_{\emph{d}}e^{4LT}\int_0^T \overline{V}_{F,\varepsilon}(x_0,s) ds ,$$
%\marginpar{Donner une valeur pour $C$ ?}
where  $C_{\emph{d}}$ is a positive constant depending only on the dimension $\emph{d}$ and
\be
\label{defOVER}
   \overline{V}_{F,\varepsilon}(x_0,s)=\sup_{\substack{ x\in E \\ d_F(x,\phi_F(x_0,s)) \leq \varepsilon }} \left\{\varepsilon^{-2} \parallel V_F(x) \parallel_1+\varepsilon^{-1}\parallel \widetilde{b}_F(x) \parallel_1\right\}.
   \ee
%et $C$ une constante ne d\'ependant que de la borne de $\widetilde{b}_F$.
%\marginpar{le carr\'e d'un vecteur d\'esigne ici le vecteur des carr\'es des coordonn\'ees } 
\end{thm}
%We do not specify the real number  $C$  in the previous result but let us stres that it depends  neither of  $X$, nor of $L,\alpha, \varepsilon,T$.\\
\noindent We refer to the two next sections for examples and applications, which involve different choices for $F$ and $(L,\alpha)$ non-expansivity with  potentially $\alpha$  or $L$ equal to $0$.
The key assumption  concerns the non-expansivity of $\psi_F$ for a suitable choice of $F$, which need to be combined with control of the fluctuations $V_F$. Before the proof of  Theorem \ref{corutile}, let us illustrate
the condition of  $L$ non-expansivity of $\psi_F$ by considering  the diffusion case ($q=0$ and $X$ continuous). This will be useful in Section
\ref{LV}.  \\
%In that case, we observe that $b_F=b$ and $\phi_F=phi$ does not depend on the transformation $F$. \\
%Let us refer to the two next Sections for examples. Here we simply mention 
%that for $\Lambda$ -coalescent,  following \cite{BBL}, we are taking $F=\log$ and then
%$\psi_F(x)$ is $(0,\alpha)$ non-expansive for some $\alpha>0$ and $V_F$ is bounded. \\

\noindent \emph{Example.}
We recall from the first Section (or table 1 in \cite{Aminzare}) that when $F(D)$ is convex and
 $\psi_F$ is differentiable on $F(D)$, $\psi_F$ is $L$ non-expansive on $F(D)$
iff Sp$(J_{\psi_F}(y)+J_{\psi_F}^*(y)) \subset (-\infty, 2L]$ for any $y\in  F(D)$. 
% Let us illustrate this
%condition in the diffusive case , which be will useful later. We are needing the following transformation 
%(see Lemma \ref{transformation})
In the case $q=0$,  choosing 
$$F(x)=(f_i(x_i) :  i=1, \ldots,   \d)$$ 
and setting $A(x)=J_{\psi_F}(F(x))$, we have for  any $i,j=1,\ldots, d$  such that $i\ne j$
%$$A_{ji}(x)=\frac{1}{f_j'(x_j)}\left(\frac{\partial}{\partial x_j} b^i(x)f_i'(x_i)+1_{j=i}f_i''(x_i)b^i(x)\right) $$
%and for any $j\ne i$
\be
\label{matrice}
A_{ij}(x)= \frac{f_i'(x_i)}{f_j'(x_j)}\frac{\partial}{\partial x_j} b^{(i)}(x), \qquad A_{ii}(x)=\frac{\partial}{\partial x_i} b^{(i)}(x)+\frac{f''_i(x_i)}{f'_i(x_i)}b^{(i)}(x).
\ee
Then $\psi_F$ is  $L$ non-expansive on   $F(D)$ iff the largest eigenvalue of
 $A(x)+A^*(x)$  is less than $2L$  for any   $x \in D$. 
 %In Lemma \ref{transformation}, we are using the trace and 
 %determinant of $A(x)$ to construct  non-expansive domains $D$ for competitive models in dimension $2$.
 % pour un exemple en dimension 2 s'appuyant sur trace et d\'eterminant de $A(x)$.
%Dans l'exemple il suffira de montrer que les valeurs sont n\'egatives 
%(trace  n\'egative et d\'eterminant positif
%en dimension 2).
%Soit :$$ $$

\begin{proof}[Proof of Theorem \ref{corutile}]
%\marginpar{Ref prolongement} 
Under Assumption \ref{assume}, we can further assume that 
 $F \in \mathcal C^2(  \R^{\d}, \R^{\d})$. Indeed,  we can  consider $\varphi F$ where  $\varphi  \in \mathcal  C^{\infty}(  \R^{\d}, \R^{\d})$  is equal to $0$ on the complementary set of $O$  and to $1$ on $\overline{D \cup E}$, since these two sets are disjoint closed sets, using e.g. the smooth Urysohn lemma. This allows to extend $F$ from $\overline{D\cup E}$ to $\R^d$ in such a way that   $F\in \mathcal C^2(  \R^{\d}, \R^{\d})$.  \\
Applying now   It\^o's formula to $F(X_t)$ (see Chapter 2, Theorem 5.1 in  \cite{IW}), we have  :
\ben
F(X_t)&=&F(x_0)+
\int_0^t  J_F(X_s)b(X_s)ds
%\sum_{i=1}^{\d}  \frac{\partial F}{\partial x_i}(X_s) b^{(i)}(X_s)ds
+ \int_0^t\int_E \left[F(X_{s-}+H(X_{s-},z))-F(X_{s-}) \right] N(ds,dz)  \\
&& + \int_0^t\sum_{i,j=1}^{\d}  \frac{\partial F}{\partial x_i}(X_s)
\sigma^{(i)}_j(X_s) dB^{(j)}_s  +\int_0^t\int_E \left[F(X_{s-}+G(X_{s-},z))-F(X_{s-}) \right] \widetilde N(ds,dz) \\
&& + \int_0^t  \widetilde{b}_F(X_{s})ds
%\frac{1}{2}\int_0^t \sum_{i,j=1}^{\d} \frac{ \partial^2 F}{\partial x_i \partial x_j}(X_s) \sum_{k=1}^{d}\sigma^{(i)}_k(X_s)
%\sigma^{(j)}_k(X_s)ds+\int_0^t\int_E \left[F(X_{s}+G(X_{s},z))-F(X_{s})-\sum_{i=1}^{\d} \frac{\partial F}{\partial x_i}(X_s)G^{(i)}(X_{s},z) \right] dsq(dz)
\een
for $t\geq 0$. Then the $\mathcal F_t$-semimartingale  $Y_t=F(X_t)$ takes values in $F(E)$ and can be written as
\be
\label{decY}
Y_t=F(x_0)+\int_0^t \psi(Y_s)ds+ A_t+M_t^c+M_t^d,
\ee
where $\psi,A,M^c$ and $M^d$ are defined as follows. First,  we consider the Borel locally bounded function
$\psi(y)=1_{\{y \in F(D)\}}\psi_F(y)$ for $y\in \R^\d$, so 
writing $\hat b_F(x)= J_F(x)b(x)+ h_F(x)$ for $x\in E$, we have 
$\psi(Y_s)=1_{\{Y_s \in F(D)\}}\hat b_F(X_s)$. Moreover,
%remark  that 
%$\psi_F ( F(x))=h_F(x)$ for $x\in  D\cap E$ where $h_F(x)= J_F(x)b(x)+ \hat{b}_F(x)$ is well defined for  $x\in E$.
%We set $\psi(y)=1_{\{y \in F(D)\}}\psi_F(y)$, so $\psi(Y_s)=1_{\{Y_s \in F(D)\}}h_F(X_s)$ and%=1_{y \in F(D)}(J_F(y)b(y)+ \int_{\mathcal X} [F(y+H(y,z))-F(X_s)] q(dz))$;
$$A_t=\int_0^t \left(  \widetilde{b}_F(X_{s})+ \1{Y_s \not\in F(D)}\hat b_F(X_s)\right)ds$$
%\frac{1}{2} f''(X_s^i)(\sigma^i(X_s))^2dt+ \int_0^t\int_E \left[f(X_{s-}^i+G^i(X_{s-},z))-f(X_{s-}^i)-f'(X_{s-}^i)G^i(X_{s-},z) \right] dsq(dz)$$
is a continuous $\mathcal F_t$-adapted process with a.s. bounded variations paths and
$$M_t^{c}=\int_0^t \sum_{i,j=1}^{\d}  \frac{\partial F}{\partial x_i}(X_s)
\sigma^{(i)}_j(X_s) dB^{(j)}_s $$
is a continuous $\mathcal F_t$-local martingale and  writing  $K=G+H$ and using Assumption $\ref{assume}$ $(iii)$, 
$$M_t^{d}=\int_0^t\int_{\mathcal X} \left[F(X_{s-}+K(X_{s-},z))-F(X_{s-}) \right] \widetilde N(ds,dz)   $$
is a c\`adl\`ag $\mathcal F_t$-local martingale purely discontinuous. \\
We observe that %Notons que $\psi_F$ est $\varepsilon-L$ non-expansif sur $F(D)$, $F$ bijective, et que le syst\`eme dynamique 
the dynamical system $y_t=F(\phi_F(x_0,t))$ satisfies for $t<T(x_0)$,
$$y_0=F(x_0), \quad y_t'=J_F(\phi_F(x_0,t))b_F(\phi_F(x_0,t))=\psi_F(y_t)=\psi(y_t),$$
since $\psi_F=\psi$ on $F(D)$.
This flow is thus associated with  the vector field $\psi$ 
and $\psi$ is locally Lipschitz  on $F(D)$. Moreover, recalling the definition (\ref{tempssortie}) and setting  $E'=F(D)$, $T'(y_0)=T_D(x_0)$,
% the notations of Section \ref{trajj} %As $F$ is bijective from $D$ into $F(D)$, 
the first  time $T_{F(D), \varepsilon}(y_0)$ when $(y_t)_{t\geq 0}$ starting from $y_0$
 is at distance $\varepsilon$ from  the boundary of   $F(D)$ for the euclidean distance
is larger than $T_{D,\varepsilon, F}(x_0)$ defined by (\ref{deftemps}) :
$$T_{F(D), \varepsilon}(y_0)=\sup\{ t \in [0,T'(y_0))  : \forall s\leq t, \    \  \overline{B}(y_s,\varepsilon) \cap  F(E) \subset F(D)\}
%= \exists y \not\in F(D), \ F(\phi(x_0,t))\in \overline{B}(x,\varepsilon)\}\wedge T'(y_0)  
\geq T_{D,\varepsilon, F}(x_0).$$  
Adding that  $\psi$ is $(L, \alpha)$ non-expansive on $F(D)$,
we apply now
 Proposition \ref{ctrltpscourt} to $Y$  with $p=q=1$ and $Y_0=y_0=F(x_0)$. Then, for any  $T< T_{D, \varepsilon,F}(x_0)\wedge  \T{\alpha}{L}{\varepsilon}$, we get
 % $\varepsilon\leq 1$ (so $\varepsilon^2 \leq \varepsilon$) and we obtain                                                          
\begin{eqnarray}
 \mathbb P\left( S_T \geq \varepsilon\right)  \nonumber
&\leq &    C_{\d}e^{4LT}
\bigg[\varepsilon^{-1}\E \left( \int_0^T \1{S_{t-} \leq \varepsilon} dÊ\parallel \vert A\vert_t \parallel_1\right) \\
&& \qquad \quad  + \varepsilon^{-2}\E\left( \int_0^T\1{ S_{t-} \leq \varepsilon} d \parallel <M^c>_t\parallel_1\right)  
+\varepsilon^{-2}\E\left(  \sum_{t\leq T}\1{ S_{t-} \leq \varepsilon} \parallel \Delta Y_{t}\parallel_2^2   \right)\bigg]
\label{inek}
\end{eqnarray}
for some constant $C_{\d}$ positive, where $S_t=\sup_{s\leq t}\parallel Y_s-y_s \parallel_2$.
%since  $X_0=x_0$ a.s. ensures that the first probability on the  right hand side in 
%Proposition \ref{ctrltpscourt} is null. 
Using now 
$$<M^{c}>_t=\int_0^t  \sum_{i,j,k=1}^{\d}  \frac{\partial F}{\partial x_i}(X_s)\frac{\partial F}{\partial x_j}(X_s)\sigma^{(i)}_k(X_s)\sigma^{(j)}_k(X_s)ds,$$
%et
%$$<M_t^{d}>=\int_0^t\int_{\mathcal X} \left[F(X_{s-}+K(X_{s-},z))-F(X_{s-}) \right]^2 dsq(dz)  $$
%f'(F^{-1}(y)_i)b^i(F^{-1}(y))+\int_{\mathcal X} [f(x_i+H^i(x,z))-f(x_i)] q(dz)+\int_{\mathcal X} [f(x_i+G^i(x,z))-f(x_i)] q(dz)
we get
$$\int_0^T \1{S_{t-} \leq \varepsilon} d\parallel <M^c>_t\parallel_1 \leq 
\int_0^T \sup_{\substack{ x\in E \\ d_F(x, \phi_F(x_0,t)) \leq \varepsilon }}  \left\{\sum_{i,j,k,l=1}^{\d}  \frac{\partial F^{(l)}}{\partial x_i}(x)\frac{\partial F^{(l)}}{\partial x_j}(x)\sigma^{(i)}_k(x)\sigma^{(j)}_k(x)\right\}dt,$$
since  $S_t=\sup_{s\leq t}\parallel Y_s-y_s \parallel_2=\sup_{s\leq t} d_F(X_s, \phi_F(x_0,s))$. Similarly,
%$$\1{S_{t-} \leq \varepsilon} <M_t^{d}> \leq \int_0^t \sup_{\substack{ x\in E \\ \parallel F(x)-\phi_F(x_0,s)\parallel_2 \leq \varepsilon }} \left(
%\int_{\mathcal X} [F(x+K(x,z))-F(x)]^2 q(dz)\right) ds$$
%\marginpar{To be checked, coh\'erence des normes}
\Bea
 \E\left( \sum_{t\leq T}\1{ S_{t-} \leq \varepsilon} \parallel \Delta Y_{t}\parallel_2^2  \right)&=&
\E\left(\int_0^T \int_{\mathcal X} \1{ S_{t-} \leq \varepsilon}\parallel F(X_{t-}+K(X_{t-},z))-F(X_{t-}) \parallel^2_2 dtq(dz) \right)  \\
&\leq & \int_0^T \sup_{\substack{ x\in E \\  d_F(x,\phi_F(x_0,t))\leq \varepsilon }}   \int_{\mathcal X} \parallel F(x+K(x,z))-F(x)\parallel_2^2 q(dz) dt
\Eea
and
combining the two last inequalities we get
\be
\label{domvar}
\E\left( \int_0^T\1{ S_{t-} \leq \varepsilon} d \parallel <M^c>_t\parallel_1\right)  
+\E\left(  \sum_{t\leq T}\1{ S_{t-} \leq \varepsilon} \parallel \Delta Y_{t}\parallel_2^2  \right)
\leq \int_0^T  \sup_{\substack{ x\in E \\ d_F(x,\phi_F(x_0,t)) \leq \varepsilon }}  \parallel V_F(x) \parallel_1 dt.
\ee
Finally, on the event $\{S_{t-} \leq \varepsilon\}$, $Y_{t-}=F(X_{t-}) \in F(D)$ for any $t\leq T$ since  $T<T_{D, \varepsilon,F}(x_0)$, so 
%\marginpar{Checker cela et la preuve en g\'en\'eral}
%$\1{S_{t-}\leq \varepsilon}\1{F(X_t) \not\in F(D)}=0$ a.e. and
\be
\label{domdrift}
\E \left(\int_0^T \1{S_{t-} \leq \varepsilon} d  \parallel \vert A\vert_t \parallel_1
\right) \leq \int_0^T   \1{S_{t-} \leq \varepsilon} \parallel \widetilde{b}_F(X_{t-})\parallel_1  dt\leq \int_0^T\sup_{\substack{ x\in E \\ d_F(x,\phi_F(x_0,t)) \leq \varepsilon }}\parallel \widetilde{b}_F(x) \parallel_1dt \ee
%using the following upper bound for  $\widetilde{b}_F$:  
% $$\overline{b}_F(x_0,s)=\sup_{\substack{ x\in E \\ d_F(x,\phi_F(x_0,s)) \leq \varepsilon }} \parallel \widetilde{b}_F(x) \parallel_1.$$
and the conclusion comes by plugging the two last inequalities in (\ref{inek}).
\end{proof}

\subsection{Adjunction of non-expansive domains}
\label{adjunction}

We relax here the assumptions required for Theorem \ref{corutile}.
Indeed finding a transformation which guarantees non-expansitivity of the flow is delicate in general.
Adjunction of simple transformations is relevant for covering the whole state space and  leading computations.
It is useful for the study of two-dimensional competitive processes in Section  \ref{LV}. 
Let us note that the trajectorial estimates obtained previously is well adapated to gluing domains, while this is a delicate problem for controls 
of stochastic processes relying for instance on Lyapounov functions.
%and associating a family of transformations as follows. \\
%\marginpar{attention a juste autoriser lesprolongements et sauts sur $E\cap D_i$ alors qu'on veut $E$}
Thus, we  decompose the domain $D$ as follows.
%into a family of subdomains  $(D_i :  i=1, \ldots, N)$ and require the following set of assumptions. \\
%\marginpar{merde d c'est la dimension}
\begin{assumption} \label{assumeplus}
(i) The domains  $D$ and $(D_i :  i=1, \ldots, N)$ are open subsets of $\R^{\emph{d}}$ and $F_i$ 
are  $\R^\d$ valued functions
 from an open set $O_i$ which contains $\overline{D_i}$   and
$$D \subset \cup_{i=1}^N D_i,  \qquad F_i \in \mathcal C^2(O_i,\R^{\emph{d}}).$$
Moreover
$F_i$ is a bijection from $D_i$ into $F(D_i)$  whose Jacobian matrix
is invertible on $D_i$. \\
(ii) There exist a distance $d$ on $ \cup_{i=1}^ND_i \cup E$ and $c_1,c_2>0$
such that for any $i \in\{1, \ldots, N\}$, $x,y\in  D_i$,
$$c_1d(x,y) \leq \parallel F_i(x)-F_i(y) \parallel_2 \leq c_2 d(x,y).$$
%\marginpar{Verifier que juste sur $D_i$ ca suffit dans calculs. \\  Flow bien unique}
 (iii) For each  $i\in \{1,\ldots ,N\}$, for any $x\in E \cap D_i$,
$$\int_{\mathcal X} \left\vert F_i(x+H(x,z))-F_i(x)\right\vert  q(dz) <\infty.$$
and the  function
$x\in E \cap D_i\rightarrow h_{F_i}(x)= \int_{\mathcal X} [ F_i(x+H(x,z))-F_i(x)] q(dz)$ can be extended
to $\overline{D_i}$.\\
Moreover this extension is locally bounded on $\overline{ D_i}$ and locally Lipschitz on $D_i$.\\
(iv) The function $b$ is locally Lipschitz on $\cup_{i=1}^N D_i$. 
\end{assumption}

Second, we consider the flow associated to the vector field $b_{F_i}$, where $b_{F_i}$ and  defined  as previously by
$b_{F_i}(x)=b(x)+J_{F_i}(x)^{-1}h_{F_i}(x)$ and is locally Lipschitz on the domain $D_i$ . But now the flow $\phi$
 may go from one domain  to an other. To glue the estimates obtained in the previous part
 by adjunction of  domains, we need to bound the number 
of times $\kappa$ the flow may change of domain. 
More precisely, 
we consider    a   flow $\phi(.,.)$ such that $\phi(x_0,0)=x_0$ for $x_0 \in D$ and 
let  $\varepsilon_0 \in (0,1)$, $\kappa\geq 1$ and $(t_k(.) :  k \leq \kappa)$ be a sequence  of elements of $[0,\infty]$ such that   
$0=t_0(x_0)\leq t_1(x_0)\leq \ldots \leq t_\kappa(x_0)$ for $x_0\in D$,   which meet the following assumption.
 \begin{assumption}
\label{assumepluss}
 %  such that $0=t_0(x_0)\leq t_1(x_0)\leq \ldots \leq t_\kappa(x_0)=T(x_0)$
For any  $x_0 \in D$, the flow $\phi(x_0,.)$ is continuous on $[0,t_{\kappa}(x_0))$ and for any  $k\leq \kappa-1$, there exists $n_k(x_0)  \in \{1, \ldots, N\}$ 
such that for any $t \in (t_k(x_0),t_{k+1}(x_0))$, 
%we write $$F_{x_0,t}= F_{n_k(x_0)}$$
$$\overline{B}_{d}(\phi(x_0,t), \varepsilon_0) \subset   D_{n_k(x_0)} \quad \text{and} \quad \frac{\partial}{\partial t} \phi(x_0,t)= 
b_{F_{n_k(x_0)}}(\phi(x_0,t)).$$
%where we write $$F_{x_0,t}= F_{n_k(x_0)}.$$
\end{assumption}
\noindent %This  allows us to make an adjunction ofdomains $D_i$ where  we can find a transformation
%$F_i$ to obtain a non-expansive vector field and then compare the process
This flow $\phi$  will be used in the continuous case  in  Section \ref{LV}.  Then we recall
that  $b_F=b$ does not depend on the  transformation $F$ and  the flow
$\phi$ is directly    given by $\phi(x_0,0)=x_0, \  \frac{\partial}{\partial t} \phi(x_0,t)= b(\phi(x_0,t))$ as expected. \\
% In that case,  Assumption \ref{assumepluss} simply means that $\overline{B}_{d}(\phi(x_0,t), \varepsilon_0)$ is included in one of the subdomains $D_i$ for $\varepsilon$ small enough.
%One can keep in mind that $d$ is a distance to control the gap between the process $X$ and its approximation by a dynamical system.\\

%$$T_{D,\varepsilon,F}(x_0)=\inf\{ tÊ\in  [0,T) :  t\in  (t_k(x_0),t_{k+1}(x_0)),  \exists x \not\in D_{n_k(x_0)}, x \in \overline{B}_{d_{F_{n_k(x_0)}}}(\phi(x_0,t), \varepsilon)\}.$$
%The flow $\phi$ leaves in $D$ and the vector field in $D_i$ is given by $b_{F_i}$ (when 
\noindent Recalling  notation
$\psi_F=(J_Fb_F)\circ F^{-1}$ and the expressions of $\T{\alpha}{L}{\varepsilon}$ and  $\widetilde{b}_F$ and $V_F$ given respectively  in (\ref{temps}),  (\ref{defbt}) and (\ref{defVF}),
%\marginpar{intervalle max a repreciser ici ? }
%$$\widetilde{b}_F(x)= \sum_{i,j=1}^{\d} \frac{1}{2} \frac{\partial^2 F}{\partial x_i\partial x_j}  (x)\sigma^{(i)}(x)\sigma^{(j)}(x) +\int_{\mathcal X} [F(x+G(x,z))-F(x)-J_F(x)G(x,z)] q(dz) $$
%and 
%$$V_F(x)=   \sum_{i,j,k=1}^{\d}  \frac{\partial F}{\partial x_i}(x)\frac{\partial F}{\partial x_j}(x)\sigma^{(i)}_k(x)\sigma^{(j)}_k(x)+\int_{\mathcal X} [F(x+H(x,z)+G(x,z))-F(x)]^2 q(dz),$$
the result can be stated as follows.
\begin{thm} \label{recollement} 
%\marginpar{attention epsilon et $\alpha$ !!! inverser $k$ et $x_0$ partout ?}
Under Assumptions \ref{assumeplus} and  \ref{assumepluss}, we assume that  for each $i\in \{1,\ldots,N\}$, $\psi_{F_i}$ is $(L_i,\alpha_i)$ non-expansive on $F_i(D_i)$
and let $T_0\in (0,\infty)$. \\
Then for any $\varepsilon \in (0, \underline{\varepsilon}]$ and   $T< \min\left\{\T{\alpha_i}{L_i}{\varepsilon} : i = 1, \ldots, N\right\}\wedge t_{\kappa}(x_0)  \wedge T_0$ and  $x_0\in E \cap D$, %we have
$$\mathbb P_{x_0}\left( \sup_{t\leq T} d(X_t,\phi(x_0,t)) \geq \varepsilon\right) \leq C \sum_{k=0}^{\kappa-1} \int_{ t_k(x_0)\wedge T}^{t_{k+1}(x_0)\wedge T}
 \overline{V}_{d,\varepsilon}(F_{n_k(x_0)}, x_0,t)  dt,$$
with     $\underline{\varepsilon}$ and $C$   positive constants which depend (only) on  \emph{d}, $c_1,c_2, (L_i)_{ i=1,\ldots, N}$, $\kappa$, $\varepsilon_0$ and $T_0$;  and
 $$   \overline{V}_{d,\varepsilon}(F,x_0,s)=\sup_{\substack{ x\in E \\ d(x,\phi(x_0,s)) \leq \varepsilon }} \left\{\varepsilon^{-2} \parallel V_F(x) \parallel_1+\varepsilon^{-1}\parallel \widetilde{b}_F(x) \parallel_1\right\}.$$
\end{thm}
%\marginpar{Plus g\'eneration, l'\'evolution continue de la transformation $F$ et l'extension \`a des processus de Markov en environnement changeant sont des perspectives int\'eressantes;
%A reprendre, am\'eliorer les domaines}
The proof  relies also  on Proposition \ref{ctrltpscourt} but it is technically more involved than  the proof of Theorem  \ref{corutile} .  We observe that $T_0$ could be chosen equal to $\infty$ in this  statement in the  case where $L_i=0$ for any $i\in \{1, \ldots, \kappa\}$.
We need now the following constants.
%$$ b_n(T)=2\sqrt{2}\exp(L_nT), \quad  a_n(T)=b_n(T)\frac{c_2}{c_1}, \quad \varepsilon_n(T)=\frac{c_1\varepsilon_0}%{c_2b_{n}(T)}=\frac{\varepsilon_0}{a_n(T)},$$ for $n=1, \ldots, N$ 
$$ b_k(x_0,T)=2\sqrt{2}\exp(L_{n_k(x_0)}T), \quad  a_k(x_0,T)=\frac{c_2}{c_1}b_k(x_0,T), \quad \varepsilon_k(x_0,T)=\frac{c_1\varepsilon_0}{c_2b_{k}(x_0,T)}=\frac{\varepsilon_0}{a_k(x_0,T)},$$
for $k=0,\ldots,\kappa-1$ and observe that $a_k(x_0,T)\geq 1$.
\begin{lem} \label{lemadj}  
Under  the assumptions of  Theorem \ref{recollement},  for any $x_0 \in E\cap  D$, $k\in \{0,\ldots,\kappa-1\}$ and $(\varepsilon,T)$ such that $\varepsilon \in (0,\varepsilon_{k}(x_0,T)]$ and 
$T< \T{\alpha_{n_k(x_0)}}{L_{n_k(x_0)}}{\varepsilon}\wedge t_{\kappa}(x_0)$,  we have
 \Bea
&&\mathbb P_{x_0}\left( \sup_{t_k(x_0) \leq t \leq t_{k+1}(x_0)\wedge T} d(X_t,\phi(x_0,t)) \geq \varepsilon a_{k}(x_0,T)\right)
\\ 
&&\qquad \qquad \leq  \p(d(X_{t_k(x_0)} , \phi(x_0,t_k(x_0)) \geq \varepsilon)
+ C \int_{t_k(x_0)\wedge T}^{t_{k+1}(x_0)\wedge T}\overline{V}_{d,\varepsilon a_{k}(x_0,T)}(F_{n_k(x_0)},x_0,s) ds,
\Eea
where $C$ is a  positive constant which depends only on $\emph{d}$ and $c_1$ and $L_{n_k(x_0)}$.
 \end{lem}
\begin{proof} Let us fix $k\in \{0,\ldots,\kappa-1\}$ and $x_0 \in E\cap D$. We  write $L= L_{n_k(x_0)}$, $\alpha=\alpha_{n_k(x_0)}$,  $F=F_{n_k(x_0)}$ and $D=D_{n_k(x_0)}$ for simplicity and consider  
$T< \T{\alpha}{L}{\varepsilon}\wedge  t_{\kappa}(x_0)$.   As at the beginning of the previous proof, we can assume that
$F\in \mathcal{C}^2(\R^{\d},\R^{\d})$ and recall that $F$ is  bijection from $D$ into $F(D)$.  We note that $z_0=\phi(x_0,t_k(x_0))\in D$ by Assumption \ref{assumepluss}
and the solution $z$  of $z_t'=b_F(z_t)$  is well defined on a non-empty (maximal) time interval since $b_F$ is locally Lipschitz on $D$ using Assumption \ref{assumeplus}.
By uniqueness in Cauchy Lipschitz theorem, $z_t=\phi(x_0,t_k(x_0)+t)$ for $t\in [t_k(x_0),t_{k+1}(x_0))$. 
We write now  $\widetilde{X}_t=X_{t_k(x_0)+t}$ and the counterpart of  (\ref{decY}) for $Y_t=F(\widetilde{X}_t)$ is
%\marginpar{attent avec def actuelle on ne peut compenser que sur $D_i$}
\be
\label{decYY}
Y_t=Y_{0}+\int_{0}^t \psi(Y_s)ds+ A_t+M_t^c+M_t^d,
\ee
for $t\geq 0$,
where
$\psi(y)=1_{\{y \in F(D)\}}\psi_F(y)$, $$M_t^{c}=\int_0^t \sum_{i,j=1}^{\d}  \frac{\partial F}{\partial x_i}(\widetilde{X}_s)\sigma^{(i)}_j( \widetilde{X}_s) dB^{(j)}_s $$
and we make here the following decomposition for $A$ and $M^d$.  Using  Assumption \ref{assumeplus} $(iii)$  for the compensation of jumps when $\widetilde{X}_{s-}\in D$,
% and writing again $\hat b_F(x)= J_F(x)b(x)+ h_F(x)$, we  
%remark that 
%\{and 
we set 
 %=1_{y \in F(D)}(J_F(y)b(y)+ \int_{\mathcal X} [F(y+H(y,z))-F(X_s)] q(dz))$;
%\marginpar{ b tilde  est perturbe par les compensations. Ajouter d\'etail ?}
\Bea
A_t&=&\int_0^t \left( \widetilde{b}_F(\widetilde{X}_s)+\1{F(\widetilde{X}_{s}) \not\in F(D) }J_F(\widetilde{X}_s)b(\widetilde{X}_s) -\1{\widetilde{X}_{s} \not\in  D,F(\widetilde{X}_{s}) \in F(D) }
h_F\circ F^{-1} (Y_s)\right)ds\\
&& \qquad +  \int_0^t\int_{\mathcal X} \1{\widetilde{X}_{s-}\not\in D} \left[F(\widetilde{X}_{s-}+H(\widetilde{X}_{s-},z))-F(\widetilde{X}_{s-} )\right]  N(ds,dz), 
\Eea
%\frac{1}{2} f''(X_s^i)(\sigma^i(X_s))^2dt+ \int_0^t\int_E \left[f(X_{s-}^i+G^i(X_{s-},z))-f(X_{s-}^i)-f'(X_{s-}^i)G^i(X_{s-},z) \right] dsq(dz)$$
which is a process with a.s. finite variations paths;
and 
\Bea
M_t^{d}&=&\int_0^t\int_{\mathcal X} \left[F(\widetilde{X}_{s-}+G(\widetilde{X}_{s-},z))-F(\widetilde{X}_{s-}) \right] \widetilde N(ds,dz)\\
&&\qquad \qquad +\int_0^t\int_{\mathcal X} \1{\widetilde{X}_{s-}\in D} \left[F(\widetilde{X}_{s-}+H(\widetilde{X}_{s-},z))-F(\widetilde{X}_{s-}) \right] \widetilde N(ds,dz)   
\Eea
is a c\`adl\`ag $\mathcal F_t$-local martingale purely discontinuous.\\
Moreover  by Assumptions  \ref{assumepluss} and \ref{assumeplus} $(ii)$, for any $t<t_{k+1}(x_0)-t_k(x_0)$,  $x_t \in D$, $y_t=F(x_t)\in F(D)$ and satisfies $y_t'=\psi(y_t)$
and for any  
$\varepsilon \in (0,c_1\varepsilon_0],$
$$ \overline{B}(y_t, \varepsilon)\cap F(E)\subset F(\overline{B}_{d_F}(z_t,\varepsilon)) \subset  F(\overline{B}_{d}(z_t,\varepsilon/c_1)) \subset F(D).$$ 
Adding that $\psi=\psi_F$ is $(\alpha,L)$ non-expansive on $F(D)$, we can   apply     Proposition \ref{ctrltpscourt} to the process $Y$ on $F(D)$ for $p=q=1$ and $E'=F(D)$ and get   for any 
$\varepsilon \in  (0, c_1\varepsilon_0]$,
 \Bea
 && \mathbb P_{x_0}\left( \sup_{t \leq T_1}  \parallel Y_t- y_t \parallel_2 \geq \varepsilon\right) \\
%&& \qquad \qquad\leq \mathbb P_{x_0}\left( \sup_{t \in I_k(x_0,T) } d_{F_{n_k(x_0)}}(X_t,\phi(x_0,t)) \geq \varepsilon\right) \\
&& \qquad \qquad \leq   \p\left( \parallel Y_0- y_0 \parallel_2   \geq  \varepsilon /b_{k}(x_0,T_0)\right)+ 
C \varepsilon^{-1}\E \left(\int_{0}^{T_1} \1{S_{t-} \leq \varepsilon} d\parallel \vert A\vert_t \parallel_1 \right) \\
&& \qquad \qquad \qquad \qquad + C \varepsilon^{-2} \bigg[\E\left( \int_{0}^{T_1}   \1{ S_{t-} \leq \varepsilon}d \parallel <M^c>_t\parallel_1\right)  
+\E\left(  \sum_{t\leq T_1}\1{ S_{t-} \leq \varepsilon} \parallel \Delta Y_{t}\parallel_2^2  \right)\bigg] 
\Eea
for any $T_1 <\T{\alpha}{L}{\varepsilon}\wedge (t_{k+1}(x_0)-t_k(x_0))$, where $C$ is positive  constant depending on $L_{n_k(x_0)}$ and $\d$. Following  (\ref{domvar}) and $(\ref{domdrift})$ in the proof of Theorem \ref{corutile}, we obtain 
\be
&&\mathbb P_{x_0}\left( \sup_{ [t_k(x_0) \wedge T,t_{k+1}(x_0)\wedge T)}  d_F(X_t,\phi(x_0,t)) \geq \varepsilon\right) \nonumber \\
& &  \qquad  \quad \leq  \p(d_F( X_{t_k(x_0)},x_{t_k(x_0))}) \geq \varepsilon/b_{k}(x_0,T)) +
 C' \int_{t_k(x_0) \wedge T}^{t_{k+1}(x_0)\wedge T}  \overline{V}_{F,\varepsilon}(x_0,s)  ds. \label{eqqq}
\ee
for some  constant $C' $ depending also only of  $L$ and $\d$, where $ \overline{V}_{F,\varepsilon}$ has beed defined in (\ref{defOVER}). Using
again Assumption \ref{assumeplus} $(ii)$ to replace $d_F$  by $d$ above,
we have  
%, which we write
%$$c_1d(x,y) \leq d_k(x,y)  \leq c_2 d(x,y)$$
%for $x,y \in D_{n_k(x_0)}$ since then $\widetilde{F}_{n_k(x_0)}=F_{n_k(x_0)}$.
%It ensures that
% $$ T_{D_{n_k(x_0)}, \varepsilon,d_k}\geq t_{k+1}(x_0)$$ 
%and $I_k(x_0,T)=[t_k(x_0) ,T \wedge t_{k+1}(x_0))$. Moreover, for any  $t\in  I_k(x_0,T)$,
%$$For , Assumption \ref{assumeplus} $(iv)$ ensures that $\overline{B}_{d_{F_{n_k(x_0)}}} (\phi(x_0,t), \varepsilon) \in D_{n_k(x_0)}$, so 
%$$\left\{ \sup_{t \in I_k(x_0,T) } d_k(X_t,\phi(x_0,t)) < \varepsilon \right\} %\subset \left\{\sup_{t \in I_k(x_0,T) } d_{F_{n_k(x_0)}}(X_t,\phi(x_0,t)) <\varepsilon \right\}\subset 
%\subset \left\{\sup_{t \in I_k(x_0,T) } d(X_t,\phi(x_0,t)) <\varepsilon/c_1 \right\}$$
%and using that $d_{\widetilde{F_i}}(x,y) < \varepsilon$ and $y \in D_i$  implies $Bx\in D_i$  and $d_{F_i}(x,y) < \varepsilon$ by
%and  
$$\left\{d( X_{t_k(x_0)},\phi(x_0, t_k(x_0))) < \varepsilon/(c_2b_{k}(x_0,T))\right\} \subset   \left\{d_F( X_{t_k(x_0)},\phi(x_0, t_k(x_0))) < \varepsilon/b_{k}(x_0,T)\right\}$$
and
$$\overline{V}_{F,\varepsilon}(x_0,s) \leq (c_1^{-1}+c_1^{-2})\overline{V}_{d,\varepsilon/c_1}(F,x_0,s)$$
and we  obtain % from (\ref{eqqq})
\Bea
&&\mathbb P_{x_0}\left( \sup_{[t_k(x_0)\wedge T,t_{k+1}(x_0)\wedge T)}  d(X_t,\phi(x_0,t)) \geq \varepsilon/c_1\right)\\
& &   \qquad \leq  \p\left(d( X_{t_k(x_0)},\phi(x_0, t_k(x_0))) \geq  \frac{\varepsilon }{c_2 b_{k}(x_0,T)}\right) +
 C' (c_1^{-1}+c_1^{-2}) \int_{t_k(x_0)\wedge T}^{t_{k+1}(x_0) \wedge T}  \overline{V}_{d,\varepsilon/c_1}(F,x_0,s)  ds.
\Eea
%Adding that \marginpar{precier $C_1,C_2$}
Using  the quasi-left continuity of $X$, this inequality can be extended to
the closed interval $[t_k(x_0) \wedge T , t_{k+1}(x_0)\wedge T]$ for $k<\kappa-1$. This  ends  the proof  by replacing $\varepsilon$ by $\varepsilon c_2 b_{k}(x_0,T)$. 
 \end{proof}
 
 \begin{proof}[Proof of Theorem \ref{recollement}] We write $T_m= T_0\wedge \min\left\{\T{\alpha_i}{L_i}{\varepsilon} : i = 1, \ldots, N\right\}\wedge t_{\kappa}(x_0)\in (0,\infty)$ and set
 $$\underline{\varepsilon}=\inf\{\varepsilon_k(x_0,T) \ : \ k=1,\ldots, N; x_0 \in  E\cap D; T<T_0\} \in (0,\infty).$$ 
 Lemma \ref{lemadj}  and Markov property  
at time  $t_k(x_0)\wedge T$ ensure that for any  $\varepsilon \in (0,\underline{\varepsilon}]$,  $x_0\in E \cap D$,  $T \in (0,T_m)$,
  \Bea
&& \p_{x_0}\left( \sup_{[t_k(x_0) ,  t_{k+1}(x_0)\wedge T]} d(X_t,\phi(x_0,t)) \geq \varepsilon a_k(x_0,T), \quad 
\sup_{ [0, t_k(x_0)\wedge T]} d(X_t,\phi(x_0,t)) < \varepsilon \right)\\
&&\qquad \qquad \qquad  \qquad  \qquad  \qquad  \qquad  \qquad  \qquad \leq 
 C \int_{t_k(x_0)\wedge T}^{  t_{k+1}(x_0)\wedge T}\overline{V}_{d,\varepsilon a_k(x_0,T)}(F_{n_k(x_0)},x_0,s) ds
\Eea
 for each $k=0,\ldots,\kappa-1$, by setting $C=\max\{ C_{\emph{d},c_1,L_i}  :  i=1, \ldots, N\}$.\\
Denoting $A_{k}(x_0,T)=\Pi_{i\leq k} a_i(x_0,T)$ and recalling that $a_i(x_0,T)\geq 1$, by iteration we obtain for $\varepsilon \leq \underline{\varepsilon}/A_{\kappa}(x_0,t)$
and $T<T_m$ that
 \Bea
&& \p_{x_0}\left( \bigcup_{k=0}^{\kappa-1} \left\{ \sup_{[t_k(x_0), t_{k+1}(x_0)\wedge T]} d(X_t,\phi(x_0,t)) \geq  \varepsilon A_{k}(x_0,T)\right\}\right)\\
&&\qquad \qquad \qquad  \qquad \leq C \sum_{k=0}^{\kappa-1}
  \int_{t_k(x_0)\wedge T}^{t_{k+1}(x_0)\wedge T}\overline{V}_{d,\varepsilon A_{k}(x_0,T)}(F_{n_k(x_0)}, x_0,s) ds,
\Eea
since $X_0=x_0=\phi(x_0,0)$. This ensures that for any $T<T_m$,
 \Bea
&& \p_{x_0}\left( \sup_{0 \leq t  \leq  T} d(X_t,\phi(x_0,t)) \geq \varepsilon  A_{\kappa}(x_0,T)\right)\\
&&\qquad \qquad \qquad  \qquad \leq   C A_{\kappa}(x_0,T)^{2} \sum_{k=0}^{\kappa-1}
 \int_{t_k(x_0)\wedge T}^{t_{k+1}(x_0)\wedge T}\overline{V}_{d, \varepsilon A_{\kappa}(x_0,T)}(F_{n_k(x_0)}, x_0,s) ds,
\Eea
%\marginpar{am\'eliorer  cela et $[a,b]=\varnothing$ si $b>a$?} 
Recalling  that  $(n_k(x_0) : k=0,\ldots, \kappa)$ takes value in a finite set,  $A_{\kappa}(x_0,T)$  is bounded
for  $x_0\in E\cap D$ and $T\in  [0,T_0)$  by a constant depending only on $\kappa$, $(L_i : i =1, \ldots, N)$, $c_1$ and $c_2$. This  
yields the result.
\end{proof}

\section{Coming down from infinity for one-dimensional Stochastic Differential Equations}
\label{dimensionun}
%\marginpar{Un ti point de c'est bien  d\'efini \\ Ajouter une partie VF serait utile pour $\Lambda$-coalescent et 
%et scaling dim 2}

In this section, we assume that
 $E\subset \R$ and $+\infty$ is a limiting   value of $E$ and   $D=(a,\infty)$  for some $a\in (0,\infty)$.
Following the beginning of the previous section, we consider a  c\`adl\`ag Markov process $X$ 
  which takes values in E and assume that it is the
unique strong solution of the following SDE on $[0, \infty)$ :
\Bea
X_t&=&x_0+\int_0^t b(X_s)ds+\int_0^t \sigma(X_s)dB_s+\int_0^t\int_{\mathcal X} H(X_{s-},z) N(ds,dz)+\int_0^t\int_{\mathcal X} G(X_{s-},z) \widetilde{N}(ds,dz),
\Eea
for any $x_0 \in E$, 
where   we recall that $(\mathcal X, \mathcal B_{\mathcal X})$ is a measurable space;
$B$ is a  Brownian motion; % mouvements Brownien ind\'ependants; 
 $N$ is a Poisson point measure   on $\R^+\times \mathcal X $ with intensity  $dsq(dz)$;
 % where $q$ is a sigma finite measure on $(\mathcal X, \mathcal B_{\mathcal X})$; and $\widetilde{N}$ is the compensated measure.
%\marginpar{%associ\'e \`a un PP de class (QL) (Definition 3.1, chapitre II \cite{IW}, mesure $\sigma$ finie et 
%compens\'ee continue adapt\'e bien d\'ef,   automatq pour PPP station,   dt on note $\widetilde{N}$ la mesure compens\'ee}
$N$ and $B$  are independent and $HG=0$.
%\item $b=(b^{(i)} : i=1,\ldots,\d)$ and $\sigma=(\sigma^{(i)}_j : i,j=1,\ldots,\d)$ and $H$ and $G$ are Borelian functions locally bounded on $E$. % fonctions bor\'eliennes localement born\'ees. \\
%\end{itemize}
%\marginpar{Attention \`a a et au temps de sortie}
We  make the following assumption, which is a slightly stronger counterpart of Assumption \ref{assume} and is convenient for the study of the coming down infinity in dimension  $1$.
% for the applications.
  \begin{assumption} \label{assumedim1}
Let $F\in \mathcal C^2((a', \infty), \R)$,  for some $a' \in [-\infty, a)$ such that $\overline{E} \subset (a', \infty)$. \\ %, \infty)$, for some on an open set $O$ containing  $\overline{D\cup E}$ and $
  (i) For any $x> a$, 
$F'(x)>0$ and $F(x)\rightarrow \infty$ as $x\rightarrow \infty$.\\
(ii) For any $x \in E$,
$\int_{\mathcal X} \left\vert F(x+H(x,z))-F(x)\right\vert  q(dz) <\infty$. \\
The function
$x\in E \rightarrow h_F(x)=\int_{\mathcal X} [F(x+H(x,z))-F(x)] q(dz)$
can be extended to 
 $\overline{E}\cup[a,\infty)$. \\
 This extension is locally bounded  on  $\overline{E}\cup[a,\infty)$ and locally Lipschitz on $(a,\infty)$. \\
 (iii) $b$ is locally Lipschitz on $(a,\infty)$.\\
(iv)  The function $b_F=b+h_{F}/F'$ is negative  on $(a,\infty)$.
\end{assumption}

 \noindent  %Let us first note that $F$ is then a bijection from $(a,\infty)$ into $(F(a), \infty)$. \\
Following  the previous sections, we consider now the flow  $\phi_F$ given  for $x_0 \in (a,\infty)$ by
$$\phi_F(x_0,0)=x_0, \qquad \frac{\partial }{\partial t}\phi_F(x_0,t)=b_F(\phi_F(x_0,t)),$$
which is  well  and uniquely defined and belongs to $(a,\infty)$ on  a maximal time interval denoted by  $[0, T(x_0))$, where $T(x_0)\in (0, \infty]$.
We first observe that $x_0\rightarrow \phi_F(x_0,t)$ is increasing where it is well defined. This can be seen by recalling that the local Lipschitz property ensures the uniqueness of solutions and thus  prevents the trajectories from intersecting. Then   $T(x_0)$ is increasing and  its limit when  $x_0\uparrow \infty$ is denoted  by $T(\infty)$ and belong to $(0,\infty]$. Moreover,  the flow  starting from infinity
is  well defined by a monotone limit : 
%\marginpar{Remplacer les $\infty$ par $+\infty$ ?}
$$\phi_F(\infty,t)=\lim_{x_0\rightarrow \infty} \phi_F(x_0,t)$$
 for any $t\in [0, T(\infty))$. Finally, under Assumption \ref{assumedim1},  for $x_0\in (a,\infty)$,  $b_F(x_0)<0$   and    for any $t<T(x_0)$,
$\int_{x_0}^{\phi_F(x_0,t)} 1/b_F(x)dx=t$. This  yields the following classification. \\

Either
$$\int_{\infty}^. \frac{1}{b_F(x)} <+\infty,$$
and then  
 $$\phi_F(\infty,t)=\inf \left\{ u \geq 0 : \int_{\infty}^u \frac{1}{b_F(x)} dx <t \right\}<\infty$$
  for any $t\in (0, T(\infty))$.
   We  say that  the dynamical system \emph{instantaneously comes down from infinity}. Moreover  the application $t \in [0, T(\infty))\rightarrow \phi(\infty,t) \in \overline{\R}$  is continuous, where $\overline{\R}=\R \cup\{\infty\}$ is endowed with the distance
 \be \label{defRbard}
 \overline{d}(x,y)=\vert e^{-x}-e^{-y}\vert.
 \ee
 %and we have$$\phi(\infty,t)=.$$

%\marginpar{Add an argument}
 Otherwise,   $T(\infty)=\infty$ and $\phi(\infty,t)=\infty$ for any  $t\in [ 0,\infty)$.  \\
 %Let us not that  it is still continuous. As detailed below,  this classification of the coming down from infinity is inherited by the stochastic as soon as the assumptions of the previous section are met. \\

\noindent Our aim now is to derive an analogous  classification  for stochastic differential equations using the results of the previous section.
Letting the process start from infinity requires additional work.
We give first a condition useful for the identification of the limiting values of $(\p_x : x\in E)$ when $x\rightarrow\infty$.
% stochastically monotone SDE. %, see r \cite{BMR} pour le cas des naissances et morts. 
\begin{defi}  The process $X$ is stochastically monotone if
for all $x_0,x_1 \in E$ such that  $x_0\leq x_1$, for all  $t>0$ and $x\in \R$, we have
$$\p_{x_0}(X_{t} \geq x)\leq \p_{x_1}(X_{t} \geq x).$$
\end{defi}
\noindent The  $\Lambda$-coalescent,  the birth and death process,
continuous diffusions with strong pathwise uniqueness and several of their extensions satisfy this property, while e.g.  the Transmission Control Protocol does not and we refer to the examples of forthcoming Section \ref{Ex} for details. \\

\subsection{Weak convergence and coming down from infinity} 
We recall that 
 $\overline{\R}=\R\cup\{\infty\}$ endowed with $\overline{d}$ defined by $(\ref{defRbard})$ is polish 
 and 
% We also note $w(f, \delta, [0,T]) =\sup_{s,t \leq T, \  \vert t-s\vert \leq \delta} \overline{d}(f_s,f_t)$.
the notation of the previous section become
$\psi_F=(F'b_F)\circ F^{-1},$
%qui est aussi localement Lipschitz, soit
%$$\psi_F(F(x))=J_F(x)b(x)+ \int_{\mathcal X} [F(x+K(x,z))-F(x)] q(dz)$$
%Enfin
$\widetilde{b}_F(x)=F''(x)\sigma(x)^2 +\int_{\mathcal X} [F(x+G(x,z))-F(x)-F'(x)G(x,z)] q(dz) $
and
$V_F(x)=  (F'(x)\sigma(x))^2+\int_{\mathcal X} [F(x+H(x,z)+G(x,z))-F(x)]^2 q(dz).$\\
In this section,   we introduce
$$\hat{V}_{F, \varepsilon}(a,t) = \sup_{ x \in  E\cap \mathfrak D_{F, \varepsilon}(a,t) }  \left\{ \varepsilon^{-2}V_F(x)+\varepsilon^{-1} \widetilde{b}_F (x) \right\},$$
where for convenience we use 
 the extension $F(\infty)=\infty$ and we set %$F^{-1}(\infty)=\infty$, we set 
$$ \mathfrak D_{F,\varepsilon} (a,t)%= \left(   a, F^{-1}(F(\phi(\infty, t) )+ \varepsilon)\right)
=\{ x\in (a,\infty) : F(x)\leq F(\phi(\infty,t))+\varepsilon\}.$$
%\sup_{\substack{ x\in E \\ d_F(x,\phi_F(x_0,s)) \leq \varepsilon }}  \vert V_F(x) \vert + \vert \widetilde{b}_F(x) \vert$$
Finally, we make the following key assumption to use the results of the previous section.
\begin{assumption} \label{controledim1}  The vector field
 $\psi_F$ is $(L, \alpha)$ non-expansive  on $(F(a),\infty)$ 
 and for any $\varepsilon >0$,
\be
\label{condevar}
\int_0^.   \hat{V}_{F, \varepsilon} (a,t)dt <\infty.
\ee
\end{assumption}
\noindent  Let us remark that   $\psi_F$ is $(L, \alpha)$ non-expansive  on $(F(a),\infty)$  iff for all 
$y_1>y_2>F(a)$,
$\psi_F(y_1) \leq  \psi_F(y_2)+ L( y_1-y_2 ) + \alpha$. This means that for all $x_1>x_2>a$, $F'(x_1)b(x_1)+h_F(x_1)\leq F'(x_2)b(x_2)+h_F(x_2)+L(F(x_1)-F(x_2))+\alpha$.\\

Let us now  give  sufficient conditions for the convergence of  $(\mathbb P_x)_{ x\in E}$ as $x\rightarrow \infty$. For that purpose, we
%Given $0 \leq A \leq B$, we call \emph{subdivision of $[A,B]$} a finite sequence $\b = (b_\ell, \ell = 0, \ldots, L)$ such that $b_0 = A < b_1 < \cdots < b_L = B$. For $\eta > 0$, we say that a subdivision $\b$ is $\eta$-sparse if $b_{\ell+1} - b_\ell > \eta$ for every $\ell = 0, \ldots, L-2$. For $\eta > 0$, $0 \leq A < B$ and $f \in D_\X$ let
introduce the modulus
\begin{equation} \label{eq:def-w}
	w' (f, \delta, [A,B]) = \inf_{{\bf b}} \max_{\ell = 0, \ldots, L-1} \sup_{b_\ell \leq s, t < b_{\ell+1}} \overline{d}(f_s, f_t)
\end{equation}
where the infimum extends over all subdivisions ${\bf b}= (b_\ell, \ell = 0, \ldots, L)$ of $[A,B]$ which are $\delta$-sparse. We refer to Chapter 3 in \cite{Billingsley}  for details on the Skorokhod topology.

%\marginpar{Check $(i)$  and domaine $\hat{V}$ !!!}
\begin{prop} \label{cvinf} We assume that $X$ is stochastically monotone. \\
(i) If    $E=\mathbb \{0,1,2, \ldots\}$, then $(\mathbb P_x)_{ x\in E}$ converges weakly as $x\rightarrow \infty$  in the space of  probability  measures on  $\mathbb{D}([0,T], \overline{\R})$.\\%, \overline{d}))$.\\
(ii)  If
  Assumptions \ref{assumedim1} and  \ref{controledim1} hold   and $\int_{\infty}^. \frac{1}{b_F(x)} <+\infty$ and
 for any $K>0$ and $\varepsilon>0$,
\be
\label{tensioncompact}
\lim_{\delta \rightarrow 0} \sup_{x \in E, \  x \leq K}\p_{x}\left( w'(X, \delta, [0,T])\geq \varepsilon\right)=0,
\ee
then  $(\mathbb P_x)_{ x\in E}$
converges weakly as $x\rightarrow \infty$ in the space of  probability  measures on  $\mathbb{D}([0,T], \overline{\R})$.
\end{prop}
%\marginpar{Faire bien la jonction avec les exemples. Lambda c'est Donnelly ou autre ? que citent Berest ?}

\noindent The   convergence  result $(i)$  concerns the discrete case $\sigma=0$. It has been obtained   
in \cite{Donnelly91}
 when the limiting probability $\p_{\infty}$ is  known a priori and the process comes down from infinity. The proof 
 of the tightness for $(i)$ follows   \cite{Donnelly91} and relies on the monotonicity and the fact that the states are non-instantaneous, which is here due to our c\`adl\`ag %\marginpar{reprendre} 
 assumption for any initial state space. The identification of the limit is derived directly from the monotonicity and the proof of $(i)$ is actually a direct extension of    
 Lemma 2.1 in \cite{BMR}.  This proof is omitted.  \\
  The tightness argument for  $(ii)$ is different and  can be applied to  processes with a continuous part and extended to  larger dimensions. 
The control of the fluctuations of the process for large values  relies on the approximation by the  continuous dynamical system $\phi_F$ using Assumption \ref{controledim1} and the previous section.  Then the tightness on compacts sets is guaranteed by  $(\ref{tensioncompact})$. The proof is given 
below. \\
In the next result, we assume that  $(\mathbb P_x)_{ x\in E}$ converges weakly and $\p_{\infty}$ is then well defined as the limiting probability. We  determine   under our assumptions    when (and how) the process comes down from infinity. More precisely, we link the coming down from infinity of the process $X$ to that
of the flow $\phi_F$, in the vein of  \cite{BBL, limxsi, BMR} who considered   some classes of discrete processes, see below for details.

\begin{thm} \label{dim1descente}  We assume that  Assumptions \ref{assumedim1} and \ref{controledim1} hold and that $(\mathbb P_x : x\in E)$
converges weakly as $x\rightarrow \infty$ in the space of  probability  measures on  $\mathbb{D}([0,T], \overline{\R})$ to  $\p_{\infty}$.

%$X$ is stochastically monotone, $b_F$ is $L, \varepsilon$ non expansive (pour un certain $\varepsilon>0$) on $[F(a),\infty)$ and non zero for $x$ large enough
%and $\widetilde{b}_f$ is bounded and
%\be\label{condevar}\int_0^.  \sup_{ x_0 \in E } V_F(x_0,s) <\infty.\ee
 (i) If  $$\int_{\infty}^. \frac{1}{b_F(x)} <\infty,$$
then 
$$\mathbb P_{\infty}(\forall t >0 : X_t<+\infty)=1 \quad \text{and } \quad
\mathbb P_{\infty}\left( \lim_{t\downarrow 0+} F(X_t)-F(\phi_F(\infty,t))  =0\right)=1.$$

(ii) Otherwise $\mathbb P_{\infty}(\forall t \geq 0 : X_t=+\infty)=1$.
\end{thm}
\noindent
After the proof given below, we  consider   examples with different size of fluctuations at infinity.  
For  $\Lambda$-coalescent, we recover  the speed of coming down from infinity
of  $\cite{BBL}$ using  $F=\log$ and in that case  $V_F$ is bounded.
For birth and death processes with polynomial death rates, fluctuations are smaller and we  use $F(x)=x^{\beta}$ ($\beta <1$)
and get a finer approximation of the process coming down from infinity by a dynamical system. But $V_F$ is no longer bounded and has to be controlled  along the dynamical system coming down from infinity.
When proving that some birth and death processes or Transmission Control Protocol do not come down from infinity, $\mathfrak D_{F,\varepsilon}(a,t)$ is non-bounded and
we are looking for  $F$ increasing slowly enough so that   $V_F$  is bounded to check (\ref{condevar}), see the next section for details.  \\
%\marginpar{Preciser domaines et je pense se debarasser de $x_0$ dans les hyp !!} \\

The proofs of the two last results need  the following lemma.  We recall  notation
$D=(a,\infty)$,
$d_F(x,y)=\vert F(x)-F(y) \vert$ 
and 
$T_{D, \varepsilon, F}(x_0)$ resp.   $\T{\alpha}{L}{\varepsilon}$ given 
  in $(\ref{deftemps})$ resp. $(\ref{temps})$.
%=\inf\{ tÊ\in  [0,T'(x_0)) : \exists x \not\in D,  \  d_F(x,\phi_F(x_0,t))\leq \varepsilon\}\wedge T'(x_0). $$
\begin{lem} \label{newlem} Under   Assumptions  \ref{assumedim1}  and  \ref{controledim1}, for 
any  $\varepsilon>0$,  $x_0 \in E \cap D$ and $T<T_{D, \varepsilon, F}(x_0)\wedge \T{\alpha}{L}{\varepsilon}$, we have
$$\mathbb P_{x_0}\left( \sup_{t\leq T}d_F(X_t,\phi_F(x_0,t)) \geq \varepsilon\right) \leq C(\varepsilon,T),$$
where
$$C(\varepsilon,T)=C\exp(4LT)\int_0^T \hat{V}_{F, \varepsilon}(a,t) dt$$
goes to $0$ when  $T\rightarrow 0$  and $C$ is a positive constant.
\end{lem}
\begin{proof} Assumption \ref{assume} and the $(L,\alpha)$ non-expansivity of $\psi_F$
are guaranteed respectively by Assumptions  \ref{assumedim1}  and \ref{controledim1}, with here $O=(a',\infty)$ and $D=(a,\infty)$. Thus, 
we can apply 
  Theorem  \ref{corutile} on the domain $D$ and   for any $x_0 \in D\cap E$ and  $\varepsilon >0$ and $T<T_{D, \varepsilon, F}(x_0)\wedge \T{\alpha}{L}{\varepsilon}$, we have
  $$\mathbb P_{x_0}\left( \sup_{t< T}d_F(X_t,\phi_F(x_0,t)) \geq \varepsilon\right) \leq C\exp(4LT) \int_0^T \overline{V}_{F,\varepsilon}(x_0,s) ds.$$
Now let $t<T_{D, \varepsilon, F}(x_0)$ and  $x\in E$ such that
 $d_F(x,\phi_F(x_0,t)) \leq \varepsilon$. Then $x>a$ and 
 $F(a) < F( x)  \leq          F(\phi_F(x_0,t))+\varepsilon $
and combining  the monotonicities of the flow $\phi_F$ %with respect to the initial condition 
and the function $F$,
 $$ F( a)  <  F(x) \leq            F(\phi_F(\infty,t))+ \varepsilon,$$
 since $\phi(x_0,t)>a$.  Thus $x \in \mathfrak D_{F,\varepsilon} (a,t)$ and
 $$\overline{V}_{F,\varepsilon}(x_0,t) \leq \hat{V}_{F,\varepsilon}(a,t),$$
which ends up the proof, since the  behavior of $C(\varepsilon, T)$  when $T\rightarrow 0$ comes from ($\ref{condevar}$).
\end{proof}

\begin{proof}[Proof of the Proposition \ref{cvinf} (ii)]
The fact that $X$ is  a  stochastically monotone  Markov process ensures that
for all $x_0, x_1 \in E$, $x_0\leq x_1$,  $k\geq 0$, $0\leq t_1\leq  \ldots \leq t_k$, $a_1, \ldots, a_k\in \R$,
$$\p_{x_0}(X_{t_1} \geq a_1, \ldots ,X_{t_k} \geq a_k)\leq \p_{x_1}(X_{t_1} \geq a_1, \ldots , X_{t_k} \geq a_k).$$
It can be shown by induction for $k\geq 1$ by  using the Markov property at time $t_1$ and writing  $X_{t_1}^{x_1}=X_{t_1}^{x_0}+ B$, where $X^{x}$ is the process $X$ starting at $x$ and $B$ is a non-negative random variable $\mathcal F_{t_1}$)-measurable. Then $$\p_{x_0}(X_{t_1} \geq a_1, \ldots, X_{t_k} \geq a_k)$$ converges  as $x_0 \rightarrow \infty$ ($x_0\in E$) by monotonicity, which identifies the finite dimensional limiting distributions of $(\p_x : x\in E)$ when $x\rightarrow \infty$.  \\

Let us  turn to the proof of the tightness in the Skorokhod space $\mathbb{D}([0,T], \overline{\R})$ and  fix $\eta>0$.
The flow $\phi_F$ comes down instantaneously from infinity since   $\int_{\infty}^. 1/b_F(x) <\infty$. Thus,  we can choose  $T_0\in (0,T(\infty))$ such that 
$\phi_F(\infty,T_0) \in D$. Using also that $F$ tends to $\infty$,  let us now fix $K_1 \in [\phi_F(\infty,T_0),\infty)$  and $\varepsilon \in (0,\eta]$ such that  $\overline{d}(K_1,\infty)\leq \eta$ and  for any $x\geq K_1$ and $y\in \R$ such that $d_F(x,y)<\varepsilon$,  we have  $\overline{B}_{d_F}(x,\varepsilon) \subset D$ and $\overline{d}(x,y)<\eta$.
%\marginpar{Check et on retrouve ce point delicat que  la boule est implicitement d\'efinie sur l'ensemble de d\'efinition de $d$}
By continuity and monotonicity of $t\rightarrow \phi_F(\infty,t)$,  there exists $T_1\in (0,T_0]$ such that $\phi_F(\infty,T_1)=K_1+1$.
Adding that $T(x_0)\uparrow T(\infty) $ and $\phi_F(x_0,T_1)\uparrow \phi_F(\infty,T_1)$ as $x_0\uparrow \infty$,  we have  $\phi_F(x_0,T_1)\geq K_1$ for any $x_0$ large enough and  then  $ T_{D,\varepsilon,F}(x_0)\geq T_1$.
Thus, Lemma  \ref{newlem} ensures that  for any $x_0$ large enough and $T< T_1\wedge \T{\alpha}{L}{\varepsilon}$,
%,  for any $x_0 \in E\cap D$ and $T< T_{D,\varepsilon,F}(x_0)\wedge  \T{\alpha}{L}{\varepsilon}$, we have
\be
\label{borne}
\limsup_{x_0 \rightarrow \infty, x_0\in E} 
\mathbb P_{x_0}\left( \sup_{t \leq  T} d_F(X_t, \phi_F(x_0,t))  \geq \varepsilon\right) \leq 
C(\varepsilon,T),
\ee
where   $C(\varepsilon,T)\rightarrow 0$ as $T\rightarrow 0$.  Let now $T_2 \in (0,T_1\wedge \T{\alpha}{L}{\varepsilon})$ such that $C(\varepsilon, T_2)\leq \eta$. 
Using  that  for  any $t\in [0,T_2]$, 
$\phi_F(x_0,t)\geq K_1$ and 
 $\overline{d}(\phi_F(x_0,t),\infty)\leq \eta$ for $x_0$ large enough,
 %and  $d_F(X_t,\phi_F(x_0,t))\geq \varepsilon$  as soon as
 %$\overline{d}(X_t,\phi_F(x_0,t))\geq \eta$, we obtain  
 $$ \left\{ \sup_{t\leq T_2} \overline{d}(X_t, \infty)  \geq 2\eta \right\} \subset \left\{ \sup_{t\leq T_2} \overline{d}(\phi_F(x_0,t), X_t)  \geq \eta\right\} 
 \subset\left\{  \sup_{t \leq  T_2} d_F(X_t, \phi_F(x_0,t))  \geq \varepsilon \right\}.$$
Writing $K= F^{-1}(F(\phi(\infty,T_2))+\eta)$ and using   that $\phi_F(x_0, T_2)\uparrow \phi_F(\infty,T_2) \in D$,  we have also% we have also  \\
$$ \{ X_{T_2}\geq K\}\subset \left\{ F(X_{T_2})\geq F(\phi_F(\infty,T_2))+\eta \right\} \subset \left\{  d_F(X_{T_2}, \phi_F(x_0,T_2))  \geq \varepsilon \right\},$$ 
since $F'$ is  positive on $D$ and $\eta\geq \varepsilon$. Then  (\ref{borne})  and the two last inclusions  ensure that  %,
% \marginpar{On doir virer le F de la distance}
$$\mathbb P_{x_0}\left(  \left\{ \sup_{t\leq T_2} \overline{d}(X_t, \infty)  \geq 2\eta \right\} \cup \{ X_{T_2}\geq K\}\right) \leq \eta$$
for $x_0$ large enough.
Moreover, by  (\ref{tensioncompact}), for any $T\geq T_2$,  for $\delta$ small enough,
$$\ \sup_{x \in  E; x\leq K}\p_{x}\left( w'(X, \delta, [0,T-T_2])\geq 2\eta\right)\leq \eta. $$
Combining these two last bounds  at time $T_2$ by Markov property, we get that  for $x_0$ large enough and $\delta$ small enough,
$\p_{x_0}\left( w'(X, \delta, [0,{T}])\geq 2\eta\right)\leq 2\eta$.
The tightness is proved.
%(Markov, Gronwall and Doob for the EDS) and actually it is enough that for any $\varepsilon>0$,
\end{proof}

\begin{proof}[Proof of Theorem \ref{dim1descente}] We fix $\varepsilon >0$ and   let   $T_0\in (0,T(\infty)\wedge  \T{\alpha}{L}{\varepsilon})$  such that $\overline{B}_{d_F}(\phi_F(\infty,T_0), 2\varepsilon) \subset D$. We  observe that  
$T_{D,\varepsilon,F}(x_0)\geq T_0$ for $x_0$ large enough  since $\phi_F(x_0,T_0)\uparrow \phi_F(\infty,T_0)$ as $x_0 \uparrow \infty$ and 
$t\in [0,T(x_0))\rightarrow \phi_F(x_0,t)$ decreases. %Using  Assumptions \ref{assumedim1} and  \ref{controledim1}, we   can
We  apply Lemma
\ref{newlem} and 
 get for any $T<T_0$,
\be
\label{limsup}
\limsup_{ x_0 \rightarrow \infty, x_0 \in E} \mathbb P_{x_0}\left( \sup_{ t \leq  T}  d_F(X_t, \phi_F(x_0,t))  \geq \varepsilon\right) \leq 
C(\varepsilon,T),\ee
where $C(\varepsilon, T) \rightarrow  0$ as $T\rightarrow 0$. 
 
 We first consider the case $(i)$ and fix now also $t_0\in (0,T_0)$. The flow $\phi_F$ comes down from infinity instantaneously, so
 $\phi_F(\infty, t) <\infty$ on $[t_0, T]$. By Dini's theorem,  $\phi_F(x_0,.)$  converges to $\phi_F(\infty, .)$  uniformly on $[t_0, T]$,
 using  the  monotonicity of the convergence and the continuity of the limit.
We obtain from (\ref{limsup}) that for any  $T< T_0$,
$$\limsup_{ x_0 \rightarrow \infty, x_0 \in E} \mathbb P_{x_0}\left( \sup_{t_0 \leq t \leq T}  d_F(X_t, \phi_F(\infty,t))\  \geq 2\varepsilon\right) \leq 
C(\varepsilon,T),$$
and the  weak convergence  of  $(\mathbb P_x : x\in E)$
 to $\p_{\infty}$ yields 
$$\mathbb P_{\infty}\left( \sup_{t_0 \leq t \leq  T} d_F(X_t,\phi_F(\infty,t)) >2 \varepsilon\right) \leq 
C(\varepsilon,T).$$
Letting $t_0 \downarrow 0$ and then $T \downarrow 0$ ensures that
 $$\lim_{T\rightarrow 0 } \mathbb P_{\infty}\left( \sup_{0 < t \leq  T} d_F( X_t, \phi_F(\infty,t))  >2\varepsilon\right) =0.$$
% where the instantaneous coming down from infinity of $\phi_F$  ensures that $\phi_F(\infty,t)< \infty$ for any $t \in (0,T^a]$.
Then
$  \p_{\infty}\left( \lim_{t\downarrow 0+} F(X_t)-F(\phi_F(\infty,t))  =0\right)=1$ 
%\marginpar{Check, not markov, $\cup$ a am\'eliorier}
and $\mathbb P_{\infty}(\forall t >0 : X_t<\infty)=1$, which proves $(i)$.  

%Letting $T$ goes to $0$, and then $\varepsilon$ goes to $0$ yields the result for $i)$. \\\marginpar{ Attention il faut etendre  le rŽsultat a epsilon plus grand 1 et a plaide encore pour un lemme}
For the case $(ii)$, i.e. $\int_{\infty}^. 1/b_F(x) =\infty,$ we recall that
   $T(\infty)=\infty$, so  (\ref{limsup}) yields
$$ \mathbb P_{\infty}\left(  F(X_T) < \limsup_{x_0\rightarrow \infty} F(\phi(x_0,T))-A\right) \leq 
C(A,T)$$
for any $T\in (0, \T{\alpha}{L}{\varepsilon})$. Adding that  $F(\phi(x_0,T))\uparrow F(\phi(\infty,T))=F(\infty)=\infty$ as $x_0\uparrow \infty$, 
$$\p_{\infty} (X_T <\infty)\leq C(A,T).$$
Since $\phi(\infty,t)=\infty$ for any $t\geq 0$, $\mathfrak D_{F,A}(a,t)=(a,\infty)$ for any $A >0$. Then
$C(A,T)\leq \frac{1}{A} C(1,T)$ for $A\geq 1$ and  $C(A, T) \rightarrow  0$ as $A\rightarrow \infty$, since $C(1,T)<\infty$ by (\ref{condevar}). We get
 $\p_{\infty}(  X_T =\infty)=1$ for any $T>0$,
which  ends up the proof recalling that $X$ is a  c\`adl\`ag Markov  process under $\p_{\infty}$.
\end{proof}

\subsection{Examples and applications}
\label{Ex}

We consider here examples of processes in one dimension and  recover some known results. We also  get new estimates   and we  illustrate the assumptions required and the choice of $F$. 
Thus, we recover classical results on the coming down from infinity for $\Lambda$-coalescent and refine some of them for birth and death processes. Here $b,\sigma=0$ and the condition allowing the compensation of jumps   (Assumption \ref{assumedim1} $(ii)$) will be obvious. We also provide a criterion for  the  coming down from infinity of the Transmission Control Protocol, which is a piecewise deterministic markov process  with $b\ne 0,  \ \sigma=0$. 
 Several extensions of these results could be achieved, such as mixing branching coalescing  processes or additional  catastrophes. They are left for future works,  while the next section  considers  diffusions in higher dimension.

\subsubsection{$\Lambda$-coalescent \cite{Pitman, BBL} }
%\marginpar{et extension au kingman ? branchement coalescence? citer $\xi$ coalescent ? TLC ? Vlada}

 Pitman \cite{Pitman} has given a Poissonian representation of  $\Lambda$-coalescent. 
  We recall that  $\Lambda$ is a finite measure on $[0,1]$
  and we set  $\nu(dy)=y^{-2}\Lambda (dy)$. Without loss of generality, we assume that 
  $\Lambda [0,1]=1$ and for simplicity, we focus on coalescent without Kingman part and assume
  $\Lambda(\{0\})=0$. We consider a Poisson Point Process
on $(\R^+)^2$ with intensity  $dt\nu(dy)$ :  each atom 
 $(t,y)$ yields a coalescence event where each block is picked 
 independently with probability $y$ and all the blocks picked merge into a single bock.
Then the numbers of blocks jumps from 
  $n$  to  $B_{n,y}+1_{B_{n,y}<n}$, where $B_{n,y}$  follows a binomial distribution with parameter  ($n$,$1-y$).  Thus, the number of blocks $X_t$ at time $t$ is the solution of the SDE
\ben
X_t=X_0-\int_0^t \int_0^{1}\int_{[0,1]^{\mathbb N}}  \left(-1+\sum_{1 \leq  i\leq X_{s-}} 1_{u_i\leq y} \right)^+N(ds,dy,du),
\een
%\marginpar{peut tre faudrait il d\'efinir du? sur les cylindres}
where $N$ is a  PPM with intensity  on $\R^+\times [0,1]\times[0,1]^{\mathbb N}$ with intensity $dt\nu(dy)du$. Thus   here $E=\{1,2, \ldots \}$,  $\mathcal X=[0,1]\times[0,1]^{\mathbb N}$ is endowed with 
the cylinder $\sigma$-algebra of borelian sets of $[0,1]$, $q(dydu)=\nu(dy)du$ where $du$ is the uniform measure on $[0,1]^{\mathbb N}$, $\sigma=0$ and
%On note que  $X$ est d\'ecroissant. Ici
$$  H(x,z)=H(x,(y,u))=-\left(-1+\sum_{1 \leq i\leq x} 1_{u_i\leq y} \right)^+.$$
We follow \cite{BBL} and we denote for $x\in (1,\infty)$,
$$ F(x)=\log(x), \qquad \psi(x)=\int_{[0,1]}(e^{-xy}-1+xy)\nu(dy).$$
%\marginpar{integrability sauts contre $q$ a faire}
In particular $F$  meets the Assumption \ref{assumedim1} $(i)$ with $a>0$ and $a'=0$.
Moreover for every $x \in \mathbb N$,
\Bea
h_F(x)&=&\int_{\mathcal X} [F(x+H(x,z))-F(x)] q(dz) \\
&=& \int_{\mathcal X} \log \left(\frac{x+H(x,z)}{x}\right) q(dz) \\
&=&\int_{[0,1]} \nu(dy)\E\left(\log\left(\frac{ B_{x,y}+1_{B_{x,y}<x}}{x}\right)\right)= -\frac{\psi(x)}{x}+h(x), 
\Eea
where $h$ is bounded thanks to  Proposition  7 in \cite{BBL}. Thus
 $h$ can be extended to a bounded  $C^1$ function on $(0,\infty)$ and   Assumption \ref{assumedim1} $(ii)$ is satisfied.  Moreover,
$$\psi_F(x)=h_F(F^{-1}(x))=-\frac{\psi(\exp(x))}{\exp(x)} +h(\exp(x))$$
and  Lemma 9 in \cite{BBL}  ensures that $x \in (1, \infty) \rightarrow \psi(x)/x$ is increasing.
Then $\psi_F$ is  $(0, 2\parallel h\parallel_{\infty})$ non-expansive on 
$(0,\infty)$. Moreover here
%\marginpar{besoin de $C^1$ ? def d'ab $b_F$ ? check $b_F$ cas gŽnŽral, comme variance en gŽnŽral}
$$b_F(x)=F'(x)^{-1}h_F(x)=-\psi(x)+xh(x).$$
Adding that $\psi(x)/x\rightarrow \infty$ as $x\rightarrow \infty$, we get $b_F(x)<0$ for $x$ large enough and Assumption \ref{assumedim1} $(iv)$ is checked. 
 %\marginpar{ result de BBL sont bien vrais tout le temps ? Expliciter $w_t$ comme inf ??}
 Finally, $\widetilde{b}_F=0$ since  $\sigma=0$ and $G=0$ and  the second part of 
 Proposition 7 in  \cite{BBL} ensures that
$$V_{F}(x)= \int_{\mathcal X} [F(x+H(x,z))-F(x)]^2 q(dz)=\int_{[0,1]} \nu(dy)\E
\left(\left(\log\left(\frac{ B_{x,y}+1_{B_{x,y}<x}}{x}\right)\right)^2\right)  $$
is bounded for $x\in \mathbb N$.
Then Assumptions  \ref{assumedim1} and \ref{controledim1} are satisfied with $F=\log$, $a'=0$ and $a$ large enough. Moreover $(\p_x :  x\in \mathbb N)$ converges weakly to $\p_{\infty}$, which can be see here from Proposition \ref{cvinf}  $(i)$ since $X$ is stochastically monotone.
Thus Theorem  \ref{dim1descente} can be applied and writing
$w_t=\phi_F(\infty,t)$, we have  \\

%$X$ is stochastically monotone, $b_F$ is $L, \varepsilon$ non expansive (pour un certain $\varepsilon>0$) on $[F(a),\infty)$ and non zero for $x$ large enough
%and $\widetilde{b}_f$ is bounded and
%\be\label{condevar}\int_0^.  \sup_{ x_0 \in E } V_F(x_0,s) <\infty.\ee
\emph{ (i) If  $\int_{\infty}^. \frac{1}{b_F(x)} <+\infty,$
then  $w_t \in \mathcal C^1((0,\infty), (0,\infty))$ ,   $w_t'=-\psi(w_t)+ w_th(w_t)$ for $t>0$ and
$$\mathbb P_{\infty}(\forall t >0 : X_t<\infty)=1 \quad \text{and } \quad
\mathbb P_{\infty}\left( \lim_{t\downarrow 0+} \frac{X_t}{w_t}  =0\right)=1.$$ \\
\indent (ii) Otherwise $\mathbb P_{\infty}(\forall t \geq 0 : X_t=+\infty)=1$.} 
$\newline$

To compare with known results, let us note that
$b_F(x) \sim -\psi(x)$  as $x\rightarrow\infty$ and
  $$ \int^{\infty} \frac{1}{\psi(x)}dx<\infty \Leftrightarrow \int_{\infty}^. \frac{1}{b_F(x)} <\infty,$$
so that we recover here the criterion of coming down from infinity obtained in \cite{BL}.
This latter   is equivalent to the criterion initially proved in \cite{Schwein} : 
$$\sum_{n=2}^{\infty} \gamma_n^{-1}<\infty,$$
where
$$\gamma_n=-\int_{[0,1]}H(n,z) q(dz)=\sum_{k\geq 2} (k-1)\binom{n}{k} \int_{[0,1]}y^k(1-y)^{n-k}\nu(dy).$$

\noindent \emph{Remark 1 :   this condition can be rewritten as  $\int_{\infty} 1/b(x)dx<\infty$, where 
$b(x)$ is a locally Lipschitz function, which is non-increasing and equal to $-\gamma_n$ for any $n\in \mathbb N$ . But the proof cannot be achieved
using $F=Id$, even if $b(x)$ is non-expansive since $V_{Id}(x)$ cannot be controlled.} \\ \\
Finally, following \cite{BBL} let us consider the flow associated to the vector field $-\psi(\exp(x))/\exp(x)$ and write $v_t$ the flow starting from $\infty$.  In the case $(i)$ when the process comes down from infinity, we can use Lemma \ref{equivflott} in Appendix to check  that $\log (w_t)-\log(v_t)$ goes to $0$ as $t\rightarrow 0$ since
$\psi_F(x) + \psi(\exp(x))/\exp(x)$ is bounded. Thus
$$w_t \sim_{t\downarrow 0+} v_t, \qquad
\text{where } \quad v_t=\inf\left\{ s>0 : \int_s^{\infty} \frac{1}{\psi(x)} dx <t\right\}$$ satisfies $v_t'=\psi(v_t)$ for $t>0$. We recover here the speed of coming down from infinity of
\cite{BBL}. \\
%, where the martingale approach for the study of coming down from infinity has been first introduced, so as trajectorial control of functions using one dimensional properties. \\

\noindent \emph{Remark 2 : we have here proved that the speed of coming down from infinity is $w$  using Theorem  \ref{dim1descente} and 
\cite{BBL} and then observe that this speed function is equivalent to $v$. It is possible to recover directly that $v$ is the speed of coming down from infinity by using Proposition \ref{ctrltpscourt}  and  a slightly different decomposition of the process $X$ following \cite{BBL} : 
\ben
\log(X_t)=\log(X_0)-\int_0^t \psi(X_s)ds +\int_0^t \int_{\mathcal X}  \log\left(\frac{ X_{s-} + \left(-1+\sum_{i\leq X_{s-}} 1_{u_i\leq y} \right)^+}{X_{s-}}\right)\widetilde{N}(ds,dy,du)+A_t,
\een
where $A_t=\int_0^th(X_s)ds$ is a process with finite variations. More generally, one could
extend  the result of Section \ref{EDS} by adding  a term with finite variations in the SDE.} \\

\noindent \emph{Remark 3 : let us also mention that the speed of coming down from infinity for some $\Xi$ coalescent has been obtained in \cite{limxsi}
with a similar method than \cite{BBL}. The reader could find there other and detailed information about  the coming down from infinity of coalescent processes.  Finally, we mention \cite{LT1, LT2} for stimulating recent results
on the description of the fluctuations of the $\Lambda$- coalescent around the dynamical system $v_t$ for small times.}

%\marginpar{Janvier 2015 : Clement Foucart me dit qu'ils travaille sur les Lambda coalescent avec branchement}

\subsubsection{Birth and death processes \cite{Doorn1991, BMR}}
\label{secBD}
%\marginpar{ (et extension aux naissances multiples).?}
We consider a birth and death process $X$ and we denote by $\lambda_k$ (resp. $\mu_k$) the birth rate (resp. the death rate)
when the population size is equal to $k \in E=\{0,1,2, \ldots\}$.  
We  assume that $\mu_0=\lambda_0=0$ and $\mu_k>0$ for $k\geq 1$ and we 
denote
$$\pi_1=\frac{1}{\mu_1},  \quad \pi_k=\frac{\lambda_1\cdots\lambda_{k-1}}{\mu_1\cdots\mu_k} \ \ (k\geq 2).$$
We also
assume that  
\be
\label{ext}
\sum_{k\geq 1} \frac{1}{\lambda_k \pi_k}=\infty.
\ee
Then the process $X$ is well defined on $E$ and eventually becomes extinct  a.s. \cite{KarlinMcGregor57,Karlin1975}, i.e. $T_0=\inf\{ t>0 : X_t=0\}<\infty$ p.s.
It is the unique strong solution on $E$ of the following SDE
\ben
X_t=X_0+\int_0^t \int_0^{\infty} (1_{z\leq \lambda_{X_{s-}}}-1_{ \lambda_{X_{s-}}< z\leq  \lambda_{X_{s-}}+ \mu_{X_{s-}}}) N(ds,dz)
\een
where $N$ is a Poisson Point Measure  with intensity $dsdz$ on $[0,\infty)^2$.
Lemma  2.1 in \cite{BMR}
ensures that $(\p_x)_{ x\in E}$ converges weakly to $\p_{\infty}$. It can  also be derived from Proposition \ref{cvinf}  $(i)$ since $X$ is stochastically monotone. 
Under the extinction assumption (\ref{ext}), the following criterion for the coming down from infinity is well known  \cite{Anderson1991} : 
 \begin{equation}
\label{cvS}
S=\lim_{n\to \infty} \Esp_{n}(T_{0})=\sum_{i\geq1}\pi_i+\sum_{n\geq1}\frac{1}{\lambda_n\pi_n}\sum_{i\geq n+1}\pi_i
% =\sum_{n\geq0} \left(\sum_{i\geq n+1}\frac{\lambda_{n+1}\cdots\lambda_{i-1}}{\mu_{n+1}\cdots\mu_i} \right)
< +\infty.
\end{equation}

%\marginpar{donner la cond de $(\alpha,L)$ non expansif  pour les BD ?}
The speed of coming down from infinity of birth and death processes has been obtained in \cite{BMR} for regularly varying rate (with index $\varrho>1$)
and a birth rate negligible compared to the death rate. Let us here get a finer result 
for a relevant subclass which
 allows  rather simple computations and  describes   competitive model in  population dynamics. It 
  contains in particular the logistic birth and death process. 
 %It can be extended easily to larger classes provided that the birth and the death rates have the same asymptotic behavior  and in  particular it could include Kingman coalescent or other coalescent without multiple collisions. 
 \begin{prop} \label{regvar} We assume that there exist   $b\geq 0$, $c>0$  and $\varrho >1$ such that
 $$\lambda_k=bk, \qquad \mu_k=ck^{\varrho} \qquad (k\geq 0).$$
 Then for any $\alpha\in (0,1/2)$,
 $$\p_{\infty}\left( \lim_{t\downarrow 0+} t^{\alpha/(1-\varrho)}(X_t/w_t-1)=0\right)=1,$$
 where $$w_t \sim_{ t\downarrow 0+}  [ct/(\varrho-1)]^{1/(1-\varrho)}.$$
  \end{prop}
\noindent  This complements the results obtained in \cite{BMR}, where 
  it was shown that $X_t/w_t\rightarrow 1$ as $t\downarrow 0$. The proof used the decomposition of the trajectory in terms of the first hitting time of integers, which works well (in one dimension) when simultaneous  deaths can not occur.
 The fact that $X$ satisfies a central limit theorem when $t\rightarrow 0$ under $\p_{\infty}$ (see Theorem 5.1 in \cite{BMR})
ensures that the previous result is sharp in the sense that it does not hold for $\alpha\geq 1/2$. \\

\noindent \emph{Remark.  Using  (\ref{decc}) and  Lemma \ref{equivv} in Appendix, a more explicit form of  the previous result can be given  for $\alpha <(\varrho-1)\wedge 1/2$ :
%\marginpar{tout ˆ checker ds cstts}
$$\p_{\infty}\left( \lim_{t\downarrow 0+} t^{\alpha/(1-\varrho)}\left( \frac{X_t}{[ct/(\varrho-1))]^{1/(1-\varrho)}}-1\right)=0\right)=1.$$}

%\marginpar{dire que c'est plus ou moins connu}

Before the proof, we   consider the critical case where the competition rate  is  slightly larger than the birth rate. 
We  recover here the criterion for the coming down from infinity using Theorem  \ref{dim1descente}. We complement this result by providing estimates
both when the process comes and does not come from infinity. 
The function $f_{\beta}$ defined by $$f_{\beta}(x)=\int_2^{2+x}  1/\sqrt{y\log(y)^{\beta}} dy.$$
 provides the best distance (i.e. the fastest increasing function going to infinity) allowing to compare the process and the flow
 by bounding the quadratic variation. It  allows in particular to capture the fluctuations   when they do not come down from infinity, see $(ii)$ below.
\begin{prop} \label{lautre}
We assume that there exist $b\geq 0$,  $c>0$ and $\beta> 0$ such that
$$\lambda_k=bk, \qquad \mu_k=ck\log(k+1)^{\beta} \qquad (k\geq 0).$$
(i)
If $\beta>1$, then $\p_{\infty}(\forall t>0 : X_t<+\infty)=1$ and 
  $$\p_{\infty}\left( \lim_{t\downarrow 0+} f_{\beta}(X_t)-f_{\beta}(w_t)=0\right)=1,$$
   where $w_t=\phi_{f_\beta}(\infty,t)\in \mathcal C^1((0,\infty),(0,\infty))$.\\ \\
%\marginpar{donner un equiv avec l'index ?}
%$$\p_{\infty}\left( \lim_{t\downarrow 0+} f(X_t)-f(w_t)=0\right)=1,$$
%where $f_{\beta}(x)=\int_1^{1+x}  1/(y\log(y)^{\beta}) dx$ and $w_t=\phi_f(\infty,t)\in \mathcal C^1((0,\infty),(0,\infty))$. 
(ii) If $\beta\leq 1,$ $\p_{\infty}(\forall t\geq0 : X_t=+\infty)=1$
and for any $\varepsilon>0$,
$$\lim_{T\rightarrow 0} \limsup_{x_0 \rightarrow \infty, x_0\in \mathbb N} \p_{x_0}\left(\sup_{t\leq T} \vert f_{\beta}(X_t) -f_{\beta}(\phi_{f_{\beta}}(x_0,t)) \vert \geq \varepsilon\right)=0.$$
\end{prop} 
We  do not provide  more explicit estimates for the  flow $\phi_{f_{\beta}}$  or for  $w_t$ in short time for  that case and we turn to the proof of the two previous propositions.
Let us specify  notation for the birth and death process.
Here $\chi=[0,\infty)$, $q(dz)=dz$ and
$$H(x,z)=1_{z\leq \lambda_x} -1_{\lambda_x < z\leq \lambda_x+\mu_x}.$$
Letting $F\in  \mathcal C^1((-1,\infty),\R)$, we have $\int_{\mathcal X} \vert F(x+H(x,z))-F(x) \vert q(dz)<\infty$ and   %$h_F(x)$  is defined  by
$$
h_F(x)%=\int_{\mathcal X}  [F(x+H(x,z))-F(x)]  q(dz)
= (F(x+1)-F(x))\lambda_x+ (F(x-1)-F(x))\mu_x
$$
for $x \in \{0,1,\ldots\}$.
 For the  classes of  birth and death  rates $\lambda,\mu$ considered in the two previous propositions,   $h_F$ is   well defined  on $(-1,\infty)$ by the identity above and  $h_F\in \mathcal C^1((-1,\infty),\R)$.  Assumption \ref{assumedim1} (ii) will be checked with $a'=-1$. Finally
$$V_F(x)=(F(x+1)-F(x))^2\lambda_x+ (F(x)-F(x-1))^2\mu_x.$$

 \begin{proof}[Proof of Proposition \ref{regvar}]
 We  consider now  $\alpha \in (0, 1/2)$ and
 $$F(x)=(1+x)^{\alpha} \quad (\alpha\in (0,1/2)).$$
 Then $F'(x)>0$ for $x>-1$ and $h_F(x)= ((x+2)^{\alpha}-(x+1)^{\alpha})bx+ (x^{\alpha}-(x+1)^{\alpha})cx^{\varrho}$ and there exists $a>0$ such that $h_F'(x)<0$ for $x\geq a$.
 This ensures that Assumption \ref{assumedim1} is checked with $a'=-1$ and $a$. Moreover
$\psi_F=h_F\circ F^{-1}$  is non-increasing and thus 
 $\text{non-expansive on } ( F(a),\infty).$ \\
Adding that here
$$h_F(x) \sim_{x\rightarrow\infty} -c\alpha x^{\varrho +\alpha -1}$$
%It is well defined  on $(-1,\infty)$ and
%\marginpar{to check}
we get
\be
b_F(x)=F'(x)^{-1}h_F(x)%=\alpha^{-1}(1+x)^{1-\alpha}h_F(x)  %\sim_{x\rightarrow\infty}    
= -c(1+x)^{\varrho} +\mathcal O(x^{\max( \varrho-1,  1)}) \qquad (x\rightarrow \infty)\label{decc}
\ee
and one can use  Lemma \ref{equivflot} in Appendix with $\psi_1(x)=b_F(x)$ and $\psi_2(x)=-cx^{\varrho}$  to check that
%\marginpar{Check and explain, equiva enough ? Il taut un lemme, deja utilize pour coalescent } 
\be 
\label{equiv}
 \phi_F(\infty,t)\sim_{t\downarrow 0+}[ct/(\varrho-1)]^{1/(1-\varrho)}.
 \ee
Finally
$$V_F(x)=((x+2)^{\alpha}-(x+1)^{\alpha})^2bx+ ((x+1)^{\alpha}-x^{\alpha})^2 cx^{\varrho}
%\sim \alpha x^{2\alpha -2}\mu_x
\sim \alpha^2 c x^{\varrho +2\alpha -2} \qquad  (x\rightarrow \infty).$$
 Adding that 
%$F'(\phi_F(\infty,t))=\alpha\phi_F(\infty, t)^{\alpha-1}$ and 
for any $T>0$, there exists $c_0>0$ such that $ \phi(\infty,t)\leq c_0t^{1/(1-\varrho)}$ for $t\in [0,T]$ and that 
$F^{-1}(y)=y^{1/\alpha}-1,$
 then   for any  $\varepsilon>0$, there exists $c_0'>0$ such that for any $t\leq T$,
%Thus,    for any $T, \varepsilon>0$ such that $\varepsilon F'(x_0) <\phi_F(\infty,T)$, we have for any $s\leq T$,
$$\hat{V}_{F, \varepsilon}(a,t) \leq \varepsilon^{-2} \sup \left\{ V_F(x) \ :  \  0   \leq 
x \leq ((\phi_F(\infty,t)+1)^{\alpha}+ \varepsilon)^{1/\alpha}-1  \right\} \leq  \ c_0' (t^{1/(1-\varrho)})^{\varrho +2\alpha -2}.$$
Using that $(\varrho +2\alpha -2)/(1-\varrho)=-1+(2\alpha -1)/(1-\varrho)>-1$ since   $\alpha<1/2$, we obtain
$$\int_0^.  \hat{V}_{F, \varepsilon}(a,t)dt < \infty.$$
Thus Assumptions   \ref{assumedim1} and \ref{controledim1} are satisfied and
Theorem  \ref{dim1descente} $(i)$ can be applied, since $\int_{\infty}^. 1/b_F(x)<\infty$. Defining 
$w_t=\phi_F(\infty,t)$, we get
$\p_{\infty}\left( \lim_{t\downarrow 0+} X_t^{\alpha}-w_t^{\alpha}=0\right)=1$ for any $\alpha \in (0,1/2)$.
This ends up the proof recalling $(\ref{equiv})$. 
\end{proof}

%\marginpar{Et maintenant pour raffiner  $F(x)=x^{1/2}/f(x)$ croissant.}
%$$V_F(x)\sim F'(x)^2 \mu_x \sim  (x^{-1/2}/f(x))^2\mu_x\sim \alpha x^{\varrho -1}/f(x)^2$$
%$$\int_0^. t^{-1}/f(t^{1/(\varrho-1)})^2$$
%et donc $f(x)=(\log x)^{\beta}$ avec $\beta >1/2$ est possible. }

%\marginpar{\paragraph{Multiple birth.} Ca marche pareil }
%dans le cas de naissances multiples (dont le nombre de est donn\'ee par une probabilit\'e  $\mu$ sur les entiers), 
%\ben
%X_t=X_0+\int_0^t \int_0^{\infty} 1_{u\leq \lambda(X_{s-})} zN(ds,du,dz)-\int_0^t \int_0^{\infty} 1_{u\leq \mu(X_{s-})}N'(ds,du)
%\een
%avec $N$ MPP d'intensit\'e $dsdu\mu(dz)$ telle que $m:=\int_0^{\infty} z\mu(dz)<\infty$. }
\begin{proof}[Proof of Proposition \ref{lautre}] The criterion $\beta>1$ for the coming down from infinity
can be derived easily  from the criterion  $S<\infty$ recalled in (\ref{cvS}).
% and the fact that $\p_n$ converges to $\p_{\infty}$. 
It is also a consequence of Theorem \ref{dim1descente} using  $F(x)=(1+x)^{\alpha}$ as in the  previous proof and the integrability criterion for $\int_{\infty}^. 1/b_F(x)$, using that
$b_F(x)=h_F(x)/F'(x)\sim -c x\log(x+1)^{\beta}$ as $x\rightarrow \infty$. 
%Moreover  following the steps of the previous proof  also ensures the second part of $(i)$.
% $\p_{\infty}\left( \lim_{t\downarrow 0+} X_t^{\alpha}-w_t^{\alpha}=0\right)=1,$
 %  where $w_t=\phi_{f_\beta}(\infty,t)\in \mathcal C^1((0,\infty),(0,\infty)).$
\\
Let us turn to the  proof of the estimates $(i-ii)$ and take  $F=f_{\beta}$. Then $F(x)\rightarrow \infty$ as $x\rightarrow\infty$,
 $$h_F(x)= bx\int_{2+x}^{3+x}  \frac{1}{\sqrt{y\log(y)^{\beta}}}dy-cx\log(x)^{\beta}\int_{1+x}^{2+x}  \frac{1}{\sqrt{y\log(y)^{\beta}}}dy$$
% \mu_x=F'(x)x(b-c\log(x)^{\beta})$$
 %decreasing for $x$ large  enough, 
 %since 
 and its derivative 
%$$F''(x)x(b-c \log(x)^{\beta} )+ F'(x)(b-c\log(x)^{\beta}-c\beta \log(x)^{\beta-1}) $$
is negative for $x$ large enough. Then Assumptions   \ref{assumedim1} is satisfied with again $a'=1$ and $a$ large enough.
 So
 $\psi_F(x)=h_F(F^{-1}(x))$
 is decreasing  and thus  non-expansive for $x$ large enough.
Moreover there exists $C>0$ such that
$$V_F(x)\leq  C x\log(x)^{\beta}\left(\int_{1+x}^{2+x}  \frac{1}{\sqrt{y\log(y)^{\beta}}}dy\right)^2.$$
So $V_F$
is bounded and Assumption \ref{controledim1}  is satisfied. Then $(i)$ comes from Theorem \ref{dim1descente} $(i)$ and $(ii)$  comes from  Lemma \ref{newlem} observing that $T_{D,\varepsilon,F}(x_0)\rightarrow\infty$  as $x_0\rightarrow \infty$.
% Theorem  \ref{dim1descente} by noting that
%$$b_F(x)= (F'(x))^{-1}h_F(x)=-O(x\log(x)^{\beta})$$
%as $x \rightarrow \infty$. Then $\int_{\infty}^. \frac{1}{b_F(x)} dx=+\infty$ iff $\beta>1$, which ends up the proof.
\end{proof}

%\marginpar{To check, http://perso.math.univ-toulouse.fr/miclo/files/2015/06/stationary.pdf
%et les papier de ] Li-Juan Cheng and Yong-Hua Mao sur les strong ergodicity qui equivaut a descendre de l infini}

\subsubsection{Transmission Control Protocol } 
The 
Transmission Control Protocol  \cite{TCP}
is a model for transmission of data, mixing a continuous (positive) drift which describes the growth of 
the data transmitted and jumps  due to congestions, where the size of the data are divided by two.
%\marginpar{
%mod\'elise les transmissions de données sur des réseaux (comme Internet).
%Les donn\'ees sont fragment\'ees en paquets . A chaque \'etape, on envoie un certain
%nombre de paquets. S'ils sont tous re\c cus alors on envoie un paquet de plus à l '\'etape suivante.
%D\`es qu'un paquet est perdu, on parle de congestion et alors on divise par deux le nombre de
%paquets que l'on envoie.}% Ce type de processus permet d'optimiser le nombre de paquets envoy\'es. Alors}
Then the size $X_t$ of  data at time $t$ is given by the unique strong  solution on $[0,\infty)$  of
$$X_t=  x_0+bt-\int_0^t  \1{u\leq r(X_{s-}) }\frac{X_{s-}}{2} N(ds,du), $$ 
where $x_0 \geq 0$, $b>0$, $r(x)$ is a continuous positive non-decreasing function and $N$ is PPM on $[0,\infty)^2$ with intensity $dsdu$.  This is a classical example of Piecewise Deterministic Markov process. Usually, $r(x)=cx^{\beta}$, with $\beta \geq 0$, $c>0$. %First, let us note that the process is not stochastically monotone and the convergence of $(\p_x : x \geq 0)$ is
%left open. 
The choice of 
$F$ is a bit more delicate here owing to the size and intensity of the fluctuations.  Consider $F$ such that $F'(x)>0$ for $x>0$. Now $E=[0,\infty)$,
$h_F(x)=r(x)(F(x/2)-F(x))$,
$$b_F=b+h_F/F', \qquad \psi_F=(bF'+h_F)\circ F^{-1}.$$
% $h_F(x)=bF'(x)+r(x)(F(x/2)-F(x))$ and  $\psi_F(x)=h_F\circ F^{-1}$.  As soon as $r$ is differentiable, we have
%$$h_F'(x)=bF''(x)+r'(x)(F(x/2)-F(x))+ r(x)(F'(x/2)/2-F'(x)).$$
 Finally
$$V_F(x)=r(x)(F(x/2)-F(x))^2 $$
and we cannot use   $F(x)=(1+x)^{\gamma}$ or  $F(x)=\log(1+x)^{\gamma}$ since then the second part of Assumption \ref{controledim1} does not hold. We need to reduce the size of the jumps
even more and take $F(x)= \log (1+ \log(1+x)).$
 The model is not stochastically monotone but Lemma \ref{newlem}  can be used to get the following result, which yields a criterion for the coming down from infinity.
%\marginpar{details  a donner, domaine a faire attention, et le cas de non descente; } 
\begin{prop} (i) If there exists $c>0$ and $\beta >1$ such that
$r(x)\geq c\log(1+x)^{\beta}$ for any $x\geq 1$, then for  any $T>0$, $\eta>0$, there exists $K$ such that

 $$\inf_{x_0 \geq 0}\p_{x_0}(\exists t\leq T :  X_t\leq K)\geq 1-\eta.$$
 (ii)  If there exists $c>0$ and $\beta \leq 1$ such that
$r(x)\leq c\log(1+x)^{\beta}$ for any $x\geq 0$, then for  any $T,K>0$, 
 $$\lim_{x_0\rightarrow \infty} \p_{x_0}(\exists t\leq T :  X_t\leq K)=0.$$
\end{prop}
\noindent Thus, in the first regime, the process comes down instantaneously and a.s. from infinity, while in the second  regime it stays at infinity, even if $\mathbb P_{\infty}$ has not been constructed here. In particular, if $r(x)=cx^{\beta}$ and $\beta,c>0$, the  process comes down instantaneously from infinity. If $\beta=0$, it does not, which can actually be seen easily
since in the case $r(.)=c$,  $X_t\geq (x_0+bt)/2^{N_t}$, where $N_t$ is a Poisson Process with rate $c$ and 
the right hand side goes to $\infty$ as $x_0\rightarrow \infty$ for any $t\geq 0$.
\begin{proof}  Here $E=[0,\infty)$ and we consider  $$F(x)= \log (1+ \log(1+x))$$ on $(a', \infty)$ where $a' \in (-1,0)$ is chosen such that $\log(1+a')>-1$. 
Then   
$$F'(x)=\frac{1}{(1+x)(1+\log(1+x))}>0.$$
% and   Assumption \ref{assumedim1} is checked (one can take $a=a'/2$). 
Moreover
 $$F(x/2)-F(x)=
 \log\left(1-\varepsilon(x)\right),$$
where
$$\varepsilon(x)=1-\frac{1+\log(1+x/2)}{1+\log(1+x)}=\frac{ \log(2)+\mathcal O(1/(1+x))}{1+\log (1+x)}.$$ 
We consider now 
$$r(x)=c\log(1+x)^{\beta}$$
with $c>0$ and $ \beta \in [0,2]$. We get
$$b_{F}(x)=b+c \log(1+x)^{\beta} (1+x)(1+\log(1+x))\log\left(1-\varepsilon(x)\right) \sim -c\log(2) x\log(x)^{\beta} $$
as $ x\rightarrow \infty$. Thus, Assumptions  \ref{assumedim1} is satisfied for $a'$ and $a$ large  enough. Moreover
  $$\int_{\infty}^. \frac{1}{b_F(x)}dx <+\infty \  \text{ if and only if }  \ \beta>1.$$ 
  We observe that when $\beta\leq 1$, $bF'+h_F$ is bounded. Adding that
 %$h_F'(x)<0  $and
   $h_F'(x)=c\beta (x+1)^{-1}\log(1+x)^{\beta-1}(F(x/2)-F(x))+ c\log(1+x)^{\beta}(F'(x/2)/2-F'(x))$, we get
   $(bF'+h_F)'(x)<0$ for $x$ large enough when $\beta > 1$. 
Thus for any $\beta\geq 0$, $\psi_F=(bF'+h_F)\circ F^{-1}$ is $(0,\alpha)$ non-expansive on $(F(a),\infty)$, for some $\alpha>0$ and $a$ large enough.  
Finally
$$V_{F}(x)=c\log(1+x)^{\beta}  \log(1-\varepsilon(x))^2 \sim c\log(x)^{\beta-2}$$
as $x\rightarrow \infty$ and $V_{F}$ is bounded for $\beta\leq 2$. 
So   Assumptions  \ref{assumedim1}  and  \ref{controledim1} are satisfied for $a'$ and $a$ large enough and we can apply 
Lemma \ref{newlem}. 
%for $F=\log \circ \log$ and $\beta >0$.
We get for any $x_0 \geq 0$ and $T> 0$,
\be
\label{coo}
\mathbb P_{x_0}\left( \sup_{t\leq T}\vert F(X_t) -F(\phi_F(x_0,t)) \vert \geq A \right) \leq C(A,T), 
\ee
for $A$ large enough, where $C(A,T)  \rightarrow 0$ as $T\rightarrow 0$ and $C(A,T)\leq C.T. \sup_{x \geq 0} V_F(x)/A^2$. \\

We can now prove $(i)$ and let $\beta>1$. There exists $\widetilde{c}>0$ such that $$\widetilde{r}(x)=\widetilde{c}\log(1+x)^{\beta \wedge 2}$$
 satisfies
for any $x\geq 1$ and $y \in [x,2x]$, $r(x)\geq \widetilde{r}(y)$. 
By a  coupling argument, we can construct 
 a TCP associated with the rate of jumps $\widetilde{r}$ such that 
 $X_t\leq \widetilde{X}_t$ a.s. for $t\in [0, \inf\{ s\geq 0 : X_s\leq 1\})$.  
Then  $\phi_F(x_0,t)\leq\phi_F(\infty,t)<\infty$  since $\beta >1$ ensures that  the dynamical system comes down  from infinity. Letting $T\rightarrow 0$ in (\ref{coo}) yields $(i)$. \\
To prove $(ii)$, we use similarly  the coupling   $X_t\geq \widetilde{X}_t$ with $\widetilde{r}(x)=\widetilde{c}\log(1+x)^{\beta}$ and $\beta\leq 1$ and let now $A\rightarrow \infty$ in (\ref{coo}).
This ends up the proof since $V_F$ bounded ensures that $C(A,T)\rightarrow 0$. 
\end{proof}

\subsubsection{Logistic Feller diffusions \cite{Cattiaux2009} and perspectives}
The 
coming down from infinity of diffusions of the form
$$dZ_t=\sqrt{\gamma Z_t}dB_t+h(Z_t)dt$$
has been studied in  \cite{Cattiaux2009} and is linked to the uniqueness of the quasistationary distribution (Theorem 7.3).
Writing $X_t=2\sqrt{Z_t/\gamma}$, it becomes
$$dX_t=dB_t-q(X_t)dt,$$
where $q(x)=x^{-1}(1/2-2h(\gamma x^2/4)/\gamma)$.  Under some assumptions (see Remark 7.4 in \cite{Cattiaux2009}), the coming down from infinity is indeed
equivalent to 
$$\int^{\infty} \frac{1}{q(x)} dx <\infty,$$
which can be compared to our criterion in Theorem \ref{dim1descente}.
Several extensions and new results could be obtained  using the results of this section.
In particular one may be interested to mix a diffusion part for competition, negative jumps due to coalescence and branching events.
In that vein, let us  mention  \cite{LP}.
This is one motivation to take into account the compensated Poisson measure in the definition of the process $X$, so that L\'evy processes and CSBP may be considered in general. It is  left for future stimulating works. Let us here simply mention that a class of particular interest is given by  the logistic Feller diffusion :
$$dZ_t=\sqrt{\gamma Z_t}dB_t+(\tau Z_t-aZ_t^2)dt.$$
The next part is determining the speed of coming down from infinity of this diffusion. This part actually deals more generally  with the two dimensional version of this diffusion, where non-expansivity and the behavior of the dynamical system are more delicate.
\newpage

%\marginpar{Vi Le and Etienne Pardoux II ?}

\section{Uniform estimates for two-dimensional competitive Lotka-Volterra processes}
\label{LV}We consider  the historical  Lotka-Volterra competitive model for two species. It is  given by 
the  unique solution $x_t=(x^{(1)}_t, x^{(2)}_t)$ of the following ODE on $[0,\infty)$ starting from $x_0=(x_0^{(1)},x_0^{(2)})$: 
\be
(x_t^{(1)})'&=& x_t^{(1)}(\tau_1- ax_t^{(1)}- c x_t^{(2)}) \nonumber \\
(x_t^{(2)})'&=& x_t^{(2)}(\tau_2- bx_t^{(2)}- d x_t^{(1)}), \label{dyncompet} 
\ee
where    $a,b,c,d \geq 0$. The associated  flow is denoted by $\phi$ : 
$$\phi : [0,\infty)^2\times [0,\infty)\rightarrow [0,\infty)^2, \qquad \phi(x_0,t)=x_t=(x^{(1)}_t, x^{(2)}_t).$$
The coefficients $a$ and $b$ are the intraspecific competition rates and $c,d$ are the interspecific competition rates.
We assume that
$$a,b,c,d >0$$
or $a,b>0$ and $c=d=0$, so that our results cover the (simpler) case of one-single competitive (logistic)  model.
It is well known \cite{BM, Kurtz2}  that this deterministic model is the large population approximation of individual-based model, namely birth and death processes with logistic competition, see also Section \ref{descente3}. Moreover and more generally, when births and deaths are accelerated, these individual-based models weakly converge
to 
 the
unique strong solution of the following SDE:
\begin{eqnarray}
X_t^{(1)}&=& x_0^{(1)}+\int_0^t X_t^{(1)}(\tau_1- aX_s^{(1)}- c X_s^{(2)})ds +\int_0^t \sigma_1 \sqrt{X_s^{(1)}}dB_s^{(1)} \nonumber\\
X_t^{(2)}&=&x_0^{(2)}+\int_0^t X_t^{(2)}(\tau_2- bX_s^{(2)}- d X_s^{(1)})ds +\int_0^t \sigma_2\sqrt{X_s^{(2)}}dB_s^{(2)}, \label{EDSLv}
\end{eqnarray}
for $t\geq 0$, where $B$ is two dimensional  Brownian motion. 
 This is the classical  (Lotka-Volterra) diffusion  for two competitive species, see e.g. \cite{CattiauxMeleard} for related issues on quasi-stationary  distributions.

 In this section,  we compare  the stochastic Lotka-Volterra competitive processes to the deterministic flow $\phi$  for two new regimes
allowing to capture the behavior of the process for large values.
These results rely on the statements of Section $\ref{EDS}$ which are applied
to a well chosen finite subfamily of transformations among
\be
\label{defFbg}
F_{\beta,\gamma}(x)=\begin{pmatrix} x_1^{\beta} \\ \gamma x_2^{\beta} \end{pmatrix}, \qquad  x\in (0,\infty)^2, \ \ \beta \in (0,1], \ \ \gamma>0,
\ee
using the adjunction procedure. Moreover  Poincar\'e's compactification technics for flows is used to describe and control the
 coming down  from infinity. 
  \\

First,  in Section \ref{descente2},  se study the small time behavior of the diffusion  $X=(X^{(1)},X^{(2)})$ starting from large values. 
We  compare
the diffusion $X$  to the flow $\phi(x_0,t)$  for a suitable distance which captures the fluctuations of the diffusion at infinity. %defined for  $x_0 \in [0,\infty)^2$  and $t\in [0,\infty)$.
We then derive the way  the process $X$ comes down from infinity, i.e.  its direction and its speed.   Second, in Section \ref{descente3},  we  prove that usual
scaling limits of competitive birth and death processes  (see (\ref{BDN}) for a definition)  hold uniformly with respect to the initial values,  for a suitable distance and relevant set of  parameters.  \\
  These results give answers  to two issues which have motivated this work : first, how classical competitive stochastic models regulate large populations (see in particular forthcoming Corollary  \ref{descenteLVb}); second, can we extend  individual based-models approximations of Lotka-Volterra dynamical system to arbitrarily large initial values and if yes, when and for which distance. They are the key for forthcoming works on coexistence of competitive species in varying environment.
 We believe  that the technics developed here   allow to study similarly the coming down from infinity
of these  competitive birth and death processes and other multi-dimensional stochastic processes.

\subsection{Uniform short time estimates for competitive Feller diffusions}
\label{descente2}

We consider the domain
$$\mathcal D_{\alpha} = (\alpha, \infty)^2$$ %\{ x=(x_1,x_2) \in (0,\infty)^2 : x_1 > \alpha, \ x_2> \alpha\}$$
and the distance $d_{\beta}$ on $[0,\infty)^2$ defined for $\beta >0$ by
\be
\label{distanceLV}
d_{\beta}(x,y)= \sqrt{ \vert x_1^{\beta}-y_1^{\beta}\vert^2 + \vert x_2^{\beta}-y_2^{\beta}\vert^2}=\parallel F_{\beta,1}(x)-F_{\beta,1}(y) \parallel_2. \ee
We recall that $a,b,c,d>0$ or ($a=b>0$ and $c=d=0)$ and we define
\be
\label{defdom}
T_{D}(x_0)=\inf \{ t \geq 0 : \phi(x_0,t) \not\in D\}
\ee the first time when the flow
$\phi$  starting from $x_0$ exits $D$.
\begin{thm}\label{descenteLV} For any  $\beta \in (0,1)$,  $\alpha>0$
and   $\varepsilon>0$, 
$$\lim_{T\rightarrow 0} \sup_{x_0 \in \mathcal D_{\alpha}} \mathbb P_{x_0}\left(\sup_{t\leq T\wedge T_{\mathcal{D}_{\alpha}}(x_0)} d_{\beta}(X_t,\phi(x_0,t))
\geq \varepsilon \right)=0.$$
\end{thm}
\noindent This yields a control of the stochastic process $X$ defined in (\ref{EDSLv})  by the dynamical system for large initial values and times small enough.
We are not expecting that this control hold outside $\mathcal D_{\alpha}$. Indeed,   
the next result  shows  that the process and the dynamical system  coming from infinity have a different behavior when they come close to the 
boundary of $(0,\infty)^2$. It is naturally due to the diffusion component and the absorption at  the boundary. \\
The proof  can not be achieved for  $\beta=1$ since then the associated quadratic variations are not integrable at
 time $0$. 
Heuristically, $\sqrt{Z_t}dB_t$ is of order $\sqrt{1/t}dB_t$ for small times.  This latter  does not become small when $t\rightarrow 0$ and the fluctuations do not vanish for $d_1$ in short time. \\
% qui n'est pas rendu arbitrairement petit en temps petit. \\

%Il faut maintenant conna\^itre le syst\`eme dynamique $(x^1_t,x^2_t)$ au voisinage de l'infini pour comprendre le processus. On va utiliser pour cela les compactifications et prolongement de flot de Poincar\'e. \\
We denote
 $\widehat{(x,y)} \in (-\pi,\pi]$ the oriented  angle in the trigonometric sense between two non-zero vectors of $\R^2$ and if $ab\ne cd$, we write
 \be
 \label{defxinf}
 x_{\infty}= \frac{1}{ ab-cd} ( b-c , a-d).
 \ee
The following  classification yields   the way the diffusion comes down from infinity.
%\marginpar{Faire  des dessins}
\begin{cor} \label{descenteLVb}
We assume that $\sigma_1>0, \sigma_2>0$ and let $x_0 \in (0,\infty)^2$.\\
(i) If  $a>d$ and $b>c$, then for any   $\eta \in (0,1)$ and $\varepsilon>0$,
$$\lim_{T\rightarrow 0} \limsup_{r\rightarrow \infty}  \mathbb P_{rx_0} \left(  \sup_{\eta T \leq t \leq T} \parallel tX_t -x_{\infty} \parallel_2  \geq \varepsilon\right)=0,$$
If furthermore $x_0$ is collinear to $x_{\infty}$, the previous limit holds also for $\eta=0$.\\
 (ii) If $a<d$ and $b<c$ and  $\widehat{( x_{\infty},x_0)}\ne 0$, then for any  $T>0$,
$$ \lim_{r\rightarrow \infty}  \mathbb P_{rx_0} \left( \inf \{ t \geq 0 : X_t^{(i)}=0 \} \leq T \right)=1,$$
where $i=1$ when  $\widehat{({x_\infty},x_{0})}\in (0,\pi/2]$  and  $i=2$ when $\widehat{({x_\infty},x_{0})}\in [-\pi/2,0)$. \\
(iii) If  ($a\leq d$ and $b>c$) or if  ($a< d$ and $b\geq c$), then for any  $T>0$,
$$ \lim_{r\rightarrow \infty}  \mathbb P_{rx_0} \left( \inf \{ t \geq 0 : X_t^{(2)}=0 \} \leq T \right)=1.$$
(iv) If $a=d$ and $b=c$, then 
$$\lim_{T\rightarrow 0} \limsup_{r\rightarrow \infty}  \mathbb P_{rx_0} \left(  \sup_{ t \leq T} \parallel tX_t -(ax_0^{(1)}+bx_0^{(2)})^{-1}x_0 \parallel_2  \geq \varepsilon\right)=0.$$
\end{cor}
\noindent In the first case $(i)$, the diffusion $X$ and the dynamical system  $x$ come down from infinity
in a single direction 
 $x_{\infty}$, with speed proportional to $1/t$. They  only need 
  a short time at the beginning of the trajectory 
to find this direction. This short time quantified by $\eta$ here could be made arbitrarily  small when $x_0$ becomes large. Let us also observe that the one-dimensional logistic Feller diffusion 
$X_t$ is given by $X_t^{(1)}$ for $c=d=0$. Thus, taking $x_0$ collinear to $x_{\infty}$,  $(i)$ yields the speed of coming down from infinity of  one-dimensional logistic Feller diffusions: 
\be
\label{Fellerun}
\lim_{T\rightarrow 0} \lim_{r\rightarrow \infty}  \mathbb P_{r} \left( 
 \sup_{ t \leq T} \vert atX_t -1\vert \geq \varepsilon\right) =0.
 \ee
In the second case $(ii)$, the direction taken by the dynamical system and the process depends
on the initial direction. The dynamical system then goes to the boundary of $(0,\infty)^2$ without reaching it.  But the fluctuations of the process make it reach the boundary and one species becomes extinct. When the process starts in the direction of $x_{\infty}$, additional work would be required to describe its behavior, linked to the behavior of the dynamical system around the associated unstable variety coming from infinity.\\
%\marginpar{renvoyer au dessin}
In the third case $(iii)$, the dynamical system $\phi$ goes to the boundary $(0,\infty)\times \{0\}$ when coming down from infinity, wherever it comes from. Then, as above, the diffusion $X^{(2)}$ hits $0$. Let us note that, even in that case,  the dynamical system may then
go to a coexistence fixed point or  to a fixed point where only the species $2$ survives.  This latter event  occurs when
$$\tau_2/b> \tau_1/c, \qquad \tau_2/d>\tau_1/a $$
and is illustrated in the third simulation below.
Obviously, the symmetric situation happens when ($b\leq c$ and
 $d<a$) or ($b< c$ and
 $d\leq a$).  
 Moreover, in  cases $(ii-iii)$, the proof  tells us that when $X$ hits the axis,  it is not close from $(0,0)$. Then it becomes a   one-dimensional  Feller logistic diffusion whose 
coming down infinity has been   given above, see $(\ref{Fellerun})$.\\
 In the case $(iv)$, the process comes down from infinity in the direction of its initial value, at speed $a/t$.\\
 Finally,  let us note that this raises several questions on the characterization of a process starting from infinity in dimension $2$. In particular, informally, the process coming down from infinity in a direction $x_0$ which is not $x_{\infty}$ has a discontinuity at time $0$ in the cases $(i-ii-iii)$. \\ \\
  \emph{ Simulations. We consider two   large initial values $x_0$ such that $\parallel x_0\parallel_1=10^5$.
We plot  the dynamical system (in black line) and two realizations of the diffusion (in red line) starting from  these two initial values. In each simulation, $\tau_1=1, \tau_2=4$ and the solutions of the dynamical system converge to the fixed point where only the second species survives.  
The coefficient diffusion terms are $\sigma_1=\sigma_2=10$. We plot here $G(x_t)$ and $G(X_t)$, where
$$G(x,y)=(X,Y)=(\log(1+x), \log(1+y))$$
to zoom on the behavior of the process when coming close to  the axes. The four regimes $(i-ii-iii-iv)$ of the corollary above, which describe the coming down from infinity,  are  successively illustrated. 
One can also compare with the pictures of Section \ref{comp} describing the flow.}
\begin{flushleft}
\begin{multicols}{2}
\includegraphics[scale=0.4]{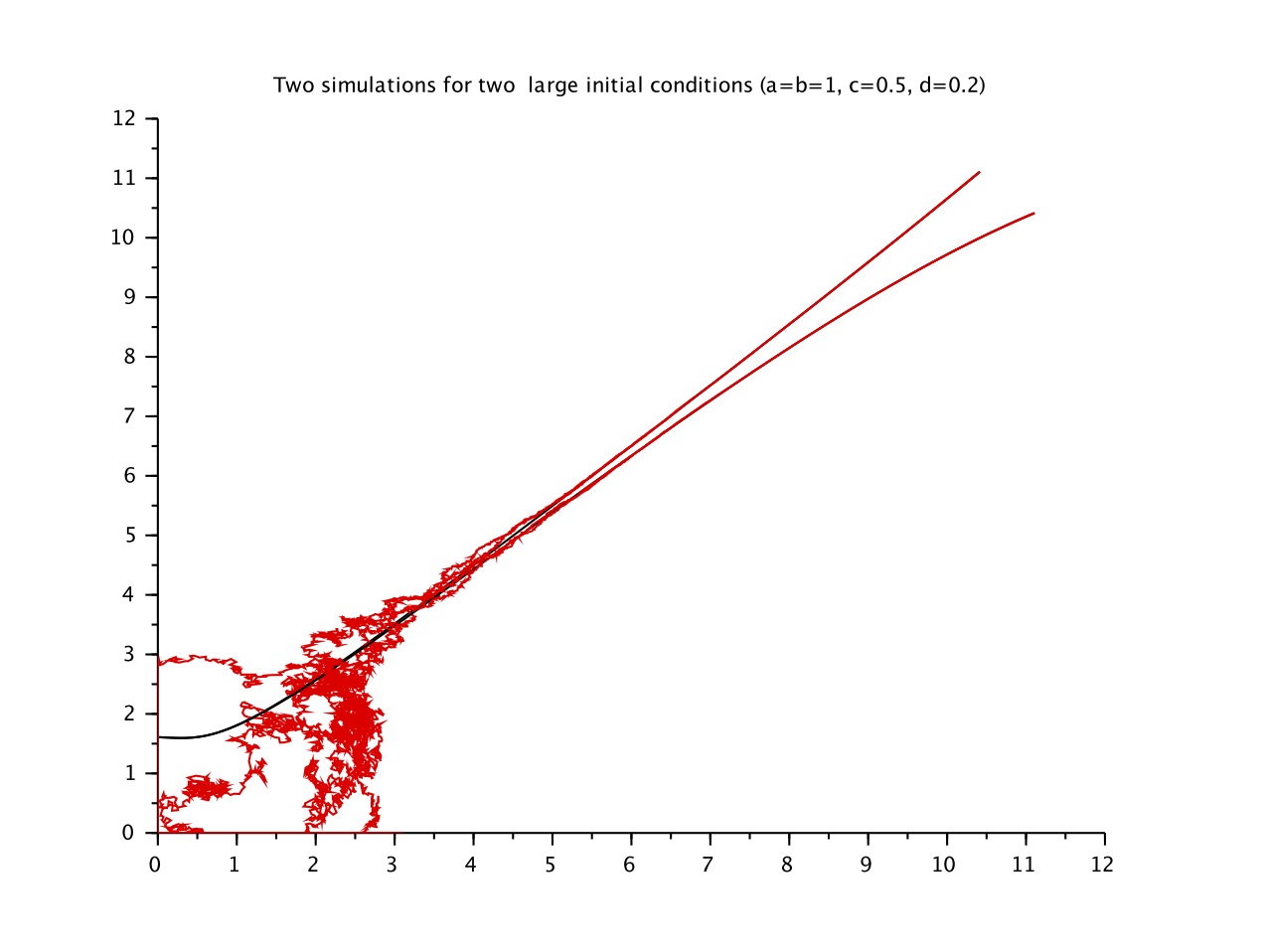}  

\includegraphics[scale=0.4]{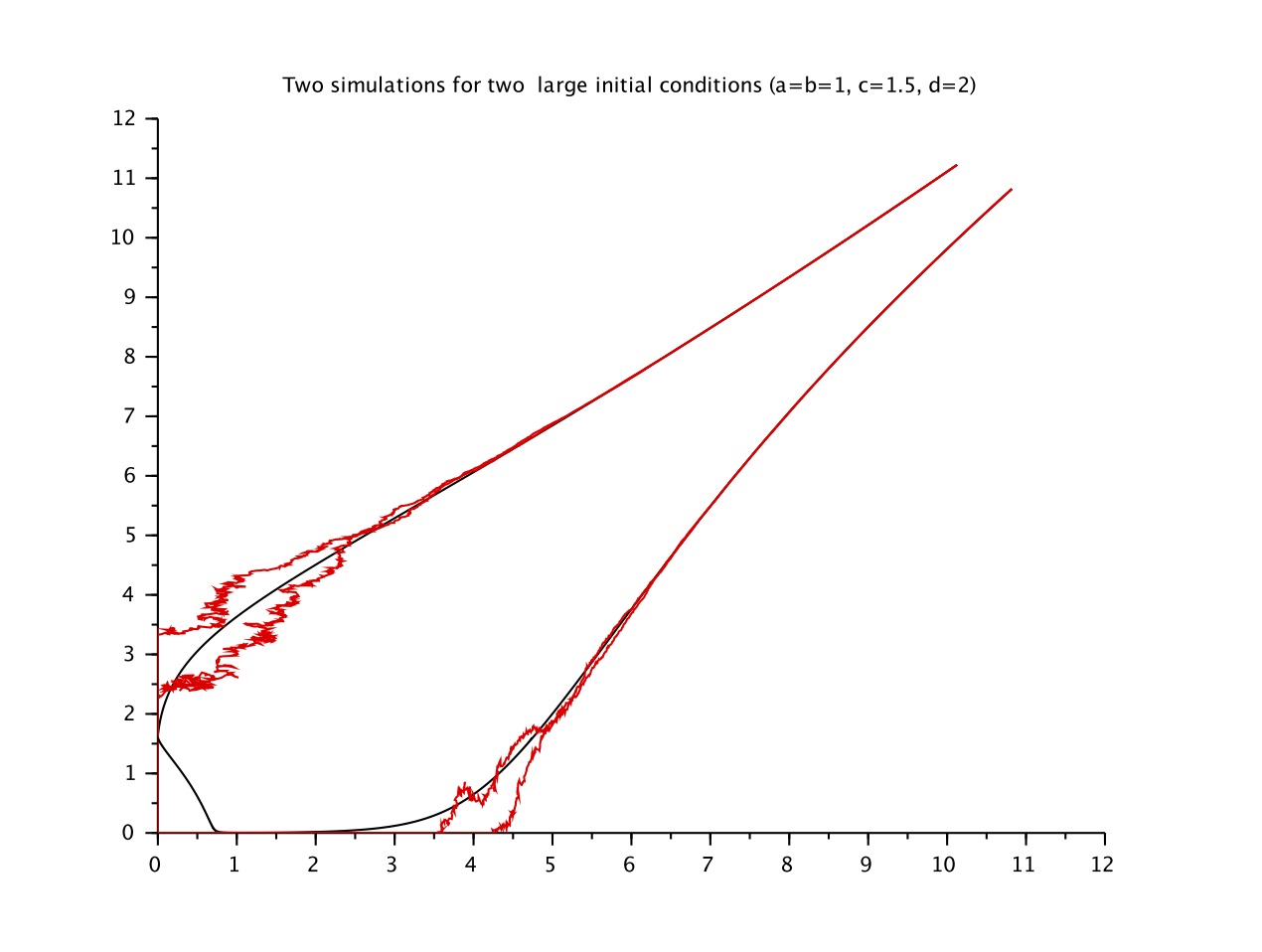}
\end{multicols}
\end{flushleft}
\begin{flushleft}
\begin{multicols}{2}
\includegraphics[scale=0.4]{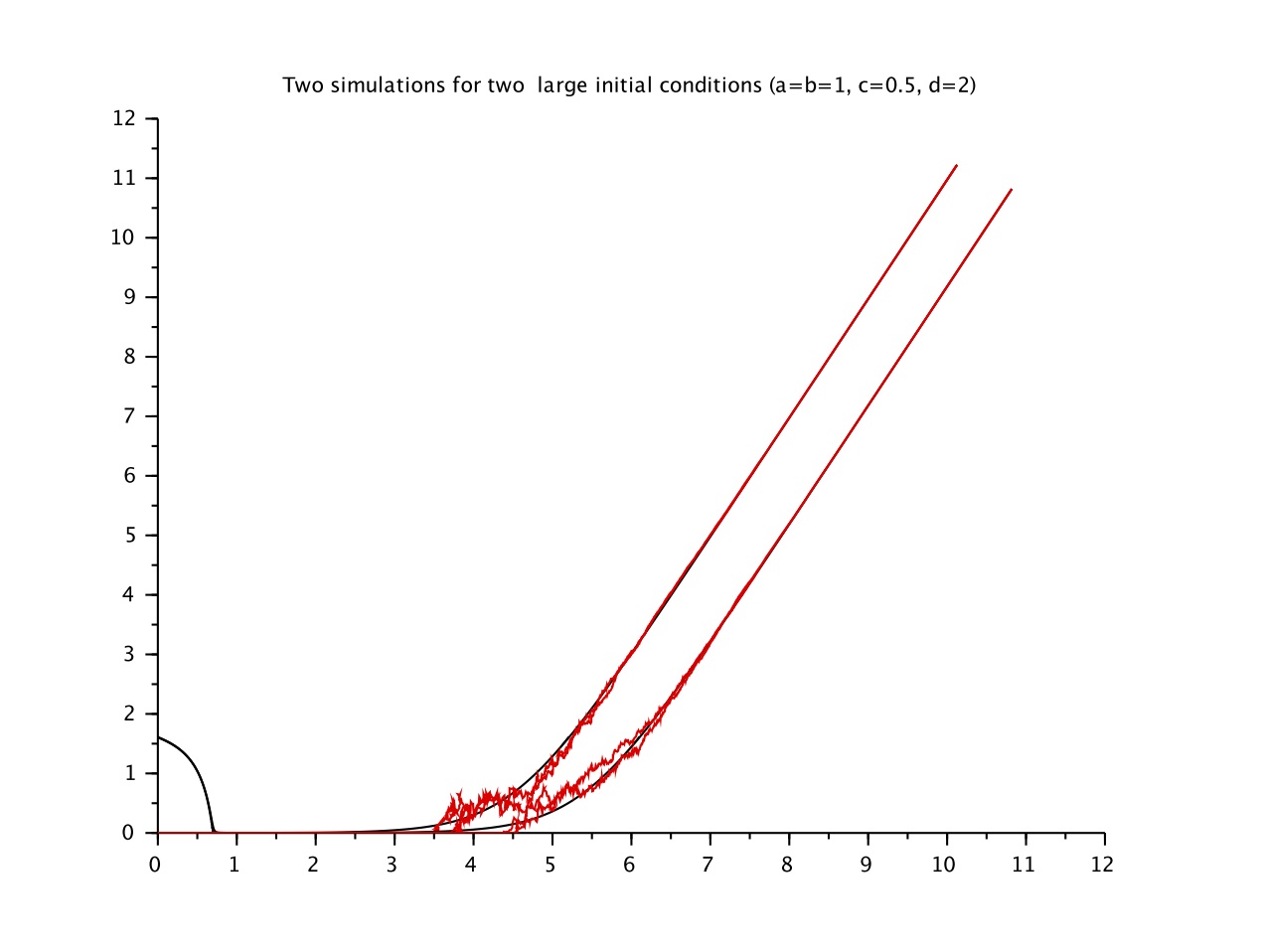} 

\includegraphics[scale=0.4]{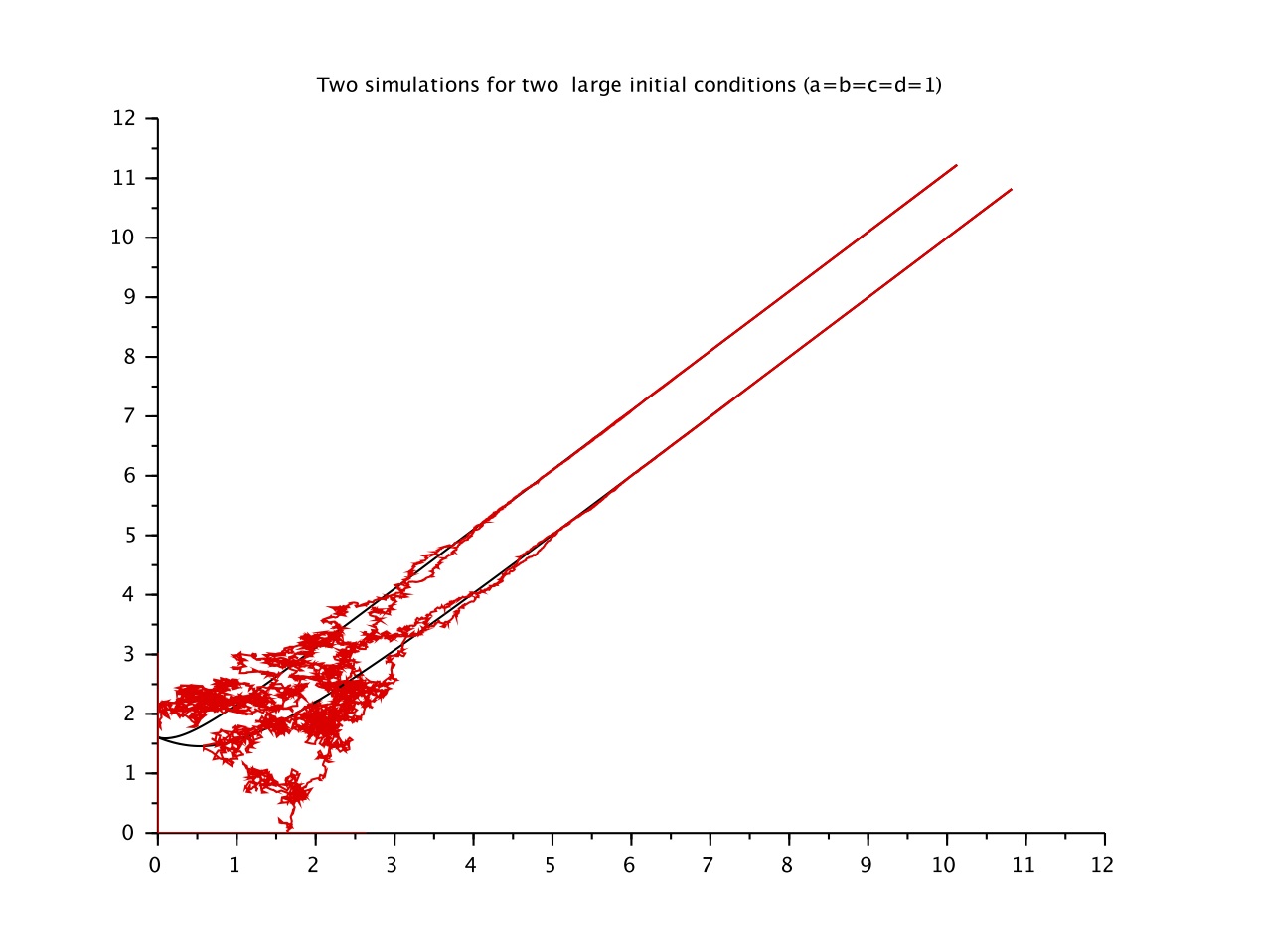}
\end{multicols}
\end{flushleft}

\newpage
%\marginpar{Ajouter simus, \\ Verifier scaling}

\subsection{Uniform scaling limits of competitive birth and death processes} 
\label{descente3}
Let us deal finally with  competitive birth and death processes and consider their scaling limits to
the Lotka-Volterra dynamical system $\phi$ given by $(\ref{dyncompet})$.
These scaling limits are usual 
%\marginpar{Et son bouquin populations ?}
approximations  in large populations of dynamical system by individual-based model, see e.g. \cite{BM, Kurtz2}.
We provide here estimates which  are uniform with respect to the initial values in a cone in the interior of $(0,\infty)^2$, for a distance capturing the large fluctuations of the process at infinity. 
% \marginpar{Kurtz 71 ? ici scaling unifor en fond initiale. Mettre un $-1$ ?}
The birth and death rates  of the two species   are given  for population sizes   $n_1, n_2 \geq 0$ and  $K\geq 1$ by
$$\lambda^K_1(n_1,n_2)=\lambda_1n_1, \qquad \mu^K_1(n_1,n_2)=\mu_1n_1+an_1.\frac{n_1}{K}+ cn_1.\frac{n_2}{K}$$
for the first species and by 
$$\lambda^K_2(n_1,n_2)=\lambda_2n_2, \qquad \mu^K_2(n_1,n_2)=\mu_2n_2+bn_2.\frac{n_2}{K}+ dn_2.\frac{n_1}{K}$$
for the second species. We assume that $$\lambda_1-\mu_1=\tau_1, \qquad  \lambda_2-\mu_2=\tau_2.$$
% and note
%$$b(x_1,x_2)=(\tau_1x_1 -ax_1^2-cx_1x_2,\tau_2 x_2 -bx_2^2-dx_1x_2)), \quad 
%\tau_1= \lambda_1-\mu_1, \quad \tau_2= (\lambda_2-\mu_2).$$ 
Dividing the number of individuals by $K$, the normalized population size $X^K$ satisfies
\be
\label{BDN}
X_t^K=x_0+\int_0^t \int_{[0,\infty)} H^K(X_{s-},z) N(ds,dz), 
\ee
where  writing $\tau^{K}_1= \lambda_1^K+\mu_1^K$ for convenience,
\be
\label{defHK}
H^{K}(x,z)=\frac{1}{K}\begin{pmatrix}
\1{z\leq \lambda_1^K(Kx)}-\1{\lambda_1^K(Kx)\leq z \leq  \tau^{K}_1(Kx)} \\

\1{ 0\leq z-\tau^{K}_1(Kx) \leq \lambda_2^K(Kx)}-\1{\lambda_2^K(Kx)\leq z -\tau^{K}_1(Kx) \leq \lambda_2^K(Kx)+\mu_2^K(Kx)} 
\end{pmatrix}.
\ee
%$$M_t^K=\frac{1}{K}\int_0^t \int_0^{\infty}\left( \1{z\leq \lambda^K(KX_{s-}^K)}-\1{\lambda^K(KX_{t-}^K)\leq z \leq \lambda^K(KX_{s-}^K)+\mu^K(KX_{t-}^K)} \right) Ê\widetilde{N}(ds,du).$$
and $N$ is a PPM   on $[0,\infty)\times[0,\infty)$ with intensity
$dsdz$.
We set
$$\mathfrak {D}_{\alpha}=\{ (x_1,x_2) \in (\alpha, \infty)^2 : \   x_1\geq \alpha x_2, \  x_2 \geq \alpha x_1 \},$$
which is required both for the control of the flow and of the fluctuations.  We only consider here the case 
\be
\label{condparam}
 (b>c>0 \text{ and } a>d>0) \quad \text{or} \quad (a,b>0 \text{ and }  c=d=0) \quad \text{or} \quad  (a=d>0 \  \text{and}  \ b=c>0) 
 \ee
%or $a=d>0$ and $b=c>0$, 
since we know from  the previous Corollary that it gives 
%\marginpar{A clarifier a t on vraiment besoin de fixed poit ˆ lint : je crois pas}
the cases when the flow does not go instantaneously  to the boundary of $(0,\infty)^2$ in short time when 
coming from infinity.  Thus the flow does not exit from $\mathfrak {D}_{\alpha}$ instantaneously, which would prevent the uniformity in the convergence below. This corresponds to   the cases $x_{\ell}=x_{\infty}$ and $x_{\ell}=\widehat{x_0}$ in the
 forthcoming Lemma \ref{Poinc} $(ii)$ and Figure $1$.
% , while the first part ensures that 
%the dynamical system goes to an interior fixed point.
%the other cases the dynamical system goes to one boundary starting from infinity so
%the exit time of $\mathfrak D_{\alpha}$ becomes small when the initial condition becomes large.
\begin{thm}  \label{scaling} For any $T>0$, $\beta\in (0,1/2)$ and $\alpha, \varepsilon>0$, there exists $C>0$ such that for any $K\geq 0$,
$$\sup_{x_0 \in \mathfrak D_{\alpha}} \p_{x_0}\left(\sup_{t\leq T}d_{\beta} (X_t^K, \phi(x_0,t)) \geq \varepsilon\right) \leq \frac{C}{K}.$$
\end{thm}
%\marginpar{A revoir}
%Again, we know the qualitative behavior of $ T_{\mathcal D_{\alpha}}(x_0)$ when  $x_0$ goes to infinity thanks to Lemma \ref{Poinc}. 
\noindent The proof, which is given below, rely on $(L, \alpha_K)$ non-expansivity of the   flow associated with $X^K$,  with $\alpha_K\rightarrow 0$. Additional work should
 allow to make  $T$  go to infinity when $K$ goes to infinity. The critical power $\beta=1/2$ is reminiscent from results obtained for one dimensional logistic birth and death process in Proposition \ref{regvar} in  Section \ref{secBD}.
% $T<<(\log K)^{\kappa}$ for some $\kappa<1$.  seems to be required to get sharper estimates.

\subsection{Non-expansivity of the flow and Poincar\'e's compactification}
\label{comp}
 The proofs of the three  previous statements  of this section   rely  on the following lemmas. \\ The first one 
provides the domains where the transformation $F_{\beta, \gamma}$ yields  a  non-expansive vector field. 
It is achieved by determining the spectrum of the symmetrized operator of the Jacobian matrix of $\psi_{F_{\beta,\gamma}}$ and provide a covering of the state space.
This is the key ingredient to use the results of Section \ref{EDS} for the study of the coming  down from infinity of  Lotka-Volterra diffusions (Theorem  \ref{descenteLV})  and  the proof of the scaling limits of  birth and death processes (Theorem \ref{scaling}). \\
We also need to control the flow $\phi$  when it comes down from infinity. 
The  lemmas of Section \ref{sectionPoinca} describe   the   dynamics  of the flow and provide some additional   results useful for the proofs.
These  proofs   rely on the extension of the flow on the boundary at infinity, using    Poincar\'e's technics, and can be achieved for more general models. \\
Finally, we combine these results in  Sections \ref{scaled} and  \ref{adjsec} and decompose the whole trajectory of the flow in a finite number of time intervals during which 
it belongs  to a domain where  non-expansivity holds for one of  the transformation $F_{\beta,\gamma}$.  \\

%We first prove these two lemmas, then   the Theorem  \ref{descenteLV} and  the Corollary which relies 
%both on this Theorem and the second lemma. Finally we prove Theorem \ref{scaling}.
As one can see on spectral computations below,  non-expansivity holds in   a  cone. We recall that a cone is a subset  $\mathcal C$ of $\mathbb R^2$ such that for all $x\in \mathcal C$ and  $\lambda>0$, $\lambda x\in \mathcal C$.  We use   the convex components of open cones, which are open convex cones. For $S$ a subset of $\R^2$, we denote by $\overline{S}$ the closure of $S$.  \\
%, which relies on the first lemma and finally the Corollary, which is a consequence of the Theorem and the second lemma.
Recalling notations of Section \ref{EDS}, we have   here $E=[0,\infty)^2$, $\d=2$ and 
$$
\psi_F=(J_F b)\circ F^{-1},$$
where 
\be
\label{defb}
b(x)=b(x_1,x_2)=\begin{pmatrix} \tau_1x_1 -ax_1^2-cx_1x_2  \\ \tau_2 x_2 -bx_2^2-dx_1x_2 \end{pmatrix}.
\ee

\subsubsection{Non-expansitivity in cones}
\label{nonexpcones}
Let us write $\overline{\tau}=\max(\tau_1,\tau_2)$ and
$$q_{\beta}=4ab(1+\beta)^2+4(\beta^2-1)cd$$
for convenience and consider
 the open cones of $(0,\infty)^2$ defined by  \be
 \label{defDbeta}
D_{\beta,\gamma}=\big\{ x \in (0,\infty)^2 : 4\beta(1+\beta)(adx_1^2+bcx_2^2)+q_{\beta} x_1x_2 
%\\ &&\qquad \qquad \qquad \qquad \qquad \qquad \qquad \qquad \qquad
 -\left(c\gamma^{-1}x_1^{\beta}x_2^{1-\beta}-d\gamma x_1^{1-\beta}x_2^{\beta}\right)^2 > 0\big\}.
\ee
 %\marginpar{baptiser $4ab(1+\beta)^2+4(\beta^2-1)cd=q_{\beta}$ ?}  
\begin{lem}\label{transformation} Let $\beta \in (0,1]$ and $\gamma>0$.  \\
The vector field  $\psi_{F_{\beta, \gamma}}$ is $\overline{\tau}$ non-expansive on
each convex component of the open cone $F_{\beta,\gamma}(D_{\beta,\gamma})$. 
 \end{lem}
\noindent In the particular case $a,b>0$ and $c=d=0$, for any $\beta \in (0,1]$ and $\gamma >0$, 
$D_{\beta,\gamma}=(0,\infty)^2$.
But this fact does hold in  general. We need the transformations $F_{\beta,\gamma}$ for well chosen values 
of $\gamma$  to get the non-expansivity property of the flow on unbounded domains.  Let us also note that  $(0,\infty)^2$ is not coverable by a single domain  of the form $D_{\beta,\gamma}$ in general and the adjunction procedure of Section \ref{adjunction}  will be needed. \\

\begin{proof}
 We write  for $y=(y_1,y_2)\in [0,\infty)^2$,
 \be
\label{deccc}
\psi_{F_{\beta, \gamma}}(y)=\psi_{1}(y) + \psi_{2,\beta, \gamma}(y),
\ee
where
$$\psi_{1}(y)
=\begin{pmatrix}\beta\tau_1y_1 \\   \beta\tau_2y_2\end{pmatrix}, \qquad
\psi_{2,\beta, \gamma}(y)=-\begin{pmatrix}\beta y_1\left(ay_1^{1/\beta}   + c\gamma^{-1/\beta}y_2^{1/\beta} \right)\\
  \beta y_2\left(b\gamma^{-1/\beta}y_2^{1/\beta}  +dy_1^{1/\beta}\right)\end{pmatrix}.$$
First, $\psi_1$ is Lipschitz on $[0,\infty)^2$ with constant $\overline{\tau}$ since $\beta\ \in (0, 1]$. Moreover, writing 
$A_{\beta,\gamma}(x)=J_{\psi_{{2,\beta, \gamma}}}(F_{\beta, \gamma}(x))$,
we have for any $x\in [0,\infty)^2$,
$$A_{\beta,\gamma}(x)+A_{\beta,\gamma}^*(x)=-\begin{pmatrix}
2a(1+\beta)x_1 +2c\beta x_2  &  c\gamma^{-1}x_1^{\beta}x_2^{1-\beta}+d\gamma x_2^{\beta}x_1^{1-\beta} \\
 c\gamma^{-1}x_1^{\beta}x_2^{1-\beta}+ d\gamma x_2^{\beta}x_1^{1-\beta} &  2b(1+\beta)x_2+2d\beta x_1\end{pmatrix}. $$
This can be seen   using (\ref{matrice}) or by a direct computation.
 We consider now the trace and the determinant of this matrix  :
%\marginpar{choix $A_{\beta,\gamma}$ $D_{????}$ }
\be
\label{TD} 
T(x)=\text{Tr}\left(A_{\beta,\gamma}(x)+A_{\beta,\gamma}^*(x)\right), \qquad \Delta(x)= \text{det} \left(A_{\beta,\gamma}(x)+A_{\beta,\gamma}^*(x)\right).
\ee
 As $\beta >0$ and $x\in (0,\infty)^2$, $T(x)<0$, while 
%a trace est n\'egative et pour que le il suffit  que le d\'eterminant de $A_{\beta,\gamma}(x)+A_{\beta,\gamma}^*(x)$, qui est
 \be
 \label{valeurdet}
 \Delta (x)=(2a(1+\beta)x_1 +2c\beta x_2)(2b(1+\beta)x_2+2d\beta x_1)-\left(c\gamma^{-1}x_1^{\beta}x_2^{1-\beta}+ d\gamma x_2^{\beta}x_1^{1-\beta}\right)^2.
 \ee
It is positive when $x=(x_1,x_2)\in D_{\beta,\gamma}$ and  then the  spectrum
of  $A_{\beta,\gamma}(x)+A_{\beta,\gamma}^*(x)$ is included in $(-\infty,0]$.
%\marginpar{Clarifier}
Recalling table 1 in \cite{Aminzare} or the beginning of  Section \ref{main},  this ensures that    $\psi_{2,\beta, \gamma}$ is non-expansive on   the open convex components of  $F_{\beta, \gamma}(D_{\beta,\gamma})$.
Then  $\psi_{F_{\beta, \gamma}}$ is $\overline{\tau}$ non-expansive on
 the open   convex components of $F_{\beta, \gamma}(D_{\beta,\gamma})$.
Let us finally  observe that $ D_{\beta,\gamma}$ and thus $F_{\beta, \gamma} (D_{\beta,\gamma})$
are open cones,
 which ends up the proof of $(i)$. 
 \end{proof}

We define now
$$\mathcal C_{\eta, \beta, \gamma} =\left\{ x \in (0,\infty)^2  : x_1/x_2  \in (0, \eta) \cup (  x_{\beta, \gamma}-\eta, x_{\beta, \gamma}+\eta   ) \cup (1/\eta, \infty)\right\},$$
 writing $x_{\beta, \gamma} =(d\gamma^2/c)^{1/(2\beta-1)}$ when it is well defined. The next result ensures that these 
 domains provide a covering by cones for which non-expansivity hold.  The case $c=d=0$ is obvious and we focus on the general case.
 \begin{lem} \label{rec}  Assume that $a,b,c,d>0$.
Let $\gamma>0$, $\beta \in (0,1)-\{1/2\}$ such that  $q_{\beta}>0$. \\
 There exists $\eta>0$ and $A>0$ and $\mu>0$  such that 

(i) $\mathcal C_{\eta, \beta, \gamma} \subset D_{\beta,\gamma}$.

(ii) 
 for any $y,y'$ which belong both  to a same convex component
 of  the cone $F_{\beta,\gamma}(\mathcal C_{\eta, \beta, \gamma})$ and to the complementary set of  $B(0,A)$, then
\be
\label{equationsup}
(\psi_{F_{\beta, \gamma}}(y)-\psi_{F_{\beta, \gamma}}(y')){\boldsymbol .}(y-y') \leq -\mu.(\parallel y \parallel_2 \wedge \parallel y'\parallel_2) . \parallel y-y'\parallel_2^2.
\ee
\end{lem}
\begin{proof}
$(i)$ The  inclusion 
 $\left\{ x \in (0,\infty)^2  : x_1=x_{\gamma} x_2 \right\} \subset D_{\beta,\gamma}$
 comes from the fact that 
$$x_1=(d\gamma^2/c)^{1/(2\beta-1)}x_2 \quad \text{ implies  
that } \quad 
\left(c\gamma^{-1}x_1^{\beta}x_2^{1-\beta}-d\gamma x_1^{1-\beta}x_2^{\beta}\right)^2
=0$$
and the fact that $q_{\beta}>0$. The inclusion 
$%\left\{ x \in (0,\infty)^2  : x_2< \eta x_1 \right\} \cup 
 \left\{ x \in (0,\infty)^2  : x_1/ x_2 \in (0,\eta)\cup (1/\eta, \infty)    \right\}   \subset D_{\beta,\gamma}$
is obtained by bounding
$$\left(c\gamma^{-1}x_1^{\beta}x_2^{1-\beta}-d\gamma x_1^{1-\beta}x_2^{\beta}\right)^2
\leq \left(c\gamma^{-1}\eta^{1-\beta}+d\gamma\eta^{\beta}\right)^2x_1^2  $$
when $x_2\leq \eta x_1$. Indeed,   $a,d> 0$ and 
letting $\eta$ be small enough such that $4\beta(1+\beta)ad>(c\gamma^{-1}\eta^{1-\beta}+d\gamma\eta^{\beta})^2$ yields the result since $\beta\in (0,1)$.    \\

$(ii)$ Recalling notation $(\ref{TD})$, 
 for any $x\in [0,\infty)^2-\{(0,0)\}$, $T(x)<0$ and   the value of $\Delta(x)$  is given by $(\ref{valeurdet})$.
Let $x_0\ne 0$ such that $\Delta(x_0)>0$, then there exist $v_1,v_2>0$ and some open ball $\mathcal  V(x_0)$ centered in $x_0$, such that for any $x\in \mathcal V(x_0)$, we have  $-v_1\leq T(x)< 0$ and $\Delta(x)\geq v_2$. So  
 for any $\lambda >0$ and $x\in \mathcal V(x_0)$,
$$T(\lambda x)=\lambda T(x)\in [- \lambda v_1,0) , \qquad \Delta(\lambda x)=\lambda^2\Delta(x) \in [ \lambda^2 v_2, \infty).$$
Writing   $E$  the largest eigenvalue of $A_{\beta,\gamma}+A_{\beta,\gamma}^*$,   we have for any $x\in \mathcal V(x_0)$,
$$E(\lambda x) \leq 2\frac{\Delta(\lambda x)}{T(\lambda x)} \leq  -2\lambda \frac{v_2}{v_1}<0, $$
since $\Delta$ (resp. $T$) gives the product (resp. the sum) of the two eigenvalues. We obtain that there exists $\mu>0$ such that for 
any $x$ in the convex cone  $\mathcal C(x_0)$
generated by $\mathcal V(x_0)$, the spectrum of $A_{\beta,\gamma}(x)+A_{\beta,\gamma}^*(x)$ is included 
in $(-\infty, -2\mu \parallel x \parallel_2]$. 
 Recalling that $A_{\beta,\gamma}=J_{\psi_{{2,\beta, \gamma}}}\circ F_{\beta, \gamma}$ and $\beta\leq 1$,  there exists $\widetilde{\mu}$ such that
 the spectrum of $J_{\psi_{{2,\beta, \gamma}}}(y)+J_{\psi_{{2,\beta, \gamma}}}^*(y)$ is included in $(-\infty,-2\widetilde{\mu} \parallel y\parallel_2]$ for any $y\in F_{\beta,\gamma}(\mathcal C(x_0))$ such that
 $\parallel y\parallel_2 \geq 1$. Then
%\marginpar{A reprendre !!}
$$(\psi_{2,\beta, \gamma}(y)-\psi_{2,\beta, \gamma}(y')){\boldsymbol .}(y-y')\leq  -\widetilde{\mu} .(\parallel y \parallel_2 \wedge \parallel y' \parallel_2).\parallel y-y'\parallel_2^2,$$
 for any $y,y'$ in  a convex  set  containing   $F_{\beta,\gamma}(\mathcal C(x_0))\cap B(0,1)^c$, see again
 % \marginpar{check table}  
 Table $1$ in  \cite{Aminzare} for details. 
 Recalling now (\ref{deccc}) and that  $\psi_1$ is Lipschitz with constant $\overline{\tau}$,  there exists $A>0$ such that  
$$(\psi_{F_{\beta, \gamma}}(y)-\psi_{F_{\beta, \gamma}}(y')){\boldsymbol .}(y-y')\leq  -\frac{1}{2} \widetilde{\mu} .(\parallel y \parallel_2 \wedge \parallel y' \parallel_2).\parallel y-y'\parallel_2^2$$
%\marginpar{check!!!} 
for any $y,y' \in B(0,A)^c$  which belong to  convex  component of   $F_{\beta,\gamma} (\mathcal C(x_0))$.
We conclude by choosing $\eta >0$ such that $\mathcal C_{\eta,\beta,\gamma} \subset \cup_{x_0 \in \{ x_{\gamma},(0,1),(1,0)\}}\mathcal C(x_0)$.
\end{proof}

\subsubsection{Poincar\'e's compactification and coming down from infinity of the flow} 
\label{sectionPoinca}
To describe the coming down from infinity of the flow $\phi$, we use the following compactification $\mathcal K$
of $[0,\infty)^2$ :   
$$\mathcal K(x)= \mathcal K(x_1,x_2)=\left(\frac{x_1}{1+x_1+x_2}, \frac{x_2}{1+x_1+x_2}, \frac{1}{1+x_1+x_2}\right)=(y_1,y_2,y_3)$$
The application $\mathcal K$ is a bijection  from $[0,\infty)^2$ into the simplex $\mathcal S$ defined by
$$\mathcal S= \{ y \in  [0,1]^2\times(0,1] : y_1+y_2+y_3=1 \}
\subset \overline{\mathcal S}=\{ y \in [0,1]^3 : y_1+y_2+y_3=1\}.$$
We note $\partial S$ the outer boundary of $\mathcal S$ :
$$\partial S= \overline{\mathcal S}-\mathcal S=\{(y_1,1-y_1,0) :  y_1 \in  [0,1]\}=\left\{\lim_{r\rightarrow \infty} \mathcal K(rx) : x \in [0,\infty)^2-\{(0,0)\}\right\}.$$
The key point to describe the direction of the dynamical system $\phi$ coming  from infinity is the following change of time. It allows  to extend the flow on the boundary and is an example of   Poincar\'e 's compactification technics, which is particularly powerful for polynomial vector field  \cite{D}. More precisely, we consider the flow $\Phi$  of the dynamical system on $\overline{\mathcal S}$ given for $z_0\in \overline{\mathcal S}$ and $t\geq 0$  by
 \be
 \label{chgmt}
 \Phi(z_0,0)=z_0, \qquad \frac{\partial}{\partial t} \Phi(z_0,t)=H( \Phi(z_0,t)),
 \ee
 where
$H$ is the  Lipschitz function on $\overline{\mathcal S}$ defined by
 \be
 H^{(1)}(y_1,y_2,y_3)&=&y_1y_2[(b-c)y_2+(d-a)y_1]+y_1y_3[(\tau_1-\tau_2-c)y_2-ay_1+y_3\tau_1] \nonumber\\
H^{(2)}(y_1,y_2,y_3)&=&y_1y_2[(a-d)y_1+(c-b)y_2]+y_2y_3[(\tau_2-\tau_1-d)y_1-by_2+y_3\tau_2] \nonumber\\
H^{(3)}(y_1,y_2,y_3)&=&y_3(ay_1^2+by_2^2+(c+d)y_1y_2-\tau_1 y_1y_3-\tau_2 y_2y_3). \label{deffH}
\ee
The study of $\Phi$ close to $\partial \mathcal S$ is giving us the behavior of $\phi$ close to infinity using the change of time $\varphi\in \mathcal C^1([0,\infty)^2\times[0,\infty), [0,\infty))$ defined by
$$\varphi(x_0,0)=x_0, \qquad \frac{\partial}{\partial t}\varphi(x_0,t)=1+\parallel \phi(x_0,t)\parallel_1.$$
\begin{lem}
\label{vraiPoinc}
For any $x_0\in [0,\infty)^2 $ and $t\geq 0$, %\rightarrow [0,\infty)$ such that  $v(x)\rightarrow \infty$ as $x\rightarrow \infty$ and 
% the function $\varphi$ defined by
 %$$\varphi(x_0,0)=0, \qquad \frac{\partial}{\partial t} \varphi(x_0,t)=v(\phi(x_0,t))$$
 %satisfies
 $$\mathcal K(\phi(x_0,t))=\Phi(\mathcal K(x_0),\varphi(x_0,t)).$$
  \end{lem}
 \begin{proof} 
 We denote by $(y_t : t\geq 0)$ the image of the dynamical system $(x_t : t\geq 0)$ through $\mathcal K$ : 
$$y_t=\mathcal K(x_t)=\mathcal K(\phi(x_0,t)).$$
Then
$$y_t'=G(x_t)=G\circ \mathcal K^{-1}(y_t)$$
where
$$G^{(1)}(x_1,x_2)=\frac{(d-a)x_1^2x_2+(b-c)x_1x_2^2+(\tau_1-\tau_2-c)x_1x_2-ax_1^2 +\tau_1 x_1}{(1+x_1+x_2)^2}$$
and
$$G^{(2)}(x_1,x_2)=\frac{(c-b)x_2^2x_1+(a-d)x_2x_1^2+(\tau_2-\tau_1-d)x_2x_1-bx_2^2+\tau_2 x_2}{(1+x_1+x_2)^2}$$
and
$$G^{(3)}(x_1,x_2)=\frac{ax_1^2+bx_2^2+(c+d)x_1x_2-\tau_1 x_1-\tau_2 x_2}{(1+x_1+x_2)^2}.$$
Using that $x_1=y_1/y_3$ and $x_2=y_2/y_3$ and recalling the definition $(\ref{deffH})$ of $H$, we have
\be
\label{idtt}
G\circ \mathcal K^{-1}(y)=\frac{1}{y_3}H(y)
\ee
for $y=(y_1,y_2,y_3) \in  \mathcal S$.
The key point in the  theory of Poincar\'e is that $H$ is continuous on $\overline{\mathcal S}$ and that
the trajectories of the dynamical system $(z_t : t\geq 0)$ associated to the vector field $H$ : 
$$z'_t=H(z_t)$$
are the same than the trajectories of $(y_t :  t\geq 0)$ whose vector field is $G\circ \mathcal K^{-1}$. Indeed the positive real number $1/y_3$ only changes the norm of the vector field and thus the speed at which the same trajectory is covered. The associated change of time $v_t=\varphi(x_0,t)$ such that
$$z_{v_t}=y_t=\mathcal K(x_t)$$
 can now be simply computed. Indeed $(z_{v_t})'=H(y_t)v_t'$ coincides with $y_t'= G\circ \mathcal K^{-1}(y_t)$  as soon as
$$v_t'=\frac{1}{y_t^{(3)}}=\frac{1}{\mathcal K^{(3)} (\phi(x_0,t))}=1+\parallel \phi(x_0,t)\parallel_1,$$
 using  (\ref{idtt}). This   completes the proof.
 %In particular, the direction taken by the dynamical system $(y_t :  t\geq 0)$ coming from  the boundary
%$\partial \mathcal S$, and thus the direction of $(x_t :  t\geq 0)$,
%is given by that of $(z_t :  t\geq 0)$.
\end{proof}

To describe the direction from which the flow $\phi$ comes down from infinity,
we  introduce  the hitting times of cones centered in $x$ :
 \be
 \label{deftt}
 t_{-}(x_0,x, \varepsilon)=\inf_{s\geq 0} \{ \widehat{(x_s, x)} \in [-\varepsilon, +\varepsilon]\}, \quad t_{+}(x_0,x, \varepsilon)=\inf_{s\geq t_{-}(x_0,x,\varepsilon)} 
 \{ \widehat{(x_s,x)} \not\in [-2\varepsilon, +2\varepsilon]\},
 \ee
where we recall that $x_s=\phi(x_0,s)$ and
 $\inf \varnothing =\infty$. The  directions  $x_{\ell}$ of  the coming down from infinity are  defined by
\begin{itemize}
 \item $x_{\ell}=x_{\infty}$ if  $b>c$ and $a>d$, where $x_{\infty}$ has been defined in (\ref{defxinf}).
 \item $x_{\ell}=(1/a,0)$ if  $b>c$  and $a\leq d$; or if $b\geq c$ and $a<d$; or  if $c>b$ and $d>a$ and    $\widehat{(x_{0},x_{\infty})}>0$.
 \item $x_{\ell}=(0,1/b)$ if  $a>d$  and $b\leq c$; or if $a\geq d$  and $b< c$; or if  $c>b$ and $d>a$ and    $\widehat{(x_0, x_{\infty})}<0$.
 \item $x_{\ell}=\widehat{x_0}$ if $a=d$ and $b=c$, where 
 $\widehat{x_0}=x_0/(ax_0^{(1)}+bx_0^{(2)})$   for  any $x_0 \in (0,\infty)^2$.
\end{itemize}
The proof is given below and rely on the previous compactification result. We can then specify the speed of coming down from infinity of the flow $\phi$ since the problem is reduced to  the one dimension where computations can be easily lead. \\ \\

\noindent \emph{Figure 1 : flow close to infinity. We draw  the four regimes of the compactified flow $\Phi$ starting close or on  the boundary $\partial \mathcal S$  and below
the associated  behavior of the original  flow $\phi$ on $[0,\infty)^2$. The fixed points of the boundary are fat.}

\includegraphics[scale=0.85]{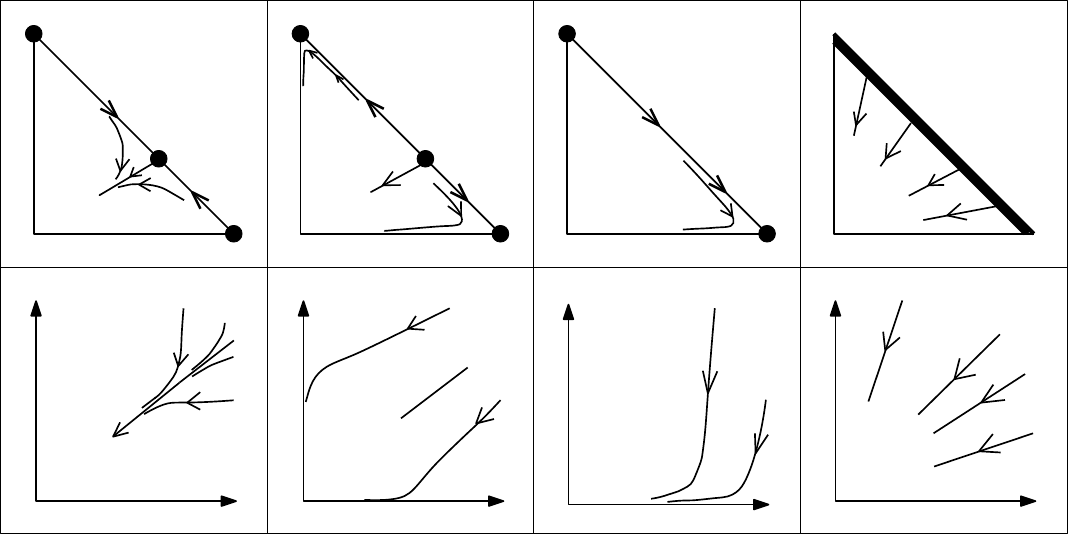} 

%\marginpar{A reprendre, mettre $t_+$ avec $2\varepsilon$ pour \'eviter soucis ?}
 \begin{lem} \label{Poinc} (i) For any $T>0$,  there  exists $c_T>0$ such that
  $\parallel \phi(x_0,t) \parallel_1\leq c_T/t$ for all $x_0 \in [0,\infty)^2$ and $t \in (0,T]$.  \\
(ii) For all  $x_0 \in (0,\infty)^2$ and  $\varepsilon>0$,
% \marginpar{A t on prouve qu'il s'agit bien de limites ???}
  $$\lim_{r\rightarrow \infty} t_{-}(rx_0,x_{\ell}, \varepsilon)= 0, \qquad \lim_{r\rightarrow \infty} t_{+}(rx_0,x_{\ell}, \varepsilon)>0.$$ 
   (iii)  %\marginpar{Lemme a part?}
  Moreover, 
 $$ \lim_{t\rightarrow 0} \limsup_{r\rightarrow \infty}  \big\vert \parallel t\phi(rx_0,t) \parallel_1 - \parallel x_{\ell}\parallel_1\big\vert=0.$$
 \end{lem}
 %where $v_{\infty}=\parallel x_{\infty}\parallel_1$. \\
% if $x_{\ell}=x_{\infty}$;  $v_{\infty}=1/a$ if $x_{\ell}=(1,0)$;  $v_{\infty}=1/b$ if  $x_{\ell}=(0,1)$. \\
\begin{proof} $(i)$ 
Using $a>0$, we first observe that 
 $$(x^{(1)}_t)'\leq -\frac{a}{2}(x^{(1)}_t)^2$$
 in the time intervals when  $x^{(1)}_t\geq 2\tau_1/a$.
 Solving  $(x^{(1)}_t)'= -(x^{(1)}_t)^2a/2$  proves $(i)$. \\
 
 % $(x_t^{(2)})'\leq -\widetilde{b}(x_t^{(2)})^2$ tant que les coordonn\'ees sont assez grandes 
%(pour des constantes choisies pour contrer les termes lin\'eaires) et donc uniform\'ement en la condition initiale
$(ii)$
 We use the notation (\ref{chgmt}) and (\ref{deffH}) above
%$$y=(y_1,y_2,y_3)=F(x)=
and the dynamics of $z_t=\Phi(z_0,t)$    on the invariant set $\partial \mathcal S$ is
simply given by the vector field $H(y_1,y_2,0)$ for $y_1 \in [0,1]$, $y_1+y_2=1$:
$$H^{(1)}(y_1,y_2,0)=-H^{(2)}(y_1,y_2,0)=y_1y_2[(b-c)y_2+(d-a)y_1].$$
The two points  $(1,0,0)$ and $(0,1,0)$  on $\partial \mathcal S$ are invariant for the dynamical system $(z_t : t\geq 0)$. \\
 Let us first consider  the case when $a\ne d$  or $b\ne c$.
There is  an additional invariant point in  $\partial S$ if and only if
$$(b-c)(a-d)>0.$$

Thus, if $(b-c)(a-d)\leq 0$,  $H^{-1}((0,0,0))\cap \partial \mathcal S=\{ (1,0,0), (0,1,0)\}$ and 
$z_t$ starting from the boundary $\partial \mathcal S$  goes either to $(1,0,0)$ whatever its initial value $z_0$ in the interior of the boundary; 
or to $(0,1,0)$ whatever its initial value $z_0$ in the interior of the boundary. These  cases are inherited  from the sign of $b-c$, which provides the stability
of the fixed points $(1,0,0)$ and $(0,1,0)$.
Then by 
Lemma \ref{vraiPoinc} the dynamical system $z_{\varphi(x_0,t)}=\mathcal K(x_t)$ starting  close to the boundary  $\partial \mathcal S$ goes 
\begin{itemize}
\item either 
 to $(1,0,0)$; and then $\widehat{(x_t,x_{\ell})}$ becomes small, where  $x_{\ell}=(1/a,0)$.
 \item
or to $(0,1,0)$; and then $\widehat{(x_t,x_{\ell})}$ becomes small, where $x_{\ell}=(0,1/b)$.
\end{itemize} 
More precisely, $z$ issued from $\mathcal K(\phi(rx_0,t))$ 
reaches  any neighborhood of $(1,0,0)$ or  $(0,1,0)$ in  a time which is bounded for $r$ large enough. Adding
that 
$\partial \varphi (rx_0,t) / \partial t=1+\parallel \phi(rx_0,t) \parallel_1$   is large before 
$z_{\varphi(rx_0,.)}$ has reached this neighborhood ensures that  this reaching time is arbitrarily small for $\mathcal K^{-1}(\phi(rx_0,.))$ when $r$ is large. 
This proves that $t_{-}(rx_0, x_{\ell},\varepsilon)\rightarrow 0$ as $r\rightarrow \infty$. 
Moreover $t_{+}(rx_0, x_{\ell},\varepsilon)$  is not becoming close to $0$ as $r\rightarrow \infty$  since the speed of the dynamical system $\phi(rx_0,.)$ is bounded on the compacts sets of  $[0,\infty)^2$.\\

%\marginpar{Clarifier}
Otherwise $(b-c)(a-d)> 0$ and
$$H^{-1}((0,0,0))\cap \partial \mathcal S=\{ (1,0,0), (0,1,0), z_{\infty}\},$$
where $z_{\infty}$ is the unique invariant point in the interior of the boundary : 
$$z_{\infty}=\frac{1}{ b-c +a-d } \left( b-c, a-d,0\right).$$
 Then  we need to see if $z_{\infty}$ is  repulsive or attractive on the invariant set $\partial  \mathcal S$. In the case $c>b$ and $d>a$, this point is attractive and $z_{\infty}$ is a a saddle and 
$$z_{\infty}=\lim_{r\rightarrow \infty} \mathcal K(rx_{\infty}).$$
So  Lemma \ref{vraiPoinc} now ensures that the dynamical system $x_t$ takes the direction 
$x_{\ell}=x_{\infty}$ when starting from a large initial value. As in the previous case,  $t_{-}(rx_0, x_{\ell},\varepsilon)\rightarrow 0$ and 
$t_{+}(rx_0, x_{\ell},\varepsilon)$  does  not. \\
In the  case $b<c$ and $a<d$,    $y_{\infty}$ is a source
and the dynamical system $z_t$ either goes to  $(1,0,0)$  (and then  $x_{\ell}=(1/a,0)$)
or to $(0,1,0)$  (and then $x_{\ell}=(0,1/b)$). This depends on the position of the initial value with respect to the second unstable variety and thus on the sign of $\widehat{(x_0,x_{\infty})}$. \\

Finally, the case $a=d$,  $b=c$ is handled similarly noting that the whole set $\partial \mathcal S$ is invariant. 
\\

$(iii)$ We know from $(ii)$ that the direction of the dynamical system coming from infinity is $x_{\ell}$
and  we  reduce now its dynamics close to infinity to a one-dimensional and solvable problem. Indeed,  let us write
$$x_t(r)=\phi(rx_0,t)$$
%\theta_{\ell}= \frac{x_{\ell}^{(2)}}{x_{\ell}^{(1)}}\in [0,\infty)
and focus on the case  $x_{\ell}^{(1)}\ne 0$.  First,
we 
observe that for any $T>0$, there exists $M_T>0$ such that for any $t\in [0,T]$ and $r\geq 1$,  
\be
\label{boundq}
x_t^{(2)}(r) \leq M_T x_t^{(1)}(r).
\ee
Indeed $\mathcal K(x_t)=z_{v_t}$ does not come close to the boundary $\{(0,u,1-u) : u\in [0,1]\}$ on compact time intervals 
when $x_{\ell}^{(1)} \ne 0$.
Plugging (\ref{boundq})
%$x_t^{(2)}(r)\leq M_Tx_t^{(1)}(r)$ for $t\in [0,T]$ and $r\geq 1$. Using this bound 
in (\ref{dyncompet}) provides a lower bound for  $x_t^{(1)}(r)$ and  we obtain for any $\varepsilon>0$,
$$%\in (0, \infty], \qquad
%\marginpar{irait bien dans descente lemme 5.7 non ?}
%and
t_1(\varepsilon)=\liminf_{r\rightarrow \infty}  \ \inf   \left\{t \geq 0 : x^{(1)}_t(r) < (\vert \tau_1\vert+1) /\varepsilon \right\}  \in (0,\infty].$$% \quad \text{where } \ t_1(r,\varepsilon):= \ \inf   \left\{t \geq 0 : x^{(1)}_t(r) < (\vert \tau_1\vert+1) /\varepsilon \right\} $$
Moreover, by definition  $(\ref{deftt})$,  for any $\varepsilon>0$ and  $r>0$  and
$t\in[t_{-}(rx_0,x_{\ell},\varepsilon),   t_{+}(rx_0,x_{\ell}, \varepsilon)]$, we have $\widehat{(x_t(r),x_{\ell})}\leq 2\varepsilon$ and  %and  for $\varepsilon$ small enough,
\be
\label{quotient}
\left\vert \frac{x_t^{(2)}(r)}{x_t^{(1)}(r)} - \frac{x_{\ell}^{(2)}}{x_{\ell}^{(1)}} \right\vert \leq u(\varepsilon),
\ee
where $u(\varepsilon) \in [0,\infty]$ and $u(\varepsilon)\rightarrow 0$ as $\varepsilon\rightarrow 0$.
We write
$$\theta_{\ell}= \frac{x_{\ell}^{(2)}}{x_{\ell}^{(1)}}, \qquad t_{-}(r)=t_{-}(rx_0,x_{\ell},\varepsilon), \qquad  t_{+}(r)=t_{+}(rx_0,x_{\ell},\varepsilon)\wedge t_1(u(\varepsilon))  $$
for convenience. Plugging  (\ref{quotient}) in  the first equation of  (\ref{dyncompet})  yields  for any $t\in [t_-(r), t_{+}(r)]$ and $r\geq1$, 
% It ensures that
%$$\tau_1 x_t^{(1)}(r)-(a+c\theta_{\ell}+u(\varepsilon))(x_t^{(1)}(r))^2  \leq (x_t^{(1)}(r))'\leq \tau_1 x_t^{(1)}(r)-(a+c\theta_{\ell}-u(\varepsilon))(x_t^{(1)}(r))^2$$
% we get
$$-(a+c\theta_{\ell}+(1+c)u(\varepsilon))  \leq \frac{(x_t^{(1)}(r))'}{(x_t^{(1)}(r))^2}\leq -(a+c\theta_{\ell}-(1+c)u(\varepsilon)).$$
 We get by integration, for any $\varepsilon$ small enough,
$$ \frac{1}{(a+c\theta_{\ell}+(1+c)u(\varepsilon))(t-t_{-}(r))+1/x^{(1)}_{t_-(r)}(r)}\leq x^{(1)}_t(r)\leq \frac{1}{(a+c\theta_{\ell}-(1+c)u(\varepsilon))(t-t_{-}(r))+1/x^{(1)}_{t_-(r)}(r)}. $$
% for any  $t\in [t_-(r,\varepsilon),\wedge t_{+}(r,\varepsilon)]$.
Using $(ii)$,   $t_{-}(r)\rightarrow 0$  and  $t_+=\liminf t_+(r)>0$ as  $r\rightarrow \infty$.  Moreover 
%\marginpar{A clarif cette valeur qui tend vers l'infini, et attention aux renom positifs}
$x_{\ell}^{(1)}\ne 0$ ensures that
$x_{t_-(r)}^{(1)}(r)\rightarrow \infty$ as $r\rightarrow \infty$. Then for any $\varepsilon$ positive small enough and  $t\leq t_+,$
$$ \frac{1}{a+c\theta_{\ell}+(1+c)u(\varepsilon)}\leq \liminf_{r\rightarrow \infty} tx^{(1)}_t(r)\leq \limsup_{r\rightarrow \infty} tx^{(1)}_t(r) \leq  \frac{1}{a+c\theta_{\ell}-(1+c)u(\varepsilon)}.$$
Letting finally $\varepsilon\rightarrow 0$, $u(\varepsilon) \rightarrow 0$ and we obtain
$$\lim_{t\rightarrow 0} \limsup_{r\rightarrow \infty} \vert tx^{(1)}_t(r)-1/(a+c\theta_{\ell}) \vert =0.$$
Using again $(\ref{quotient})$ provides the counterpart for $tx^{(2)}_t$ and
 ends  the proof  in the case $x_{\ell}^{(1)}\ne 0$. The case $x_{\ell}^{(2)}\ne 0$ is treated similarly . 
\end{proof}

\subsubsection{Approximation of  the flow of  scaled  birth and death processes}
 \label{scaled}
We use  notation  of Sections \ref{EDS} for 
 $$X^K=\begin{pmatrix} X^{K,(1)} \\ X^{K,(2)} \end{pmatrix}$$
with here   $E=\{0,1,2, \ldots\}^2$,  $\chi=[0,\infty)$, $q(dz)=dz$ and
$$h_F^K(x)=\int_{0}^{\infty} [F(x+H^K(x,z))-F(x)]dz,$$
where $H^K$ is defined in (\ref{defHK}).
Recalling the definition of  $F_{\beta,\gamma}$ from  (\ref{defFbg}), we get
\be \label{hnm}
 h_{F_{\beta,\gamma}}^{K}(x) =\begin{pmatrix} \lambda_1Kx_1\left((x_1+1/K)^{\beta}-x_1^{\beta}\right)+ Kx_1\left(\mu_1+ax_1 +cx_2\right)\left((x_1-1/K)^{\beta}-x_1^{\beta}\right)\\
 \gamma\lambda_2Kx_2\left((x_2+1/K)^{\beta}-x_2^{\beta}\right)+ \gamma Kx_2\left(\mu_2+bx_2 +dx_1\right)\left((x_2-1/K)^{\beta}-x_2^{\beta}\right)\end{pmatrix}.
 \ee
  We  consider
$$b_{F_{\beta,\gamma}}^K=J_{F_{\beta,\gamma}}^{-1}h_{F_{\beta,\gamma}}^K, \qquad \psi^K_{F_{\beta,\gamma}}=h^K_{F_{\beta,\gamma}}\circ F_{\beta,\gamma}^{-1} $$
and 
we  recall  that $\mathfrak {D}_{\alpha}=\{ (x_1,x_2) \in (\alpha, \infty)^2 : \   x_1\geq \alpha x_2, \  x_2 \geq \alpha x_1 \}$ and 
$$b(x)=\begin{pmatrix} \tau_1x_1 -ax_1^2-cx_1x_2  \\

 \tau_2 x_2 -bx_2^2-dx_1x_2 \end{pmatrix}, \qquad  \psi_{F_{\beta,\gamma}}=
(J_{F_{\beta,\gamma}}b)\circ F_{\beta,\gamma}^{-1}.$$
 To compare these quantities and approximate the flow associated with $b^K$, we introduce
$$\Delta^K_{\beta, \gamma}(x)=\frac{\beta (\beta-1)}{2K} \begin{pmatrix} (ax_1+cx_2)x_1^{\beta-1} \\
\gamma  (bx_2+dx_1)x_2^{\beta-1}
\end{pmatrix}.$$
% \marginpar{$\beta=1$ marcheriat bien aussi je cross}
\begin{lem} \label{pourri}   For any $\alpha>0$ and $\beta \in (0,1]$ and $\gamma>0$, there exists $C>0$ such that for any $x \in \mathfrak D_{\alpha}$ and $y\in 
F_{\beta,\gamma}(\mathfrak D_{\alpha})$  and $K\geq 2/\alpha$, \\
(i)
$$\parallel h_{F_{\beta,\gamma}}^K(x)-J_{F_{\beta,\gamma}}(x)b(x)-\Delta^K_{\beta, \gamma}(x)\parallel_2 \leq \frac{C}{K}\parallel x\parallel_2^{\beta-1}.$$
(ii) $$\parallel  b_{F_{\beta,\gamma}}^K(x)-b(x) \parallel_2 \leq \frac{C}{K} \parallel x\parallel_2.$$
(iii) $$
\psi^K_{F_{\beta,\gamma}}(y) =\psi_{F_{\beta,\gamma}}(y)+\Delta^K_{\beta, \gamma}(F_{\beta,\gamma}^{-1}(y))+ R^K_{\beta, \gamma}(F_{\beta,\gamma}^{-1}(y)), 
$$
%\marginpar{indq depdc cst en $\beta, \gamma$ ?}
where $\parallel R^K_{\beta, \gamma}(x)\parallel_2\leq C/K$. \\
(iv) Moreover $\psi^K_{F_{\beta,\gamma}}$ is $(C,C/K)$ non-expansive on each convex component of $F_{\beta,\gamma}(D_{\beta,\gamma}\cap \mathfrak D_{\alpha})$, where we recall that $D_{\beta,\gamma}$
is defined in (\ref{defDbeta}).\\
(v) Finally, 
$$
%\label{deuzz}
\parallel \psi^K_{F_{\beta,\gamma}}(y) -\psi_{F_{\beta,\gamma}}(y)\parallel_2 \leq C\frac{1+ \parallel y\parallel}{K}.
$$
\end{lem}
\begin{proof}
First, by Taylor-Lagrange formula applied to $(1+h)^{\beta}$,  there exists $c_0>0$ such that 
$$\bigg\vert \left(z+\frac{\delta}{K}\right)^{\beta}-z^{\beta}-\frac{\delta}{K}\beta z^{\beta -1} -\frac{\delta^2}{2K^2}\beta (\beta-1) z^{\beta -2}\bigg\vert
\leq \frac{c_0}{K^2} z^{\beta -3}$$
for any $z>\alpha$ and $K\geq 2/\alpha$ and $\delta \in \{-1,1\}$, since $h=\delta/(Kz)\in (-1/2, 1/2)$.
Using then $(\ref{hnm})$
and $$
J_{F_{\beta,\gamma}}(x)=\begin{pmatrix} \beta x_1^{\beta-1}   & 0 \\ 0 & \gamma \beta x_2^{\beta-1} \end{pmatrix}, \qquad  J_{F_{\beta,\gamma}}(x)b(x)=\begin{pmatrix}\beta x_1^{\beta-1} x_1(\tau_1-ax_1-cx_2) \\

\gamma \beta x_2^{\beta-1} x_2(\tau_2-bx_2-dx_1)
 \end{pmatrix}  $$
yields $(i)$, since $\parallel x \parallel_2$, $x_1$ and $x_2$ are equivalent up to a positive constant when  $x\in \mathfrak D_{\alpha}$. We immediately get
$(iii)$ since $\parallel x\parallel_2^{\beta-1}$ is bounded on $[\alpha, \infty)^2$ when $\beta\leq 1$. \\

Then $(i)$  and the fact that there exists  $c_0>0$ such that for any $x\in \mathfrak  D_{\alpha}$ and $u \in [0,\infty)^2$,
$$\parallel J_{F_{\beta,\gamma}}(x)^{-1}\Delta^K_{\beta, \gamma}(x)\parallel_2 \leq c_0 \frac{\parallel x\parallel_2}{K},  \qquad \parallel J_{F_{\beta,\gamma}}(x)^{-1} u\parallel_2 \leq c_0\parallel x \parallel_2^{1-\beta}\parallel u \parallel_2$$  %for  $x\in  \mathfrak D_{\alpha}$  
%$$\parallel  b_{F_{\beta,\gamma}}^K(x)-b(x) \parallel_2 \leq c_0'\frac{\parallel x\parallel_2}{K}.$$
%Moreover
%$$
%\psi^K_{F_{\beta,\gamma}} =\psi_{F_{\beta,\gamma}}+\Delta^K_{\beta, \gamma}(F_{\beta,\gamma}^{-1})+ R^K_{\beta, \gamma}(F_{\beta,\gamma}^{-1}), 
%$$where $\parallel R^K_{\beta, \gamma}(x)\parallel_2\leq c_0''/K$ for some $c_0''>0$. 
 proves $(ii)$.  \\

We observe that
$\Delta^K_{\beta, \gamma}\circ F_{\beta,\gamma}^{-1}$ is uniformly Lipschitz on $F_{\beta,\gamma}(\mathfrak D_{\alpha})$ with constant $L$ since its partial derivative are 
 bounded on this domain. Recalling  then  from
Lemma \ref{transformation} $(i)$ that $\psi_{F_{\beta,\gamma}}$ is $\bar{\tau}$ non expansive on $F_{\beta, \gamma}(D_{\beta,\gamma})$, the decomposition 
$(iii)$ ensures   that
%\marginpar{ici aussi clarifier l'utilisation de comb de non exp}
$\psi^K_{F_{\beta,\gamma}}$ is $(\overline{\tau}+L,C/K)$ non-expansive on $F_{\beta,\gamma}(D_{\beta,\gamma}\cap \mathfrak D_{\alpha})$. So $(iv)$ holds. \\

Finally, using $(iii)$ and adding that 
$$\sup_{y\in  F_{\beta,\gamma}(\mathfrak D_{\alpha}), K\geq 1 } K \frac{\parallel \Delta^K_{\beta, \gamma}(F_{\beta,\gamma}^{-1}(y)) \parallel_2}{\parallel y\parallel_2}=  \sup_{x\in
\mathfrak D_{\alpha}, K\geq 1} K \frac{\parallel \Delta^K_{\beta, \gamma}(x) \parallel_2}{\parallel F_{\beta,\gamma}(x)\parallel_2} \ < \ \infty$$ 
%is bounded for $x\in $, there exists $C>0$ such that
 proves $(v)$ and ends up the proof.% for any $i=1, \ldots, N$, for any $y \in F_i(D_i)$, 
\end{proof}

\subsubsection{Adjunction of  open  convex cones}
\label{adjsec}

%Finally, we  need the following additional results on the dynamical system coming from infinity.
We  decompose the trajectory of the flow  in $\mathcal {D}_{\alpha}=  (\alpha, \infty)^2$
into time intervals where a non-expansive transformation can be found.  This relies on the next Lemma and the results of Section \ref{nonexpcones}.
Recall   from (\ref{defdom}) that $T_D(x_0)$ is the exit time of $D$ for the flow started from $x_0$.
%=\inf\{ t \geq 0 : \phi(x_0, t) \not\in D\}$ 
Moreover 
$d_{\beta}(x,y)=\parallel F_{\beta,1}(x)-F_{\beta,1}(y) \parallel_2$ from $(\ref{distanceLV})$, while  the definition of $x_{\ell}$ is given in previous Section \ref{sectionPoinca}.
\begin{lem}
\label{Poincc}
(i) Let $\alpha>0$, $\beta \in (0,1]$,  $N \in \mathbb N$ and   $(C_i)_{i=1, \ldots ,N}$  be a family of open  convex cones of $(0,\infty)^2$
such that $$(0, \infty)^2 =\cup_{i=1}^N C_i.$$ Then, there exists  
 $\kappa \in \mathbb N$  and $\varepsilon_0>0$ and
% \marginpar{ $t(.,x_0) \in $ $n(x_0, .) \in \{1, \ldots\}^K$}
$(t_k(x_0) : k=0, \ldots,  \kappa)$  and $(n_k(x_0) : k=1,\ldots,\kappa-1)$   such that  for any $x_0\in \mathcal D_{\alpha}$,
 $$0=t_0(x_0)\leq t_1(x_0)\leq \ldots \leq t_\kappa(x_0)=T_{\mathcal D_{\alpha}}(x_0),  \quad n_k(x_0) \in \{1, \ldots, N\}$$
and  for any 
%\marginpar{Il va falloir justifier ce qui se passe au retour de l'infini, en particulier ici on autorise $T(x_0)=\infty$ ce qui demande par exemple d'indiquer la convergence}
$k\leq \kappa-1$ and $t \in [t_k(x_0),t_{k+1}(x_0))$, we have
$$\overline{B}_{d_{\beta}}(\phi(x_0,t), \varepsilon_0) \subset   C_{n_k(x_0)}.$$
$(ii)$ In the case $x_{\ell}=x_{\infty} \in (0, \infty)^2$, for any $x_0\in (0,\infty)^2$ and $\varepsilon>0$, 
$$\liminf_{r\rightarrow \infty} T_{\mathcal D_{\varepsilon}}(rx_0)>0.$$
$(iii)$ In the case $x_{\ell}=(1/a,0)$,  for any $x_0\in (0,\infty)^2$ and $\varepsilon>0$ and $T>0$,  for $r$ large enough,
$$T_{\mathcal D_{\varepsilon}}(rx_0)=\inf\{ t\geq 0 :  \phi(rx_0,t) \in [0,\infty)\times [0,\varepsilon] \} \leq T.$$
(iv) Under Assumption (\ref{condparam}), for any $\alpha_0>0$,
$$\inf_{x_0 \in \mathfrak{D}_{\alpha_0}} T_{\mathfrak{D}_{\alpha}}(x_0) \quad  \stackrel{\alpha\rightarrow 0}{\longrightarrow}+ \infty.$$
%\marginpar{a reprendre : pourquoi le system dynamique reste dans un tel domaine : les bords sont rentrants c'est plus simple ?}
%\marginpar{pas tout a fait suffisant de rentrer}
\end{lem}

\begin{proof} 
$(i)$
We  define
$$ C_i^{\varepsilon} =\{ x \in \mathcal D_{\alpha} \cap C_i \ : \ \overline{B}_{d_{\beta}}(x, \varepsilon) \subset C_i\}$$
%\marginpar{Ajouter un argument ?}
and we first observe that for $\varepsilon$ small enough,
$$\cup_{i=1}^N C_i^{2\varepsilon}=\mathcal D_{\alpha},$$ 
since $\beta>0$ and the open convex cones $C_i$  are   domains between two half-lines of $(0,\infty)^2$ and their collection for $i=1, \ldots, N$ covers $(0,\infty)^2$. 
We define 
$$u_0^i(x_0)=\inf\{t \geq 0  :  \phi(x_0,t) \in C_i^{2\varepsilon}\}, \quad  v_0^i(x_0)=\inf\{ t \geq u_0^i(x_0) : \phi(x_0,t) \not\in C_i^{\varepsilon}\}$$ and by recurrence for $k\geq 1$,
$$u_k^i(x_0)=\inf \{ t\geq  v^i_{k-1} (x_0): \phi(x_0,t) \in C_i^{2\varepsilon}\}, \quad v_k^i(x_0)=\inf\{ t \geq u_k^i(x_0) : \phi(x_0,t) \not\in C_i^{\varepsilon}\}.$$
 Let us then  note that
%\overline{\mathcal S} =\cup_ {i=1}^N \overline{\mathcal K(C_i)}, \qquad 
$$\partial {\mathcal S} =\cup_ {i=1}^N \partial {\mathcal K(\overline{C_i})}, \quad \text{where} \quad 
\partial \mathcal K (\overline{C_i}) =\overline{\mathcal K({C}_i)}- \mathcal K(\overline{C_i}) = \{(t,1-t,0) : t\in [a_i,b_i]\}$$
for some
%\marginpar{A pr\'eciser, clairifier; $\varepsilon$ lie a lecart entre les cones !!!!, expliquer qu'on peut prendre la distance $d_{\beta}$ qu'on veut car c'est des \'ecarts line\'eaires entres les c™nes.}
$0\leq a_i \leq b_i\leq 1$.
Recall that $z_t= \Phi(z_0,t)$ has been introduced in (\ref{chgmt})   and  is defined on $\overline{\mathcal S}$.  On the boundary $\partial{\mathcal S}$, it is 
 given by $(z^{(1)}_t,1-z^{(1)}_t,0)$ where $z^{(1)}_t$ is monotone. Outside this boundary, 
$(z_t :  t\geq 0)$ goes to a fixed point since the competitive Lotka-Volterra dynamical system  $(x_t  :  t \geq 0)$ does. This ensures that for any $i\in\{1, \ldots, N\}$, \\
$$M^i(x_0)=\max \{ k : v_k^i(x_0) < \infty\}$$
is bounded for $x_0 \in \mathcal D_{\alpha}$. 
The collection of time intervals $[u_k^i(x_0),v_k^i(x_0)]$ for $i=1, \ldots, N$ and  $k\leq M^i(x_0)$ provides a finite covering of $[0,T_{\mathcal D_{\alpha}}(x_0)]$.\\
Adding  that for $t \in [u_k^i(x_0),v_k^i(x_0)]$, $\overline{B}_{d_{\beta}}(\phi(x_0,t), \varepsilon) \subset C_i$ ends up the proof.  \\

$(ii)$ comes simply from Lemma \ref{vraiPoinc} which ensures that  in the case $x_{\ell}=x_{\infty}$, the dynamical system comes down from infinity in the interior of $(0, \infty)^2$, see also the first picture in Figure $1$ above. \\

%\marginpar{CLARIFIER}
$(iii)$ We use again the dynamical system $(z_t : t\geq 0)$ given by $\Phi$ and defined in $(\ref{chgmt})$.  More precisely,  the property here comes from 
the continuity of the associated  flow with respect to the initial condition. Indeed, in the case $x_{\ell}=(1/a,0)$, the trajectories of $(z_t: t\geq 0)$ starting from $r$ large 
go  to $(1,0,0)$ along the boundary $\partial \mathcal S$ and then remain close to boundary $\{(u,0,1-u) :  u \in [u_0,1]\}$ for some fixed $u_0<1$.
This ensures that  $(x_t : t\geq 0)$ exits from $\mathcal D_{\varepsilon}$  through $(0,\infty)\times \{\varepsilon\}$ and in finite time for $r$ large enough. 
%\marginpar{Bon argument a ajouter ! ? cette norm rest grande tang qu'une des composites lest}
The fact that this exit time $T_{\mathcal D_{\varepsilon}}(rx_0)$ goes to zero as $r\rightarrow \infty$ is due to the 
fact that the dynamics of $(x_t : t\geq 0)$ is an acceleration of  that of  $(z_t: t\geq 0)$ when starting close to infinity, with time change  $1+\parallel \phi(x_0,t)\parallel_1$.  \\
%Using the continuity of the flow of $(z_t : t\geq 0)$ around the point  $$z_{\infty}=\lim_{r\rightarrow \infty} F(rx_0),$$

Finally $(iv)$ is a consequence of Lemma \ref{Poinc}, noticing that  Assumption (\ref{condparam}) ensures that $x_{\ell}\in \{x_{\infty},x_0\}$, so the dynamical system does  not come fast to the boundary of $(0,\infty)^2$. 
 \end{proof}

 \begin{lem} \label{nonetsys} Let $\beta \in (0,1)-\{1/2\}$ such that $q_{\beta}=4ab(1+\beta)^2+4cd(\beta^2-1) > 0$ and $\alpha>0$. \\
  There exists $N\geq 1$,  $(\gamma_i : i=1,\ldots, N) \in (0,\infty)^N$,    convex cones   $(C_i : i=1, \ldots, N)$,    $\kappa\in \mathbb N$,  $\varepsilon_0>0$,
  %$(t_k(x_0) : k=0, \ldots,  \kappa)$  and $(n_k(x_0) : k=1,\ldots,\kappa-1)$   such that  for any $x_0\in \mathcal D_{\alpha}$,
 $0=t_0(x_0)\leq t_1(x_0)\leq \ldots \leq t_\kappa(x_0)= T_{\mathcal D_{\alpha}}(x_0)$ and   $n_k(x_0) \in \{1, \ldots, N\}$
 such that  : 
 
  (i) For each $i=1, \ldots, N$,  $\psi_{F_{\beta, \gamma_i}}$ is $\overline{\tau}$ non-expansive on
  $F_{\beta, \gamma_i}(C_{i})$ and
 $\cup_{i=1}^N C_i= (0,\infty)^2.$

%\marginpar{mutualiser la suite de temps $t_i, \alpha$...}
(ii) For any $x_0\in \mathcal D_{\alpha}$, $k=0,\ldots, \kappa-1$, $t\in (t_k(x_0), t_{k+1}(x_0))$,
 $$\overline{B}_{d_{\beta}}(\phi(x_0,t), \varepsilon_0) \subset   C_{n_k(x_0)} \cap \mathcal D_{\alpha/2}.$$

(iii) Finally, for $K$ large enough, there exists a continuous  flow $\phi^K$    such that  for any $x_0\in \mathfrak D_{\alpha}$,
$\phi^K(x_0,0)=x_0$ and for  any $k=0, \ldots,  \kappa-1$  and $t\in (t_k(x_0), t_{k+1}(x_0)\wedge  T_{\mathfrak D_{\alpha}}(x_0))$,
 %\marginpar{virer sur le $(2)$. $(iii)$ pour case $c=d=0$ ?}
 $$\overline{B}_{d_{\beta}}(\phi^K(x_0,t), \varepsilon_0/2) \subset   C_{n_k(x_0)} \cap  \mathfrak D_{\alpha/2}
\qquad \text{and} \qquad \frac{\partial}{\partial t} \phi^K(x_0,t)= b_{F_{n_k(x_0)}}^K(\phi^K(x_0,t))$$
 and for any $T>0$,
\be
\label{appflowK} 
\sup_{\substack{x_0 \in \mathfrak D_{\alpha}, \\ t < T_{\mathfrak D_{\alpha}}(x_0)\wedge T}} d_{\beta}( \phi^K(x_0,t), \phi(x_0,t)) \stackrel{K\rightarrow \infty}{\longrightarrow}0.
\ee
 \end{lem}
 \begin{proof}
 We only deal  with the case $c\ne 0$ (and then $d\ne 0$). Indeed,  we recall from Lemma \ref{transformation} that the proofs of $(i-ii)$ in the case $c=d=0$ is obvious, since   one can take $N=1$ and $C_1=(0,\infty)^2$.  Moreover the proof of $(iii)$ is simplified in that case. \\
 By Lemma 
\ref{rec}, for any $\gamma>0$, there exists $\eta(\beta, \gamma)>0$ such that 
  $\mathcal C_{\eta(\beta, \gamma), \beta,\gamma} \subset D_{\beta,\gamma}$ and $(\ref{equationsup})$ holds for some $A_{\beta,\gamma}, \mu_{\beta, \gamma}\geq 0$.
 The collection   of the convex components of $(\mathcal C_{\eta(\beta,\gamma), \beta,\gamma} :  \gamma>0)$   
covers $(0,\infty)^2$, since it contains the half lines $\{(x_1,x_2) \in (0,\infty) : x_1=x_2x_\gamma\}$ and $\{ x_{\gamma} : \gamma >0\}=(0,\infty)$.  We underline that this collection also
contains the cones $\{ (x_1,x_2) \in (0,\infty)^2 : x_1< \eta(\beta,\gamma) x_2\}$ and   $\{ (x_1,x_2) \in (0,\infty)^2 : x_2 < \eta(\beta,\gamma) x_1\}$.
 Then, by a compactness argument, we  can extract  a finite covering of $(0,\infty)^2$ from this collection of open convex cones. This means that there exists
 $N\geq 1$ and $(\gamma_i : i=1,\ldots, N) \in (0,\infty)^N$ and convex cones 
 $(C_i :i=1, \ldots,N)$ such that $\cup_{i=1}^N C_i= (0,\infty)^2$ and $C_i\subset C_{ \eta(\beta,\gamma_i), \beta, \gamma_i}$.
By  Lemma 
\ref{transformation},      $\psi_{F_{\beta,\gamma_i}}$ is $\overline{\tau}$  is non-expansive on
  $F_{\beta,\gamma_i}(C_{i})$  for each
 $i=1, \ldots,N$, which proves $(i)$. \\
  
     We let now $\alpha >0$. The point $(ii)$ is   a direct consequence of 
     Lemma  
\ref{Poincc}  $(i)$ applied to   the covering $(C_i : i=1, \ldots, N)$ of $(0,\infty)^2$. 
Indeed, one just need to choose $\varepsilon_0$ small enough so that $\overline{B}_{d_{\beta}}(x, \varepsilon_0)\subset \mathcal D_{\alpha/2}$ for any
$x\in \mathcal D_{\alpha}$. \\
%Considering  now the  sets $D_i^{\alpha}=C_i^\cap \mathcal D_{\alpha/2}$ ends the proofs of $(i)$.

Let us now deal with $(iii)$. First, from the proof of $(i)$ and writing
$F_i=F_{\beta,\gamma_i}$, $A_i=A_{\beta, \gamma_i}$ and  $\mu_i=\mu_{\beta, \gamma_i}$,
   (\ref{equationsup})  becomes 
\be
\label{superexp}
(\psi_{F_{i}}(y)-\psi_{F_{i}}(y')).(y-y') \leq -\mu_i (\parallel y \parallel_2\wedge \parallel y' \parallel_2)  \parallel y-y'\parallel_2^2,
\ee
for any $i=1,\ldots,N$  and $y,y' \in F_i(C_i)\cap B(0,A_i)^c$, since $F_i(C_i)$ is convex by construction and included in 
$F_i(C_{\eta(\beta,\gamma_i),\beta, \gamma_i})$.   \\
We define  the flow $\phi_i^K$ associated to $b^K_{F_i}$ on $C_i$ : 
 $$\phi^K_i(x_0,0)=x_0, \qquad \frac{\partial}{\partial t}  \phi_i^K(x_0,t)=b_{F_i} ^K(\phi_i^K(x_0,t))$$
%\marginpar{attention aux mvx temps $T_i$}
 for  $x_0\in C_i$   and $t<T_{i}^K(x_0)$, where $T_{i}^K(x_0)$ is the maximal time when this flow is well defined and belongs to $C_i$.   
 We consider  the image  $\widetilde{\phi}_i^K(y_0,t)=F_i(\phi^K_i(F_i^{-1}(y_0),t))$ of this  flow. It  
satisfies 
 $$ \widetilde{\phi}_i^K(y_0,t)=y_0, \quad \frac{\partial}{\partial t} \widetilde{\phi}_i^K(y_0,t)=\psi_{F_i}^K( \widetilde{\phi}^K_i(y_0,t))$$
 for any  $y_0 \in  F_i(C_i)$ and $ t<  T_i^K(F_i^{-1}(y_0))$.
Similarly, writing  $\widetilde{\phi}_i(y_0,t)=F_i(\phi(F_i^{-1}(y_0),t))$, we have 
 $$ \widetilde{\phi}_i(y_0,t)=y_0, \quad \frac{\partial}{\partial t} \widetilde{\phi}_i(y_0,t)=\psi_{F_i}( \widetilde{\phi}_i(y_0,t))$$
  for any  $y_0 \in  F_i(C_i)$ and $ t<  T_{C_i}(F_i^{-1}(y_0))$. \\
%\marginpar{La a vaudrait pas le coup de mettre le reste du flot dans le reste du process}
%Using then $(\ref{superexp})$ 
Combining $(\ref{superexp})$ with Lemma \ref{pourri} $(v)$ 
%$(\ref{superexp})$ and $(\ref{deuzz})$ 
and observing that  $\parallel y \parallel_2\wedge \parallel y' \parallel_2\geq \parallel y \parallel_2 (1-\varepsilon_0/A)$ when $y'\in B(y,\varepsilon_0)$ and $\parallel y \parallel \geq A$, the
 assumptions of   Lemma \ref{ctrflo} in Appendix  are met for $\psi_{F_i}$ and $\psi^K_{F_i}$ on the domain $F_i( C_i\cap\mathfrak D_{\alpha/2})$. We apply this lemma with 
 $\eta=Kr_K$. It ensures that  for any $T>0$ and any sequence $r_K\rightarrow 0$, 
 $$\sup_{\substack{y_0 \in F_i(  C_i  \cap \mathfrak D_{\alpha}), \  y_1\in \overline{B}(y_0,r_K)   \\ t <   \widetilde{T}_{i,\varepsilon_0}(y_0) \wedge T} } \parallel \widetilde{\phi}_i^K(y_1,t)- \widetilde{\phi}_i(y_0,t) \parallel_2 \stackrel{K\rightarrow \infty}{\longrightarrow}0,$$
where $ \widetilde{T}_{i,\varepsilon}(y_0)=\sup\{ t \in (0,  T_{C_i}(F_i^{-1}(y_0))) : \forall  s\leq t,  \ \overline{B}(\widetilde{\phi}_i(y_0,s), \varepsilon) \subset F_i(C_i\cap\mathfrak D_{\alpha/2})\}$.
%Choose now $\varepsilon_0$ small enough so that for any $x\in \mathfrak{D}_{\alpha}$, $\overline{B}(x,\varepsilon_0)\subset  \mathfrak{D}_{\alpha/2}$
%and then  choose $\varepsilon$ small enough to get from the previous limit that
  Then
   \be
   \label{limK}
   \sup_{\substack{x_0 \in C_i\cap \mathfrak D_{\alpha}, \ x_1\in \overline{B}_{d_{\beta}}(x_0,r_K) \\ t <  T_{i,\varepsilon_0}(x_0)\wedge T}} d_\beta({\phi}^K_i(x_1,t), {\phi}(x_0,t))
 \stackrel{K\rightarrow \infty}{\longrightarrow}0,
 \ee
 where
${T}_{i,\varepsilon}(x_0)=\sup\{ t \in (0, T_{C_i}(x_0)) :  \forall s\leq t, \  \overline{B}_{d_\beta}(\phi(x_0,t),\varepsilon)\subset C_i \cap \mathfrak{D}_{\alpha/2}\}$.  
%Let now $K$ be such that for any $i=1, \ldots, N$,
 % $$\sup_{\substack{x_0 \in  C_i \cap \mathfrak D_{\alpha},  \\ t < T_{i,\varepsilon_0}(x_0)\wedge T}} d_\beta({\phi}^K_i(x_0,t), {\phi}(x_0,t))\leq \frac{\varepsilon_0}{2\kappa}.$$
From $(ii)$, we also know   that   $\overline{B}_{d_\beta}(\phi(x_0,t),\varepsilon_0)\subset C_{n_k(x_0)} \cap \mathfrak{D}_{\alpha/2}$ for $t\in [t_k(x_0),t_{k+1}(x_0) \wedge T_{\mathfrak D_{\alpha}}(x_0))$,
so $$\sup_{\substack{x_0 \in \mathfrak D_{\alpha} \\  x_1\in \overline{B}_{d_{\beta}}(\phi(x_0,t_k(x_0)),r_K)  \\ t\in [t_{k}(x_0), t_{k+1}(x_0) \wedge T_{\mathfrak D_{\alpha}}(x_0)\wedge T) }} d_\beta\left({\phi}^K_{n_k(x_0)}\left(x_1,t-t_k(x_0)\right), {\phi}(x_0,t)\right) \stackrel{K\rightarrow \infty}{\longrightarrow}0.$$ 
%for $k=0,\ldots, \kappa-1$.
   Then for $K$ large enough, we construct   the continuous flow $\phi^K$  inductively for $k=0,\ldots, \kappa-1$ such that for any $x_0\in \mathfrak D_{\alpha}$, 
     $$\phi^K(x_0,0)=x_0, \qquad  \phi^K(x_0,t)=\phi_{n_k(x_0)}^K ( \phi^K(x_0,t_k(x_0)),t-t_k(x_0))$$
 for any  $t\in [t_k(x_0),t_{k+1}(x_0)\wedge T_{\mathfrak D_{\alpha}}(x_0))$. This construction satisfies
       $$\sup_{\substack{x_0 \in \mathfrak D_{\alpha},  \\ t\in [t_{k}(x_0), t_{k+1}(x_0) \wedge T_{\mathfrak D_{\alpha}}(x_0)\wedge T) }} d_\beta({\phi}^K(x_0,t), {\phi}(x_0,t)) \stackrel{K\rightarrow \infty}{\longrightarrow}0$$ 
and for $K$ large enough,  for any  $t\in [t_k(x_0),t_{k+1}(x_0)\wedge T_{\mathfrak D_{\alpha}}(x_0))$, 
        $$ \overline{B}_{d_{\beta}}(\phi^K(x_0,t), \varepsilon_0/2) \subset   C_{n_k(x_0)}\cap \mathfrak D_{\alpha/2}.$$
%Recalling $(\ref{limK})$, the right-hand-side of the latter  inequality can be made arbitrarily small choosing $K$ large enough. 
Adding that  $\phi^K_i$ is the flow associated with the vector field $b^K_{F_i}$
 ends the proof. 
   \end{proof}
 
  \subsection{Proofs of Theorem \ref{descenteLV} and Corollary \ref{descenteLVb} and Theorem \ref{scaling}} 
 
We can now prove the Theorem \ref{descenteLV} for the diffusion $X$ defined by (\ref{EDSLv}) using the results of  Section \ref{EDS}. Here $E=[0,\infty)^2$, d$=2$,
$q=0$ ($H=G=0$), $\sigma_j^{(i)}=0$ if $j\ne i$ and $$\sigma_1^{(1)}(x)=\sigma_1\sqrt{x_1}, \qquad \sigma_2^{(2)}(x)=\sigma_2\sqrt{x_2}.$$
Moreover  $b_{F_{\beta,\gamma}}=b$ is given by $(\ref{defb}$), $\psi_{F_{\beta,\gamma}}=(J_{F_{\beta,\gamma}}b_{F_{\beta,\gamma}})\circ F_{\beta,\gamma}^{-1}$ and
\be
\label{bFF}
\widetilde{b}_{F_{\beta,\gamma}}(x)=\frac{1}{2} \sum_{i=1}^{2}  \frac{\partial^2 F_{\beta,\gamma}}{\partial^2 x_i}  (x)\sigma^{(i)}_i(x)^2
=\frac{1}{2}  \beta(\beta-1) \begin{pmatrix} \sigma_1^2 x_1^{\beta-1} \\ \gamma \sigma_2^2 x_2^{\beta-1} \end{pmatrix}
\ee
and 
\be
\label{VFF}
V_{F_{\beta,\gamma}}(x)=  \sum_{i=1}^{2}  \left( \frac{\partial F_{\beta,\gamma}}{\partial x_i}  (x)\sigma^{(i)}_i(x)\right)^2=\beta^2\begin{pmatrix} \sigma_1^2 x_1^{2\beta-1} \\
(\gamma \sigma_2)^2 x_2^{2\beta-1}\end{pmatrix}.
\ee

\begin{proof}[Proof of Theorem \ref{descenteLV}]  
Let  $\beta \in (1/2,1)$ close enough to $1$ so that
$q_{\beta}=4ab(1+\beta)^2+4cd(\beta^2-1) > 0$. 
Using Lemma \ref{nonetsys} $(ii)$, we can check
Assumptions \ref{assumeplus} and  \ref{assumepluss}  of Section \ref{EDS} with $D=\mathcal D_{\alpha}$, $D_i= C_i \cap \mathcal D_{\alpha/2}$, $O_i=  \mathcal D_{\alpha/4}$  $(i=1, \ldots , N)$, $d=d_{\beta}$ and $\phi$ defined by (\ref{dyncompet}). Moreover, writing $F_i=F_{\beta,\gamma_i}$ for convenience, Lemma \ref{nonetsys} $(i)$ ensures that
$\psi_{F_i}$ is $\overline{\tau}$ non-expansive on
  $F_i(D_{i})$.  We recall also that  $\T{0}{\bar{\tau}}{\varepsilon}=\infty$ 
  and     apply then Theorem \ref{recollement}  to the diffusion $X$ and get for any $\varepsilon$ small enough,  for any  $T  <1$ and $x_0\in \mathcal D_{\alpha}$, 
$$\mathbb P_{x_0}\left( \sup_{t\leq T\wedge T_{\mathcal D_{\alpha}}(x_0)} d_{\beta}(X_t,\phi(x_0,t)) \geq \varepsilon\right) \leq C \sum_{k=0}^{\kappa-1} \int_{ t_k(x_0)\wedge T}^{t_{k+1}(x_0)\wedge T}
 \overline{V}_{d_{\beta},\varepsilon}(F_{n_k(x_0)}, x_0,t)  dt$$
 for some positive constant $C$, by a.s. continuity of $d_{\beta}(X_t,\phi(x_0,t))$ at time $T\wedge T_{\mathcal D_{\alpha}}(x_0)$.
We need now to control $\overline{V}$. First, we recall from Lemma \ref{nonetsys} $(ii)$  that $\overline{B}_{d_{\beta}}(\phi(x_0,t), \varepsilon_0) \subset\mathcal D_{\alpha/2}$ for $x_0\in \mathcal D_{\alpha}$ and
$t< T_{\mathcal D_{\alpha}}(x_0)$. Then we use
$(\ref{bFF})$ to see that $\widetilde{b}_{F_i}$ is bounded on $\mathcal D_{\alpha/2}$, so  
%\marginpar{le small enough c'est que  $ d_{\beta}(x,\phi(x_0,s)) \leq \varepsilon$ implique que le flot est dans $\mathcal D_{\alpha/2}$}
$$c'_i(\varepsilon):=\sup_{\substack{x_0 \in \mathcal D_{\alpha},  \ t < T_{\mathcal D_{\alpha}}(x_0) \\ d_{\beta}(x,\phi(x_0,t)) \leq \varepsilon}} \parallel \widetilde{b}_{F_i} (x) \parallel_1<\infty$$
%$V_{F_i}$ given here by
%$$V_{F_i}(x)=V_{F_i}(x_1,x_2)= ( \beta^2 x_1^{2\beta-2} . x_1, \gamma^2 \beta^2 x_2^{2\beta-2}.x_2)=( \beta x_1^{2\beta-1}, \gamma \beta x_2^{2\beta -1})$$% +\int_{\mathcal X} [f_i(x_i+K^i(x,z))-f_i(x_i)]^2 q(dz).$$
for $\varepsilon\leq \varepsilon_0$. Moreover plugging  Lemma \ref{Poinc} $(i)$
%, there exists $c_T>0$ such that we   
%$$x^{(1)}_t \leq c_T/t, \quad x^{(2)}_t\leq c_T/t,$$
%for $t\in (0,T]$ and $x_0\in (0,\infty)^2$. So using 
into
(\ref{VFF}) to control $V_{F_i}$,
there exists  $c''_i(\varepsilon)>0$ such that for any $x_0\in \mathcal D_{\alpha}$  and $t<  T_{ \mathcal D_{\alpha}}(x_0)$,
$$\overline{V}_{d_{\beta},\varepsilon}(F_i, x_0,t)= \sup_{\substack{x \in [0,\infty)^2 \\ d_{\beta}(x,\phi(x_0,t)) \leq \varepsilon}} \left\{ \varepsilon^{-2}\parallel V_{F_i} (x) \parallel_1+\varepsilon^{-1} \parallel \widetilde{b}_{F_i} (x) \parallel_1\right\}\leq \varepsilon^{-2}\frac{c''_i(\varepsilon)}{t^{2\beta-1}}+\varepsilon^{-1}c'_i(\varepsilon).$$
Adding that
$\int_0^. \left(\varepsilon^{-2}\frac{c''_i(\varepsilon)}{t^{2\beta-1}}+\varepsilon^{-1}c'_i(\varepsilon) \right)dt<\infty$
for $\beta<1$, we get
$$\lim_{T\downarrow 0} \sup_{x_0 \in \mathcal D_{\alpha}}\mathbb P_{x_0}\left( \sup_{t\leq T\wedge  T_{\mathcal D_{\alpha}}(x_0) } d_{\beta}(X_t,\phi(x_0,t)) \geq \varepsilon\right) =0$$
for $\varepsilon$ small enough.
This ends up the proof for $\beta<1$ close enough to $1$, which is enough to conclude, since $d_{\beta'}$ is dominated by $d_{\beta}$  on $\mathcal D_{\alpha}$ if $\beta'\leq \beta$.
 \end{proof}
$\newline$

We can now describe the coming down from infinity of the  two-dimensional competitive  Lotka-Volterra diffusion $X$.
\begin{proof}[Proof of Corollary \ref{descenteLVb}] 
Let us deal with $(i)$, so $x_{\ell}=x_{\infty} \in (0,\infty)^2$ and we fix $x_0 \in (0,\infty)^2$ and $\eta\in (0,1)$.
First, plugging  Lemma \ref{Poinc}  $(ii)$ and $(iii)$ in the inequality
$$\parallel tx_t(r) -x_{\infty} \parallel_2 \leq  \big\vert \parallel tx_t (r)\parallel_1 - \parallel x_{\infty}\parallel_1\big\vert + \min(\parallel tx_t(r)\parallel_2, \parallel x_{\infty} \parallel_2) \left\vert\sin \left(\widehat{x_t,x_{\infty} }\right) \right\vert $$
ensures that
\be
\label{descendons}
\lim_{T\rightarrow 0} \limsup_{r\rightarrow \infty} \sup_{\eta T\leq t \leq T} \parallel tx_t (r)-x_{\infty} \parallel_2=0.
\ee
Moreover, for any $\varepsilon>0$,  Lemma \ref{Poincc} $(ii)$ ensures that
$$\liminf_{r\rightarrow \infty} T_{\mathcal D_{\varepsilon}}(rx_0)>0,$$
where we recall  definition  $(\ref{defdom})$ for the exit time  $T_{\mathcal D_{\varepsilon}}(.)$.  % for $r$ large enough and $T_0$ small enough,
 % $\phi(rx_0,t) \in \mathcal{D_{\alpha}}$ for $t\in [0,T_0]$.
%Then  for $\varepsilon$ small enough, 
%$ T_{\mathcal{D}_{\alpha}}(x_0)> T_0$
Writing again $x_t(r)=\phi(rx_0,t)$ for convenience, 
Theorem  \ref{descenteLV}  ensures that    for any $\beta \in (0,1)$,
%\be
%\label{avantgeom}
$$\lim_{T\rightarrow 0} \limsup_{r\rightarrow \infty} \mathbb P_{rx_0}\left(\sup_{t\leq T} d_{\beta}(X_t,x_t(r))
\geq \varepsilon \right)=0.$$
%\ee
 Then, using that $d_{\beta}(tx,ty)=t^{\beta}d_{\beta}(x,y)$ and $\parallel tx_t(r) \parallel_1$ is bounded for $t\leq 1$ and $r>0$ by Lemma $\ref{Poinc} (i)$, the last 
 limit yields
 %\be
%\label{avantgeom}
\be
\label{avantgeom}
\lim_{T\rightarrow 0} \limsup_{r\rightarrow \infty} \mathbb P_{rx_0}\left(\sup_{ t\leq T}  \parallel tX_t- tx_t(r) \parallel_2
\geq \varepsilon \right)=0,
\ee
for any $\varepsilon>0$,
since  the euclidean distance is uniformly continuous from the bounded sets of $[0,\infty)^2$ endowed with $d_{\beta}$ to $\R^+$ endowed with the absolute value.

%Using  the two last limits displayed and % yields %($\ref{avantgeom}$) and $(\ref{geom})$ and
%$ \parallel tX_t-x_{\infty} \parallel_2 \leq  \parallel tX_t -t x_t\parallel_2+ \parallel tx_t -x_{\infty} \parallel_2,$
Combining $(\ref{descendons})$ and $(\ref{avantgeom})$ ensures that for any $\varepsilon>0$, 
%$$\lim_{T\rightarrow 0}\limsup_{r \rightarrow \infty}\sup_{\eta T \leq t \leq T} \parallel t\phi(x_0,t) -x_{\infty} \parallel_2=0$$
%Adding that  $\parallel tX_t
%-x_{\infty} \parallel_2 \leq  \parallel tX_t -t x_t\parallel_2+ \parallel tx_t -x_{\infty} \parallel_2$ and for each compact set $K$ and $\alpha>0$,  there exists $\varepsilon>0$ such that
%$x\in K$ and $\parallel x-y \parallel_2\geq \eta$ implies that 
%$%d_{\beta}(x,y)\geq \varepsilon$,  the two last bounds ensure that
%\marginpar{Ameliorer}
$$\lim_{T\rightarrow 0} \limsup_{r\rightarrow \infty} \mathbb P_{rx_0}\left(\sup_{\eta T \leq t \leq T} \parallel tX_t
-x_{\infty} \parallel_2
\geq  \varepsilon \right)=0.$$
This proves the first part of $(i)$. The second part of $(i)$ (resp. the proof of $(iv)$) is obtained similarly just by noting 
that $t_{-}(rx_0, x_{\infty}, \varepsilon)=0$ (resp. $t_{-}(rx_0, \widehat{x_{0}}, \varepsilon)=0$)   if $x_0$ is collinear to $x_{\infty}$.  \\

For the cases $(ii-iii)$, we know from  Lemma \ref{Poinc}   that the dynamical system is going to  the boundary of $(0,\infty)^2$ in short time. Let us deal with the case
 $$x_{\ell}=(1/a,0)$$
and the case $x_{\ell}=(0,1/b)$ would be handled similarly. 
We fix $x_0\in (0,\infty)^2$, $T_0>0$, $\varepsilon\in (0,1]$, $\eta>0$ and $\beta \in (0,1)$. By Theorem 
  \ref{descenteLV}, there exists    $T\leq T_0$ such that for $r$ large enough
$$
 \mathbb P_{rx_0}\left(\sup_{t\leq T\wedge T_{\mathcal D_{\varepsilon}}(rx_0)} d_{\beta}(X_t,x_t(r))
\geq \varepsilon \right)\leq \eta.
$$
By  Lemma \ref{Poincc} $(iii)$, for $r$ large enough, 
we have $T_{\mathcal D_{\varepsilon}}(rx_0)=\inf\{ t \geq 0 :  x^{(2)}_t(r) \leq  \varepsilon\}\leq T$. Thus,
$$
 \mathbb P_{rx_0}\left(d_{\beta}(X_{T_{\mathcal D_{\varepsilon}}(rx_0)} , x_{T_{\mathcal D_{\varepsilon}}(rx_0)}(r)) \geq \varepsilon \right)\leq \eta
\quad 
%$\mathbb P_{rx_0}\left( d_{\beta}(X_{ T_{\mathcal D_{\varepsilon}}(rx_0)},(\varepsilon, .))
%\geq \varepsilon \right)$
\text{and} \quad x^{(2)}_{T_{\mathcal D_{\varepsilon}}(rx_0)}(r)=\varepsilon.$$
Fix now $c\geq 1$ such that 
$c^{\beta}\geq 2$.
We get
\Bea
 \mathbb P_{rx_0}\left( X_{T_{\mathcal D_{\varepsilon}}(rx_0)}^{(2)} \geq c\varepsilon\right)
 %= \mathbb P_{rx_0}\left( \left(X_{T_{\mathcal D_{\varepsilon}}(rx_0)}^{(2)}\right)^{\beta}  \geq c^{\beta}\varepsilon^{\beta},   \   \left(\phi^{(2)}(rx_0,T_{\mathcal D_{\varepsilon}}(rx_0))\right)^{\beta} \leq \varepsilon^{\beta}  \right)\leq \eta.
 &=&\mathbb P_{rx_0}\left( \left(X_{T_{\mathcal D_{\varepsilon}}(rx_0)}^{(2)}\right)^{\beta} -    \varepsilon^{\beta} \geq (c^{\beta}-1)\varepsilon^{\beta}  \right) \\
 &\leq & \mathbb P_{rx_0}\left(d_{\beta}(X_{T_{\mathcal D_{\varepsilon}}(rx_0)} , x_{T_{\mathcal D_{\varepsilon}}(rx_0)}(r) )
\geq \varepsilon \right)\leq  \eta,
\Eea
since  $\varepsilon^{\beta}\geq \varepsilon$.
By Markov property and the fact that the boundaries of $[0,\infty)^2$ are absorbing, we obtain for $r$ large enough
\Bea
\p_{rx_0}\left(X^{(2)}_{2T_0}=0\right)&\geq &  \p\left(X_{T_{\mathcal D_{\varepsilon}}(rx_0)}^{(2)} \leq c\varepsilon,  \ 
\exists t \in [T_{\mathcal D_{\varepsilon}}(rx_0), T_{\mathcal D_{\varepsilon}}(rx_0)+T_0] :  X_t^{(2)}=0\right)\\
&\geq & (1-\eta)
   p(c\varepsilon),
\Eea
where
$$p(x)= \p_x\left(X^{(2)}_{T_0}=0\right).$$
Moreover $X^{(2)}$ is stochastically smaller than  a  one-dimensional Feller diffusion
and   $\sigma_2\ne 0$, so $\lim_{x\downarrow 0+}p(x)=1.$ 
 Letting $\varepsilon \rightarrow 0$ in the previous inequality yields
 $$\liminf_{r\rightarrow \infty} \p_{rx_0}\left(X^{(2)}_{2T_0}=0\right)\geq  1-\eta.$$
 Letting $\eta \rightarrow 0$ ends up the proof of $(ii-iii)$. 
  \end{proof}

Recalling notation of Section \ref{scaled}, we finally prove the scaling limit stated in Theorem \ref{scaling}.

\begin{proof}[Proof of Theorem \ref{scaling}] 
%\marginpar{indq depdc cst en $\beta, \gamma$ ?}
Let $T_0>0$ and $\beta \in (0,1/2)$ and $\alpha_0>\alpha>0$. We first observe that  assumption (\ref{condparam})
%($a>d$ and $b>c$) or ($a=d>0$, $b=c$) 
ensures that  $q_{\beta}=4ab(1+\beta)^2+4cd (\beta^2-1)> 0.$
%\marginpar{A revoir}
Using Lemma \ref{nonetsys}  $(iii)$,
Assumptions \ref{assumeplus} and  \ref{assumepluss} are satisfied for the process $X^K$, with the  domains $D=\mathfrak D_{\alpha}$ and 
$D_i=C_i\cap \mathfrak D_{\alpha/2}$,  the continuous flow $\phi^K$,    the transformations  $F_i=F_{\beta, \gamma_i}$, the times $t_k(.)\wedge T_{\mathfrak D_{\alpha}}(.)$ and the integers $n_k(.)$. Recalling that $C_i$ is convex and $C_i\subset C_{\eta(\beta, \gamma_i),\beta, \gamma_i}\subset D_{\beta,\gamma_i}$, we know from   Lemma \ref{pourri} $(iv)$ that $\psi^K_{F_i}$ is $(c_i,c_i/K)$ non-expansive on $F_i(D_i)$ for some constant $c_i\geq 0$.
 Thus,   we apply Theorem \ref{recollement}   
and there exists $\underline{\varepsilon}=\underline{\varepsilon}^K$ which does not depend on $K$ so that 
 for any $K\geq 1$, $\varepsilon\in (0,\underline{\varepsilon}]$,   $T<\min (\T{c_i/K}{c_i}{\varepsilon} : i=1, \ldots, N)\wedge (T_0+1)$ and    $x_0\in \mathfrak D_{\alpha}$,
$$
\mathbb P_{x_0}\left( \sup_{t< T \wedge T_{\mathfrak D_{\alpha}}(x_0)} d_{\beta}(X_t^K,\phi^K(x_0,t)) \geq \varepsilon\right) \leq C \sum_{k=0}^{\kappa-1} \int_{ t_k(x_0)\wedge T}^{t_{k+1}(x_0)\wedge T}
 \overline{V}^K_{d_{\beta},\varepsilon}(F_{n_k(x_0)},x_0,t)  dt,
 $$
 where $C$ is positive constant which does not depend on  $K,x_0$ and
  $$\overline{V}^K_{d_{\beta},\varepsilon}(F_i, x_0,t)= \sup \{ \varepsilon^{-2}\parallel V^K_{F_i} (x) \parallel_1 : x \in [0,\infty)^2, d_{\beta}(x,\phi^K(x_0,t)) \leq \varepsilon \}.$$ 
 Moreover for  $K$ large enough, we  have 
$4c_iT_0 \exp(2L_i T_0)< K\varepsilon$, so that  
$T_0<  \T{c_i/K}{c_i}{\varepsilon} \quad \text{for } \  i=1, \ldots, N$ and 
\be
\label{inegg}
\mathbb P_{x_0}\left( \sup_{t < T_0\wedge  T_{\mathfrak D_{\alpha}}(x_0)} d_{\beta}(X_t^K,\phi^K(x_0,t)) \geq \varepsilon\right) \leq C \sum_{k=0}^{\kappa-1} \int_{ t_k(x_0)\wedge T_0}^{t_{k+1}(x_0)\wedge T_0}
 \overline{V}^K_{d_{\beta},\varepsilon}(F_{n_k(x_0)},x_0,t)  dt.
 \ee
%$ T<  \T{c_i/K}{c_i}{\varepsilon} \quad \text{for } \  i=1, \ldots, N.$
%\marginpar{Preciser pourquoi le quotient est born\'e (bon r\'egime).}
Adding that
$$ V_{F_{\beta,\gamma}}^{K}(x)=\begin{pmatrix}  V_{F_{\beta,\gamma}}^{K,(1)}(x) \\  V_{F_{\beta,\gamma}}^{K,(2)}(x))\end{pmatrix}
=\int_0^{\infty} \left(F_{\beta,\gamma}(x+H^K(x,z))- F_{\beta,\gamma}(x+H^K(x,z))\right)^2dz $$
and recalling $(\ref{defHK}$) and writing $\gamma_1=1, \gamma_2=\gamma$, we have for $i\in\{1,2\}$ and $x\in \mathcal D_{\alpha}$,
\Bea
 V_{F_{\beta,\gamma}}^{K,(i)}(x)
 &=&\gamma_i \left[\lambda_i^K(Kx)\left((x_i+1/K)^{\beta}-x_i^{\beta}\right)^2 +\mu_i^K(Kx)\left((x_i-1/K)^{\beta}-x_i^{\beta}\right)^2\right] \\
 &\leq & \frac{\text{cst}}{K} x_i^{2\beta-2} x_i(1+x_1+x_2)
\Eea
for some $\text{cst}>0$, which depends on $\beta, \gamma, \alpha$ and can now change from line to line. \\
Then for $x\in \mathfrak{D}_{\alpha}$,
$$\parallel V_{F_{\beta,\gamma}}^K(x)\parallel_1 \leq \frac{\text{cst}}{K}\left( x_1^{2\beta} (1+x_2/x_1) + x_2^{2\beta} (1+x_1/x_2)\right) \leq \frac{\text{cst}}{K}\left( x_1^{2\beta}  + x_2^{2\beta}\right).   $$
Moreover from Lemma \ref{nonetsys} $(iii)$  that for $K$ large enough, 
 $\overline{B}_{d_\beta}(\phi^K(x_0,t),\varepsilon_0/2)\subset \mathfrak D_{\alpha/2}$ for any $x_0\in \mathfrak
 D_{\alpha}$ and $t<t_{\kappa}(x_0)$.  Combining the last part of Lemma \ref{nonetsys} $(iii)$
 and  Lemma \ref{Poinc} $(i)$,  
 $\parallel \phi^K(x_0,t) \parallel_1  \leq c_T/t$ for $t\in [0,T]$. We obtain 
%and the flow stays in $\mathfrak D_{\alpha}$ for $\alpha$ small enough. It can be seen using
%  and the fact that the process goes to an interior fixed point since $\tau_1,\tau_2$.
that 
for any $x_0\in \mathfrak{D}_{\alpha}$ and $\varepsilon\leq \varepsilon_0/2$, 
$$ \int_{ t_k(x_0)\wedge T_0}^{t_{k+1}(x_0)\wedge T_0} \overline{V}^K_{d_{\beta}, \varepsilon}( F_{n_k(x_0)},x_0,t)dt \leq  \varepsilon^{-2}\frac{\text{cst}}{K} \int_{ t_k(x_0)\wedge T_0}^{t_{k+1}(x_0)\wedge T_0} t^{-2\beta} dt$$
for $k\in\{0,\ldots, \kappa-1\}$. Using the fact that $\int_0^. t^{-2\beta} dt<\infty$ for $\beta<1/2,$
we get 
$$ \sum_{k=0}^{\kappa-1} \int_{ t_k(x_0)\wedge T_0}^{t_{k+1}(x_0)\wedge T_0}
 \overline{V}_{d,\varepsilon}(F_{n_k(x_0)}, x_0,t)  dt \leq  \varepsilon^{-2}\frac{\text{cst}}{K}.$$

Recall  now from Lemma \ref{Poincc} that  under Assumption (\ref{condparam}),
 we can choose $\alpha\in (0,\alpha_0)$ small enough so that $T_{\mathfrak D_\alpha}(x_0)\geq T_0$ for any $x_0 \in \mathfrak D_{\alpha_0}$.
Using also $(\ref{appflowK})$,
 %that the flow $\phi^K$ is uniformly close to  the flow $\phi$ for the distance $d_{\beta}$ when $K$ goes to infinity, 
 $(\ref{inegg})$ becomes
 $$\sup_{x_0 \in \mathfrak D_{\alpha_0}} \p_{x_0}\left(\sup_{t < T_0}d_{\beta} (X_t^K, \phi(x_0,t)) \geq \varepsilon\right) \leq  \varepsilon^{-2}\frac{C}{K},$$
for   $\varepsilon\leq \underline{\varepsilon}\wedge \varepsilon_0/2$ and $K$ large enough, where 
  $C$ is a  positive constant which does not depend on $K$. %This  ends up the proof.
\end{proof}

\noindent \emph{Remark. Let us mention an alternative approach. Using Proposition \ref{ctrltpscourt}  (or extending the Corollary of Section \ref{EDS}), one could  try to  compare directly the process $X$  to the flow $\phi$ (instead of $\phi^K$) and put the remaining term $R^K$ in a finite variation part $A_t$.}

\section{Appendix}

 We give here first  three technical results to study the coming down from infinity of dynamical systems in one dimension. 
 Let $\psi_1$ and $\psi_2$ be two locally Lipschitz functions defined on $(0,\infty)$
which are negative for $x$ large enough. 
Let $\phi_1$ and $\phi_2$ the flows associated respectively to $\psi_1$ and $\psi_2$.
We state  simple conditions to guarantee that two such flows are close or  equivalent near $+\infty$, when $\phi_1$ comes down from infinity.
\begin{lem}
\label{equivflott}
We assume that $\psi_1$ is $(L,\alpha)$ non-expansive and $\int_{\infty}^. \frac{1}{\psi_1(x)} dx <\infty$ and 
 $$\psi_2(x)= \psi_1(x) +  h(x),$$
where  $h$ is  a bounded function.
Then $\phi_2$ comes down from infinity and
$$\lim_{t\downarrow 0+} \phi_2(\infty,t)-\phi_1(\infty,t)=0.$$
\end{lem}
\begin{proof}
This result can be proved using Lemma \ref{key}  or actually mimicking its proof which can be greatly simplified since here both processes are deterministic. 
By analogy, we set
$$ x_t=\phi_1(x_0,t), \qquad X_t=\phi_2(x_0,t)=x_0+\int_0^t \psi_1(\phi_2(x_0,s))ds+R_t,$$
where $R_t=\int_0^t h(X_s)ds=\int_0^t h(\phi_2(x_0,s))ds$. Then
$$\vert \widetilde{R}_t\vert =2. \1{ S_{t-} \leq \varepsilon}\left\vert \int_0^t   (X_{s}-x_s)dR_s\right\vert \leq 2 \varepsilon t\parallel h\parallel_{\infty}$$
and Lemma \ref{key}  ensures that  for any $\varepsilon>0$, for   $T$ small enough, 
$$\sup_{x_0\geq 1} S_T=\sup_{t\leq T, x_0\geq 1} \vert \phi_2(x_0,t)-\phi_1(x_0,t) \vert \leq \varepsilon.$$
Letting $x_0\rightarrow \infty$ yields the result, recalling that $\int_{\infty}^. \frac{1}{\psi_1(x)} dx <\infty$ ensures that $\phi_1(\infty,t)<\infty$ for any $t>0$.
\end{proof}

\begin{lem}
\label{equivflot}
If  $\psi_1(x)<0$ for $x$ large enough and $ \int_{\infty}^. 1/\psi_1(x) dx <\infty$ and $\psi_1(x)\sim_{x\rightarrow \infty} \psi_2(x)$, 
then
$\int_{\infty}^.  1/\psi_2(x) dx <\infty$ and $\phi_2$ comes down from infinity.\\
If additionally $\phi_1(\infty,t) \sim ct^{-\alpha}$  as $t\downarrow 0+$ for some $\alpha>0$ and $c>0$,  then
 $$\phi_2(\infty,t)\sim_{t\rightarrow 0}ct^{-\alpha}.$$
\end{lem}
\begin{proof}
Let $\varepsilon\in (0,1)$ and choose $x_1>0$ such that
$$ (1+\varepsilon) \psi_2(x) \leq \psi_1(x)<0,$$
%\marginpar{Amleiorer en donnant les temps d'entree de domaines pour le small enough}
for $x\geq x_1$. Then for any $x_0>x_1,$
$$\phi_1(x_0,t) \geq (1+\varepsilon)\int_0^t \psi_2(\phi_1(x_0,s))ds$$
for $t$ small enough. Then, $\phi_1(x_0,t)\geq \phi_2(\infty,(1+\varepsilon)t)$ and
$$\phi_1(\infty,t/(1+\varepsilon))\geq \phi_2(\infty,t)$$
for $t$ small enough. Proving the symmetric inequality ends up the proof.
\end{proof}
$\newline$

In the case of polynomial drift, we specify here the error term when coming from infinity. 
\begin{lem} \label{equivv}
Let $\varrho>1, c>0,\alpha>0, \varepsilon >0$ and
$$\psi(x)=- cx^{\varrho}(1+r(x)x^{-\alpha}),$$
where  $r$ is locally Lipschitz and bounded on $(x_0, \infty)$ for some $x_0>0$. \\
Denoting by $\phi$ the flow associated to $\psi$,  we have
$$\phi(\infty,t)=(ct/(\varrho-1))^{1/(1-\varrho)}(1+\widetilde{r}(t)t^{\alpha/(\varrho-1)}),$$
where $\widetilde{r}$ is 	a bounded function.
\end{lem}
\begin{proof} As $r$ is 	bounded, there exists $c_1,c_2$   such that
$$ - cx^{\varrho}(1+c_1x^{-\alpha})\leq \psi(x)\leq - cx^{\varrho}(1+c_2x^{-\alpha})$$
for $x$ large enough.
Then, there exists,  $c_1',c_2'$ such that
$$-c x^{-\varrho}(1-c_2'x^{-\alpha})\leq  \frac{1}{\psi(x)}\leq -c x^{-\varrho}(1-c_1'x^{-\alpha}).$$
for $x$ large enough and
$$-c  \int_{\phi(x_0,0)}^{\phi(x_0,t)}x^{-\varrho}(1-c_2'x^{-\alpha})dx\leq \int_{\phi(x_0,0)}^{\phi(x_0,t)}  \frac{dx}{\psi(x)}\leq -c\int_{\phi(x_0,0)}^{\phi(x_0,t)} x^{-\varrho}(1-c_1'x^{-\alpha})dx,$$
where the middle term is equal to $t$.  Letting $x_0\rightarrow \infty$
$$   c_2''\phi(\infty,t)^{-\varrho-\alpha+1} \leq  t -\frac{c}{\varrho-1} \phi(\infty,t)^{-\varrho+1}\leq  c_1''\phi(\infty,t)^{-\varrho-\alpha+1}$$
for some $c_1'',c_2''$.
We know from the previous lemma that
$\phi(\infty,t) \sim (c\varrho^{-1}t)^{1/(1-\varrho)}$ as $t\rightarrow 0$ and we get here 
$$\phi(\infty,t)= (ct/(\varrho-1))^{1/(1-\varrho)}(1+\mathcal O( t^{-1+(-\varrho+1-\alpha)/(1-\varrho)}))=(ct/(\varrho-1))^{1/(1-\varrho)}(1+\mathcal O( t^{\alpha/(\varrho-1)})),$$
which ends up the proof.
\end{proof}

%\marginpar{empecher qu'il rentre et sorte de $A$ tout le temps : SOIT UNE FOIS avant A, puis APRES, soit NB fini. POUR L instant recoupement}
We need also the following estimates. We assume 
 that $\psi$ and $\psi^K$ are locally Lipschitz vectors  fields on the closure $\overline D$
of  an open domain $D \subset \R^{\d}$ and their respective flows on $D$ are denoted by  $\phi$ and $\phi^K$.
We assume that there are well defined and belongs to $D$ respectively  until a maximal time $T_D$ and $T^K_D$.
We write again  $T_{D,\varepsilon}(x_0)=\sup\{ t \geq 0 : \forall s<T(x_0), \overline B(\phi(x_0,s),\varepsilon) \subset D\}$.
%the maximal time when the flow $\phi$ started at $x_0$  belongs to $D$ and remains at distance $\varepsilon$ from its boundary.
\begin{lem} \label{ctrflo}
We  assume that 
 there  exist
 $A \geq 1$, $ c,\mu>0$ and $\varepsilon\in (0,1]$ such that 
 \be
 \label{A1}
 ( \psi(x)-\psi(y)){\boldsymbol .} (x-y) \leq -\mu \parallel x\parallel_2 \parallel x-y \parallel_2^2
 \ee
for any $x \in D \cap B(0,A)^c $ and $y\in \overline{B}(x,\varepsilon)$
 and
\be
\label{A2}
\parallel \psi(x) -\psi^K(x) \parallel_2 \leq c\frac{1+ \parallel x\parallel_2}{K}
\ee
for any $x\in D$ and $K\geq 1$.  
Then,  writing $M=3c/\mu$,  there exists $L\geq 0$ such that
 for all $T\geq 0$, $\eta>0$,  $K \geq 2 \max(M,\eta)\exp((L+1/M)T)/\varepsilon$,    $x_0 \in D$ and $x_1 \in \overline{B}(x_0,\eta/K)$, we have
 $T_D^K(x_1)\geq T_{D,\varepsilon}(x_0)$ and 
$$\sup_{t < T_{D,\varepsilon}(x_0)\wedge T}\parallel \phi(x_0,t) -\phi^K(x_1, t) \parallel_2 \leq \frac{\max(M,\eta)\exp((L+M)T)}{K}.$$
%\frac{3ce^{LT}}{K\mu}.$$ 
\end{lem}
\begin{proof}  Let $T>0$ and  $K \geq 2\max(M, \eta) \exp((L+1/M)T)/\varepsilon$, so that
$$ \max(M, \eta)/K\leq \max(M, \eta) e^{(L+1/M)T}/K   \leq \varepsilon/2.$$
Write
$$x_t=\phi(x_0,t), \qquad x_t^K=\phi^K(x_1, t), \qquad T^K=T_D(x_0)\wedge T_D^K(x_1)$$
for convenience and 
consider  the time
$$t_1^K=\inf\{ t\in [0,T^K)  \ : \   \parallel x_t-x_t^K \parallel_2^2\geq M/K\}\in(0,\infty].$$
Let us   assume that $t_1^K<T_{D,\varepsilon}(x_0)\wedge T \wedge T^K$
and set 
$$t_2=\inf\{ t \in (t_1^K,T^K)  \ : \ \parallel x_t-x_t^K \parallel_2^2\geq \varepsilon \ \text{or} \
 \parallel x_t-x_t^K \parallel_2^2< M/K  \}.
 $$
We  show now that for any $t\in [t_1^K,t_2^K\wedge T^K)$, we have
\be
\label{inegGron}
\frac{d}{dt} \parallel x_t-x_t^K \parallel_2^2 = 2(\psi(x_t)-\psi^K(x_t^K)){\boldsymbol .}(x_t-x_t^K) \leq 2(L+1/M) \parallel x_t-x_t^K \parallel_2^2.
\ee
to get from Gronwall inequality  and  $ \parallel x_{t_1^K}-x_{t_1^K}^K \parallel_2\leq \max(M, \eta)/K$ that
$$ \parallel x_t-x_t^K \parallel_2 \leq % M \exp((L+1/M)(t-t_1^K)/K\leq
 \max(M, \eta) \exp((L+1/M)T)/K.%\leq \varepsilon/2 
$$
This will be  enough to
 prove the lemma since the right hand side is smaller than $\varepsilon/2$.
 
First, using  that  on the closure of   $D\cap B(0,A+1)$,  $\psi$ is Lipschitz 
%on compact domains of $D$
% and  compact on the compact domains of
%$(0,\infty)^2$, 
and that $K\parallel  \psi^K(.) -\psi(.) \parallel_2$ is bounded on   $D\cap B(0,A+1)$  by $(\ref{A2})$,
%and   $\psi_{F_i}$ is Lipschitz,   there exists $L$ such that
%$$\parallel \psi_{F_i}(x)-\psi^K_{F_i}(y) \parallel_2 \leq L \parallel x-y \parallel_2$$
%$$\parallel \psi_{F_i}(x)-\psi_{F_i}^K(y) \parallel_2 \leq L (\parallel x-y \parallel_2+1/K)$$
%for any $x,y$ in a compact domain. \\ 
there exists $L>0$ such that
$$
\parallel \psi(x)-\psi^K(y) \parallel_2 \leq L (\parallel x-y \parallel_2+1/K),
$$
for any $x,y \in D \cap B(0, A+1)$.  
Then, using   Cauchy-Schwarz inequality, for any $t\in [t_1^K,t_2^K\wedge T^K)$ such that $x_t \in B(0,A)$, 
\Bea
\frac{d}{dt} \parallel x_t-x_t^K \parallel_2^2  &\leq & 2   \parallel x_t-x_t^K \parallel_2  \parallel \psi(x_t)-\psi^K(x_t^K) \parallel_2 \\
&\leq &   2L  \parallel x_t-x_t^K \parallel_2^2 +\frac{2}{K}  \parallel x_t-x_t^K \parallel_2 \\
&\leq &   2(L+1/ M) \parallel x_t-x_t^K \parallel_2^2 .
\Eea
since $ \parallel x_t-x_t^K \parallel_2\geq M/K$ for $t\leq t_2^K$. This proves $(\ref{inegGron})$  when $x_t \in B(0,A)$. 

To conclude, we consider $t\in [t_1^K,t_2^K\wedge T^K]$ such that $x_t \in B(0,A)^c$.  Then (\ref{A1}) and (\ref{A2})
% assumptions above
 and Cauchy-Schwarz inequality give
\Bea
\frac{d}{dt} \parallel x_t-x_t^K \parallel_2^2 &=&   2(\psi(x_t)-\psi(x_t^K)).(x_t-x_t^K) +2(\psi(x_t^K)-\psi^K(x_t^K)).(x_t-x_t^K)  \\
&\leq & 2\left(-\mu \parallel x_t\parallel_2 \parallel x_t-x_t^K \parallel_2  +c\frac{1+ \parallel x_t^K\parallel_2}{K}\right)  \parallel x_t-x_t^K \parallel_2.
\Eea
Moreover $\parallel x_t \parallel_2 \geq A\geq 1$ and $x_t^K\in \overline{B}(x_t,\varepsilon)$, so 
$$ 1+ \parallel x_t^K\parallel_2 \leq 1+\parallel x_t\parallel_2+\parallel x_t^K-x_t\parallel_2 \leq 3 \parallel x_t\parallel_2,$$
and adding that $\parallel x_t-x_t^K \parallel_2 \geq M/K=3c/(K\mu)$ since $t\leq t_2^K$, we get
$$\frac{d}{dt} \parallel x_t-x_t^K \parallel_2^2 \leq 0.$$
This ends up the proof of $(\ref{inegGron})$ and thus of the lemma.
%Letting $K \varepsilon\mu  \geq  3ce^{LT}$, it ensures that if $x_0 \in  D \cap B(0,A)^c$ and $\parallel x_0-x_0^K \parallel_2 \leq 3ce^{Lt_0}/(K\mu)$ for some $t_0 \leq T$, then
%$$\parallel x_t-x_t^K \parallel_2 \leq  \frac{3c}{K\mu}e^{Lt_0} \leq 1 $$
%on the  time interval when  $x_t \in   D \cap B(0,A)^c$. \\
\end{proof}
$\newline$
{\bf Acknowledgment.}  The author   thanks warmly  Pierre Collet, Sylvie M\'el\'eard and St\'ephane Gaubert for stimulating and fruitful discussions in Ecole polytechnique. The author is also grateful
 to thank Cl\'ement Foucart and Sylvie M\'el\'eard for their reading
of a previous version of this paper and their various comments.  This work  was partially funded by the Chaire Mod\'elisation Math\'ematique et Biodiversit\'e VEOLIA-\'Ecole Polytechnique-MNHN-F.X.


\begin{thebibliography}{10}

\bibitem{Anderson1991}
W.~J. Anderson.
\newblock {\em Continuous-time {M}arkov chains}.
\newblock Springer Series in Statistics: Probability and its Applications.
  Springer-Verlag, New York, 1991.
\newblock An applications-oriented approach.

\bibitem{Aminzare} Z.  Aminzare and E. D. Sontag.  
Contraction methods for nonlinear systems:
a brief introduction and some open problems. \emph{53rd IEEE Conference on Decision and Control}
December 15-17, 2014. Los Angeles, California, USA.

\bibitem{TCP} R. Aza\"is, J-B Bardet, A. G\'enadot, N. Krell and P-A Zitt.
ESAIM: Proceedings, January 2014, Vol. 44, p. 276-290 SMAI Groupe MAS - Journ\'ees 
MAS 2012 - Session th\'ematique.

\bibitem{BM} V. Bansaye and  S. M\'el\'eard. \emph{Stochastic Models for Structured Populations}. Mathematical Biosciences Institute Lecture Series, 2015.  

\bibitem{BMR} V. Bansaye, S. M\'el\'eard and M.  Richard. The speed of coming down from infinity for birth and death processes. To appear in \emph{Adv. Appl. Probab.}, 2016.

\bibitem{BBL}
J.~Berestycki, N.~Berestycki, and V.~Limic.
\newblock The {L}ambda-coalescent speed of coming down from infinity.
\newblock {\em Ann. Probab.}, 38(1):207--233, 2010.

\bibitem{BL}
J. Bertoin and J-F LeGall. Stochastic flows associated to coalescent processes. III. Limit theorems. \emph{Illinois J. Math.} (2006), Vol. 50, p. 147-181.

\bibitem{Billingsley} 
P. Billingsley.
\newblock {\em Convergence of probability measures}.
\newblock Wiley Series in Probability and Statistics: Probability and
  Statistics. John Wiley \& Sons Inc., New York, second edition, 1999.

\bibitem{CattiauxMeleard} P. Cattiaux, S. M\'el\'eard.
Competitive or weak cooperative stochastic Lotka-Volterra systems conditioned to non extinction.
\emph{J. Math. Biology.}  60(6):797-829, 2010.

\bibitem{Cattiaux2009}
P.~Cattiaux, P.~Collet, A.~Lambert, S.~Mart{\'{\i}}nez, S.~M{\'e}l{\'e}ard, and
  J.~San~Mart{\'{\i}}n.
\newblock Quasi-stationary distributions and diffusion models in population
  dynamics.
\newblock {\em Ann. Probab.}, 37(5):1926--1969, 2009.

\bibitem{CV} N. Champagnat and D. Villemonais. Exponential convergence to quasi-stationary
distribution and Q-process. To appear in \emph{Probab. Th. Rel. Field},  2015.

\bibitem{DL}  D. A. Dawson and Z. Li. Stochastic equations, flows and measure-valued processes. \emph{Ann. Probab.} 40 (2012), no. 2, 813-857.

\bibitem{DN} R.   Darling and   J.  Norris. 
Differential equation approximations for Markov chains. 
\emph{Probab. Surv.} 5 (2008), 37-79. 

\bibitem{Donnelly91}
P.~Donnelly.
\newblock Weak convergence to a {M}arkov chain with an entrance boundary:
  ancestral processes in population genetics.
\newblock {\em Ann. Probab.}, 19(3):1102--1117, 1991.

\bibitem{DM} C.
Dellacherie and P. A. Meyer. \emph{Probabilities and potential. B. Theory of martingales}.
% Translated from the French by J. P. Wilson. 
North-Holland Mathematics Studies, 72. North-Holland Publishing Co., Amsterdam, 1982. 

\bibitem{D} F. Dumortier, J. Llibre and J. C. Art\'es.  \emph{Qualitative Theory of
Planar Differential Systems}. Springer-Verlag, Berlin, 2006.

\bibitem{EK} S. N. Ethier and T. G. Kurtz. \emph{Markov processes. Characterization and convergence.} Wiley Series in Probability and Mathematical Statistics: Probability and Mathematical Statistics, New York, 1986. .

\bibitem{FW} M. I. Frendlin and A. D. Wentzell. \emph{Random perturbations of dynamical systems} . Third edition. Grundlehren der Mathematischen Wissenschaften, 260. Springer, Heidelberg, 2012

%\bibitem{H} M. Hutzenthaler and A. Jentzen (2015): Numerical approximations of stochastic differential equations with non-globally Lipschitz continuous coefficients, Memoirs of the American Mathematical Society 236, no. 1112

\bibitem{IW} 
N. Ikeda and S. Watanabe.
\newblock {\em Stochastic differential equations and diffusion processes, 2nd
  ed.}
\newblock North-Holland, 1989.


\bibitem{JS}
J. Jacod and A.N. Shiryaev.
 \newblock {\em Limit Theorems for Stochastic Processes}.
Springer,  2nd edition, 1987.


\bibitem{KarlinMcGregor57}
S.~Karlin and J.~L. McGregor.
\newblock The differential equations of birth-and-death processes, and the
  {S}tieltjes moment problem.
\newblock {\em Trans. Amer. Math. Soc.}, 85:489--546, 1957.

\bibitem{Karlin1975}
S.~Karlin and H.~M. Taylor.
\newblock {\em A first course in stochastic processes}.
\newblock Academic Press [A subsidiary of Harcourt Brace Jovanovich,
  Publishers], New York-London, second edition, 1975.

\bibitem{Kurtz2}   Thomas G. Kurtz. \emph{Approximation of population processes}. Vol. 36 of CBMS-NSF Regional Conference Series in Applied Mathematics. Society for Industrial and Applied Mathematics (SIAM), Philadelphia, Pa., 1981.

%\bibitem{K2} G. Kersting. A law of large numbers for stochastic difference equations. \emph{Stochastic Process. Appl.} 40 (1992), no. 1, 1-13.

%marginpar{To keep ?}
%\bibitem{KKP}  Kang, Hye-Won; Kurtz, Thomas G.; Popovic, Lea Central limit theorems and diffusion approximations for multiscale Markov chain models. \emph{Ann. Appl. Probab.} 24 (2014), no. 2, 721-759.

\bibitem{LP} V. Le and E. Pardoux
Height and the total mass of the forest of genealogical trees of a large population with general competition. To appear in
\emph{ESAIM P \& S}.

\bibitem{limxsi} V. Limic. On the speed of coming down from infinity for $\Xi$ -coalescent processes. \emph{Electron. Journal Probab.}
Vol. 15, 217--240, 2010.
%\marginpar{Citer ou retirer ?}
\bibitem{LT1} V. Limic and A. Talarczyk.
\newblock Diffusion limits for mixed with Kingman coalescents at small times. 
\newblock {\em Electron. J. Probab.}, vol. 20, Paper 45,  2015.

\bibitem{LT2} V. Limic and A. Talarczyk.
\newblock  Second-order asymptotics for the block counting process in a class of regularly varying Lambda-coalescents.   \emph{Ann. Probab.}, Vol. 43, 2015.
\bibitem{Pitman} J. Pitman. %$\Lambda$-coalescent. \
Coalescents with multiple collisions. \emph{Ann. Probab. } 27 (1999), no. 4, 1870-1902.
\bibitem{Schwein} J. Schweinsberg. A necessary and sufficient condition for the $\Lambda$-coalescent to come down from infinity. \emph{Electron. Comm. Probab}. (2000) Vol. 5, P. 1-11.
\bibitem{SH}  A. V. Skorokhod, F. C. Hoppensteadt and H. Salehi. \emph{Random perturbation methods with applications in science and engineering} . Applied Mathematical Sciences, 150. Springer-Verlag, New York, 2002.


\bibitem{Doorn1991}
E.~A. van Doorn.
\newblock Quasi-stationary distributions and convergence to quasi-stationarity
  of birth-death processes.
\newblock {\em Adv. in Appl. Probab.}, 23 (4) : 683--700, 1991.
\end{thebibliography}
\end{document}